\documentclass[12pt,twoside]{article}
\pdfoutput=1

\usepackage{amsmath}
\usepackage{amssymb}
\usepackage{epsfig}
\usepackage{xcolor}

\usepackage{tabularx}
\usepackage{booktabs}
\usepackage{pdflscape}
\usepackage{rotating}

\usepackage{a4wide}

\newcommand{\plminul}{0,\pm}
\newcommand{\plmin}{\pm}
\newcommand{\mitdE}{\; \biggm| \;}
\newcommand{\opin}[2]{\mathopen] #1, #2 \mathclose[}
\newcommand{\RO}[3]{\operatorname{Osc}(#1{:}#2{:}#3)}
\newcommand{\quotient}[1]{{\! / \!}_{\textstyle #1}}

\newcommand{\X}[1]{X_{\textstyle \! \mbox{\small $#1$}}}

\newcommand{\Frac}[2]{\mbox{\footnotesize
                            ${\displaystyle \frac{#1}{#2} }$}}
\newcommand{\R}{\mathbb{R}}

\newcommand{\Z}{\mathbb{Z}}

\newcommand{\C}{\mathbb{C}}
\newcommand{\T}{\mathbb{T}}
\newcommand{\dd}{\mathrm{d}}
\newcommand{\ee}{\mathrm{e}}
\newcommand{\ii}{\mathrm{i}}
\newcommand{\setS}{\mathbb{S}}
\newcommand{\setR}{\mathbb{R}}
\newcommand{\setZ}{\mathbb{Z}}

\newcommand{\setT}{\mathbb{T}}
\newcommand{\SL}{\mathrm{SL}}
\newcommand{\BV}{\mathcal{B}}
\newcommand{\CV}{\mathcal{C}}
\newcommand{\DV}{\mathcal{D}}
\newcommand{\FV}{\mathcal{F}}
\newcommand{\TV}{\mathcal{T}}
\newcommand{\Reg}{\mathcal{R}}
\newcommand{\cP}{{\mathcal{P}}}
\newcommand{\cH}{{\mathcal{H}}}
\newcommand{\EM}{{\cal E \!\:\!\! M}}
\newcommand{\re}{\ensuremath{\operatorname{Re}}}
\newcommand{\im}{\ensuremath{\operatorname{Im}}}

\newcommand{\HHsup}[1]{\ensuremath{\mathbf{HH}^{+}_{#1}}}
\newcommand{\HHsub}[1]{\ensuremath{\mathbf{HH}^{-}_{#1}}}
\newcommand{\HHdeg}[1]{\ensuremath{\mathbf{HH}^{0}_{#1}}}
\newcommand{\Cusp}[1]{\ensuremath{\mathbf{CB}_{#1}}}
\newcommand{\CS}[1]{\ensuremath{\mathbf{CS}_{#1}}}

\newcommand{\fref}[1]{figure~\ref{#1}}
\newcommand{\sref}[1]{section~\ref{#1}}
\newcommand{\Sref}[1]{Section~\ref{#1}}

\newtheorem{theorem}{Theorem}
\newtheorem{lemma}[theorem]{Lemma}
\newtheorem{proposition}[theorem]{Proposition}

\newtheorem{remark}[theorem]{Remark}

\newcommand{\qed}{\nolinebreak\hfill {$\Box$} \par\medbreak}

\newenvironment{proof}{\removelastskip\medskip\noindent\textbf{Proof}\quad}
                      {\qed}

\begin{document}

\title{\protect\Large Bifurcations and Monodromy \\
                      of the Axially Symmetric
                      $1{:}1{:}{-}2$~Resonance}

\author{{\protect\normalsize Konstantinos Efstathiou} \protect\\[-1mm]
   {\protect\footnotesize\protect\it Bernoulli Institute,
                             University of Groningen} \protect\\[-2mm]
   {\protect\footnotesize\protect\it PO Box 407,
                    9700~AK~Groningen, The Netherlands}     \protect\\
   {\protect\normalsize Heinz Han{\ss}mann}           \protect\\[-1mm]
   {\protect\footnotesize\protect\it Mathematisch Instituut,
                                Universiteit Utrecht} \protect\\[-2mm]
   {\protect\footnotesize\protect\it Postbus 80010,
                     3508~TA Utrecht, The Netherlands}      \protect\\
   {\protect\normalsize Antonella Marchesiello}       \protect\\[-1mm]
   {\protect\footnotesize\protect\it Department of Applied Mathematics,
                   Faculty of Information Technology} \protect\\[-2mm]
   {\protect\footnotesize\protect\it Czech Technical University
          in Prague, Th\'{a}kurova 9, 160 00 Prague 6, Czech Republic}}

\date{\protect\normalsize 21 July 2019} 

\maketitle

\begin{abstract}
\noindent
We consider integrable Hamiltonian systems in three degrees of
freedom near an elliptic equilibrium in $1{:}1{:}{-}2$~resonance.
The integrability originates from averaging along the periodic
motion of the quadratic part and an imposed rotational symmetry
about the vertical axis.
Introducing a detuning parameter we find a rich bifurcation diagram,
containing three parabolas of Hamiltonian Hopf bifurcations that
join at the origin.
We describe the monodromy of the resulting ramified $3$--torus
bundle as variation of the detuning parameter lets the system
pass through $1{:}1{:}{-}2$~resonance.
\end{abstract}

\section{Introduction}
\label{Introduction}

\noindent
Let $H^{\gamma}$ be a family of Hamiltonian systems in three degrees of
freedom depending on parameters $\gamma \in \mathbb{R}^k$ and defined
on~$\R^6$ with canonical co-ordinates $x_i, y_i$, $i = 1, 2, 3$.
We are interested in the dynamics near the elliptic equilibria, which
are isolated for fixed~$\gamma$.
Moving the equilibrium to the origin we expand
\begin{equation}
\label{originalhamiltonian}
   H^{\gamma}(x, y) \;\; = \;\; \alpha_1(\gamma) I_1 \; + \;
   \alpha_2(\gamma) I_2 \; + \; \alpha_3(\gamma) I_3 \; + \; h.o.t.
\end{equation}
where $I_i = \frac{1}{2}(y_i^2 + x_i^2)$, $i=1,2,3$.
The dynamical behaviour near the origin now depends on number-theoretic
properties of the frequencies $\alpha_i = \alpha_i(\gamma)$.
In the non-resonant case, where there are no integer relations
\begin{equation}
\label{resonance}
   k_1 \alpha_1 \; + \; k_2 \alpha_2 \; + \; k_3 \alpha_3 \;\; = \;\; 0
\end{equation}
among the frequencies, the normal form truncated at order~$4$ reads as
\begin{equation}
\label{normalizedhamiltonian}
   H(I) \;\; = \;\; \sum_{i=1}^3 \alpha_i I_i
   \; + \; \sum_{i,j=1}^3 \alpha_{ij} I_i I_j \enspace ,
\end{equation}
see~\cite{AKN85, SVM07} and references therein,
and generically satisfies Kolmogorov's non-degeneracy condition
\begin{displaymath}
   \det (\alpha_{ij})_{ij} \;\; \neq \;\; 0  \enspace.
\end{displaymath}
The integrable Hamiltonian function~\eqref{normalizedhamiltonian}
defines a ramified torus bundle with regular fibres~$\T^3$, singular
fibres~$\T^2$ parametrised by the planes $I_i = 0$, $i=1,2,3$,
periodic orbits (the normal modes, also singular fibres of the
ramified torus bundle) parametrised by the $I_i$--axes and the
equilibrium (giving the most singular fibre of the ramified torus
bundle) at the origin $I_1 = I_2 = I_3 = 0$; all this is also valid
in the indefinite case (where the frequencies~$\alpha_i$ do not all
have the same sign).
Kolmogorov's non-degeneracy condition allows to apply
{\sc kam}~theory~\cite{AKN85} whence the original
Hamiltonian~\eqref{originalhamiltonian} defines a Cantorised ramified
torus bundle, with fibres $\T^n$, $n=2,3$ parametrised by Cantor sets
--- obtained from their smooth counterparts above by strong
non-resonance conditions (e.g.\ Diophantine conditions) on the internal
frequencies of the tori.

This description of the local dynamics of~\eqref{originalhamiltonian}
remains correct in case of resonances if these are of order
$|k| := |k_1| + |k_2| + |k_3| \geq 5$ since then the normal form
truncated at order~$4$ is still given by~\eqref{normalizedhamiltonian}.
Thus, for an open and dense subset of parameters~$\gamma$ a
substantial part of the dynamics near the origin is rather transparent.
The elliptic equilibrium has three normal modes and the majority of
bounded trajectories is quasi-periodic, with three Cantor sets of
Hausdorff-dimension~$2$ organizing the distribution of invariant tori.
Note that in the positive definite case $\alpha_i > 0$, $i=1,2,3$ (as
well as in the negative definite case) the equilibrium is stable in the
sense of Lyapunov, while indefinite elliptic equilibria are expected to
be unstable due to Arnol'd diffusion.

In case of a single resonance~\eqref{resonance} of order $|k| \leq 4$
the normal form truncated at order~$4$ is still integrable but contains
extra `resonant terms' of order~$|k|$.
The resulting ramified torus bundle and its Cantorised counterpart thus
depend on the resonance at hand.
For instance, an indefinite elliptic equilibrium with resonance
$2 \alpha_1 + \alpha_2 = 0$ may have (in three degrees of freedom) only
two normal modes, see~\cite{AO84, mlb87}.
Single resonances~\eqref{resonance} among the normal frequencies
$\alpha_i = \alpha_i(\gamma)$ define hypersurfaces in the parameter
space and detuning the frequencies shows how to pass from one open
region to a neighbouring one.

In case of two independent resonances~\eqref{resonance} the frequencies
are integer multiples $\alpha_i = n_i \alpha$, $i=1,2,3$ (with
$\gcd(n_1, n_2, n_3) = 1$) of a basic frequency $\alpha \in \R$ and one
speaks of the $n_1{:}n_2{:}n_3$~resonance
\begin{displaymath}
   \RO{n_1}{n_2}{n_3}: \qquad K \;\; = \;\;
   n_1 I_1 \; + \; n_2 I_2 \; + \; n_3 I_3
\end{displaymath}
(scaling time allows to achieve $\alpha = 1$).
In this paper we study the indefinite $1{:}1{:}{-}2$~resonance
where~\eqref{originalhamiltonian} reads as
\begin{equation}
\label{originalresonanthamiltonian}
   H^{\gamma}(x, y) \;\; = \;\; (\alpha + \delta_1(\gamma)) I_1 \; + \;
   (\alpha + \delta_2(\gamma)) I_2 \; - \;
   (2 \alpha + \delta_3(\gamma)) I_3 \; + \; h.o.t.
\end{equation}
with detuning $\delta = \delta(\gamma)$; for the moment we refrain from
scaling time to achieve $\alpha = 1$.
Smooth changes of parameters
$\gamma \mapsto \delta(\gamma)$ allow to skip the
$\gamma$--dependence in~\eqref{originalresonanthamiltonian} altogether
and study~$H^{\delta}$ instead.
The Hamiltonians with a $1{:}1{:}{-}2$~resonant equilibrium at the
origin are thus given by $2 \delta_1 = 2 \delta_2 = \delta_3$.

\begin{remark}
  We expect that the three `resonant terms' of order~$3$ make the
  normal form truncated at order~$3$ non-integrable, similar to the
  (definite) $1{:}1{:}2$~resonance for which non-integrability has
  been proven in the absence of extra symmetries~\cite{jjd84}; see
  also~\cite{oc12} where the same result could be achieved for the
  $1{:}2{:}3$ and $1{:}2{:}4$~resonances.
\end{remark}

\noindent
To enforce integrability we impose an axial $\setS^1$--symmetry of
rotations about the $x_3$--axis.
From Noether's theorem it follows that the third component
\begin{displaymath}
   N \;\; = \;\; x_1 y_2 \; - \; x_2 y_1
\end{displaymath}
of the angular momentum is an integral of motion.
For an axially symmetric detuning we have
$\delta_1 = \delta_2 =: \delta$ and subsume $\delta_3$ into~$2\alpha$.
Adding the axially symmetric detuning $\beta N$ of the
$1{:}1$~subresonance
the Hamiltonian~\eqref{originalresonanthamiltonian} becomes
\begin{equation}
\label{originaldetunedhamiltonian}
   H^{\delta} \;\; = \;\;
   \alpha L \; + \; \beta N \; + \; \delta R \; + \; h.o.t.
\end{equation}
with
\begin{displaymath}
\end{displaymath}
\begin{align*}
   \RO{1}{1}{{-}2}: \qquad & L \;\; = \;\;
   I_1 \; + \; I_2 \; - \; 2 I_3
\intertext{and}
   \RO{1}{1}{0} : \qquad & R \;\; = \;\; I_1 \; + \; I_2
   \enspace.
\end{align*}
Let $H$ denote the normal form of~$H^{\delta}$ with respect to~$L$
truncated at order~$4$.
The conserved quantity $N$ is inherited by~$H$ and the normalizing
procedure makes~$L$ an integral of motion as well.
Since $\{ N, L \} = 0$ the Hamiltonian~$H$ admits a $\T^2$--symmetry
and the energy-momentum mapping
\begin{equation}
\label{energymomentummapping}
   \EM  \, := \, (N, L, H) \; : \;\; \R^6 \;\; \longrightarrow \;\; \R^3
\end{equation}
turns~$\R^6$ into a ramified torus bundle.

\begin{remark}
  In the literature on integrable Hamiltonian systems, the diagram
  showing the set of regular and critical values of the
  energy-momentum mapping~$\EM$ and the type of the corresponding
  fibres of~$\EM$ is sometimes called the `bifurcation diagram
  of~$\EM$'.
  In this work, to avoid confusion we use the term \emph{bifurcation
  diagram} only to refer to the set of (internal and external)
  parameter values for which the system undergoes a bifurcation and
  we use the term \emph{set of critical values of~$\EM$} to refer
  to that `bifurcation diagram of~$\EM$'.
\end{remark}

\noindent
After reduction of the $\T^2$--symmetry the values $\mu$ of~$N$ and
$\ell$ of~$L$ serve as (internal) parameters.
The set of critical values of the energy-momentum
mapping~\eqref{energymomentummapping} provides for a concise
description of the dynamics defined by~$H$.
This set of critical values is in turn determined by the bifurcation
diagram (in the $(\mu, \ell)$--plane) of the reduced system.
In \fref{fig:introductorybifurcationdiagrams} we give two of the
bifurcation diagrams we derive in \sref{Dynamics} of the present
paper.
During the process it is instructive to include the
detuning~$\delta$ as (external) parameter, even if one were only
interested in the case $\delta = 0$ of the $1{:}1{:}{-}2$~resonance
itself.

\begin{figure}[t!]
\begin{center}
\includegraphics[width=6cm]{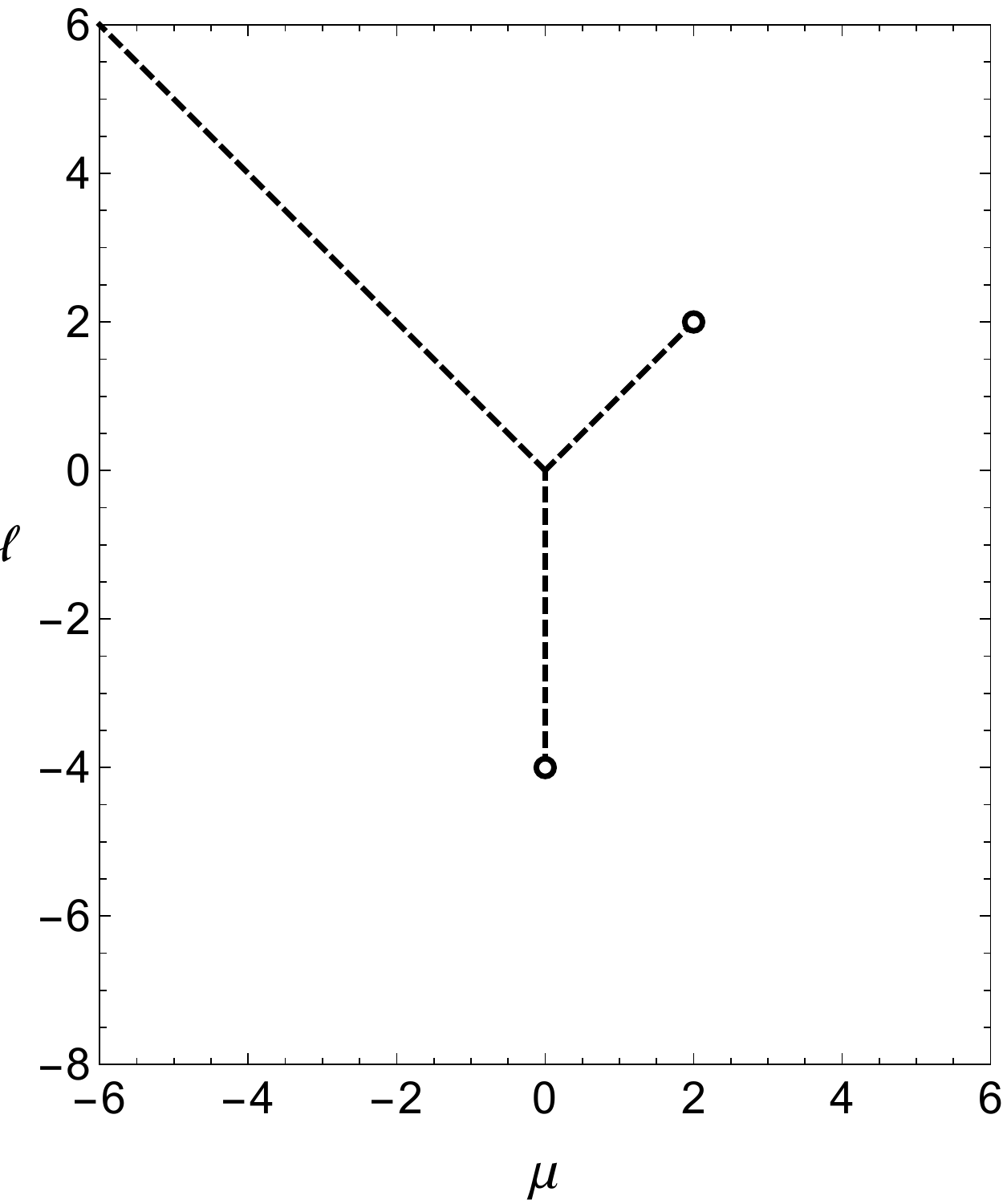}
\hspace{1cm}
\includegraphics[width=6cm]{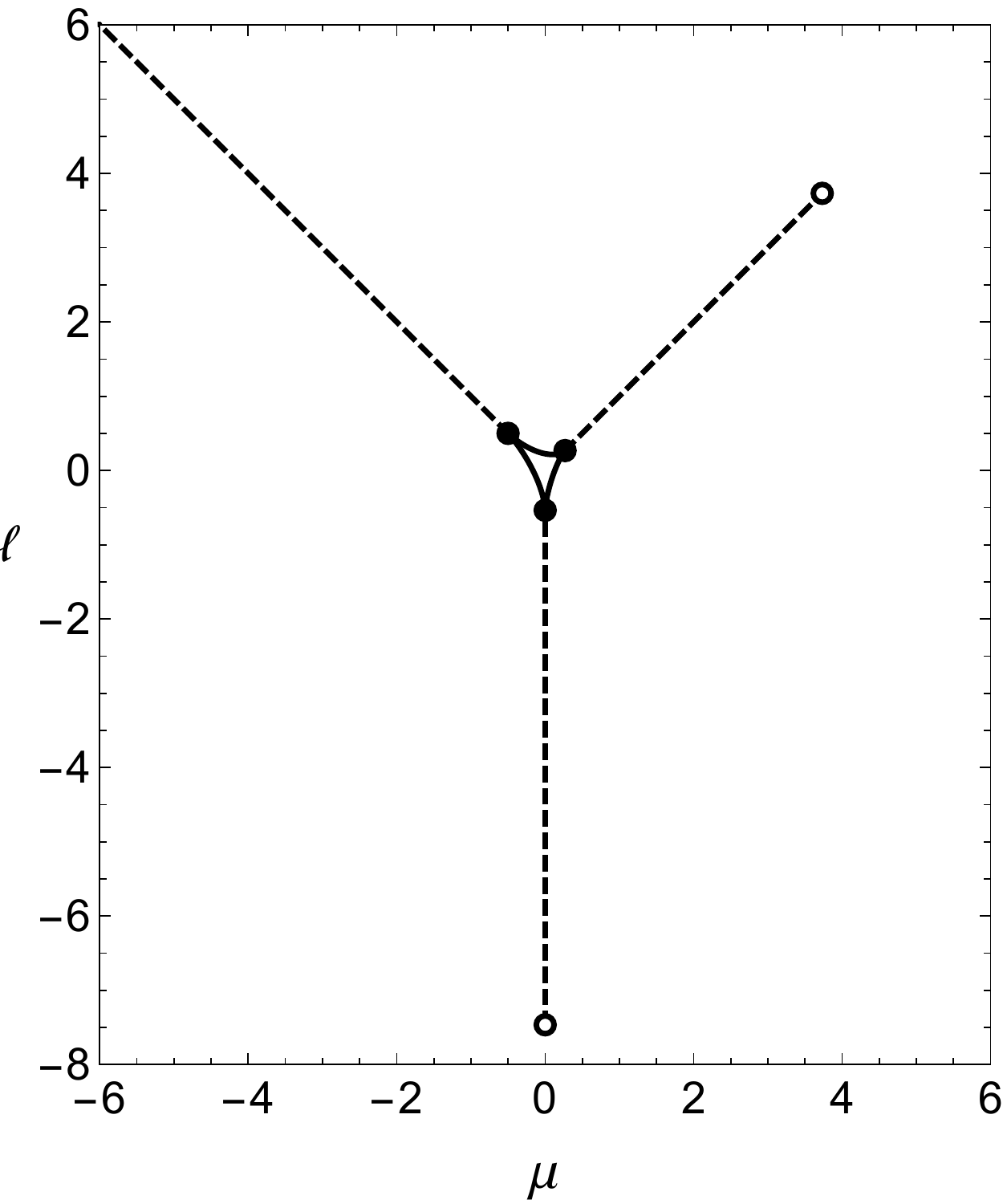}
\end{center}
\caption{
   The bifurcation diagram for a typical choice of normalized
   axially symmetric higher order terms
   in~\eqref{originaldetunedhamiltonian} showing two relevant
   phenomena:
   the effect of detuning near the origin and the
   termination of unstable normal modes at supercritical
   Hamiltonian Hopf bifurcations.
   Solid lines stand for quasi-periodic centre-saddle bifurcations
   while dashed lines stand for unstable periodic orbits which undergo
   Hamiltonian Hopf bifurcations at the bullets --- subcritical at a
   solid bullet~$\bullet$ and supercritical at an open bullet~$\circ$.
   The left figure is without detuning ($\delta = 0$); here, the origin
   $(\mu, \ell) = (0, 0)$ stands for the $1{:}1{:}{-}2$~resonant
   equilibrium and has three unstable normal modes.
   In the right figure $\delta \neq 0$ and the equilibrium has three
   stable normal modes that then undergo subcritical Hamiltonian
   Hopf bifurcations to become unstable.
\label{fig:introductorybifurcationdiagrams} }
\end{figure}

This paper is organized as follows. Before determining in
\sref{Dynamics} the general form of $\T^2$--symmetric higher order
terms of~\eqref{originaldetunedhamiltonian} we pass in
\sref{Kinematics} to rotated co-ordinates that better reveal that
not only $L$ and~$R$ but also~$N$ is a resonant oscillator.
\Sref{Kinematics} details the reduction of the $\T^2$--symmetry,
i.e.\ the kinematics, while in \sref{Dynamics} the
one-degree-of-freedom dynamics is used to construct the bifurcation
diagram.
The set of critical values and the resulting monodromy (and their
dependence on external parameters like~$\delta$) are discussed in
\sref{Critical values} and \sref{Monodromy}.
The final \sref{Conclusions} concludes the paper, coming back to the
relation between~\eqref{originaldetunedhamiltonian} and its normal
form truncated at order~$4$.

\section{Kinematics}
\label{Kinematics}

\noindent
The axial symmetry ensures that the
Hamiltonian~\eqref{originaldetunedhamiltonian} admits the three
normal modes
\begin{eqnarray*}
   (x_1(t), y_1(t)) \enspace , & &
   \mbox{$x_2 = -y_1$, $y_2 = x_1$, $x_3 = 0$, $y_3 = 0$} \\
   (x_2(t), y_2(t)) \enspace , & &
   \mbox{$x_1 = -y_2$, $y_1 = x_2$, $x_3 = 0$, $y_3 = 0$} \\
   (x_3(t), y_3(t)) \enspace , & &
   \mbox{$x_1 = 0$, $y_1 = 0$, $x_2 = 0$, $y_2 = 0$.}
\end{eqnarray*}
The term normal mode is often restricted to periodic orbits where
the remaining co-ordinates rest at~$0$ as in the normal $3$--mode,
instead of performing an `enslaved' oscillation as in the normal
$1$-- and $2$--modes.
To achieve the former for all normal modes we apply the orthogonal
change of variables defined by
\begin{equation}\label{cushmantransformation}
  \begin{pmatrix} x_1 \\ x_2 \\ y_1 \\ y_2  \end{pmatrix}
  \;\; = \;\; \frac{1}{\sqrt{2}}
  \begin{pmatrix}
    1 & 0 & 0 & -1 \\
    0 & 1 & -1 & 0 \\
    0 & 1 & 1 & 0 \\
    1 & 0 & 0 & 1
  \end{pmatrix}
  \begin{pmatrix} q_1 \\ q_2 \\ p_1 \\ p_2 \end{pmatrix}
  \enspace, \quad \mbox{$x_3 = q_3$, $y_3 = p_3$,}
\end{equation}
which turns the symplectic structure $\dd x \wedge \dd y$ into
$\dd q \wedge \dd p$.
The transformation~\eqref{cushmantransformation} leaves the form
of~$L$ and~$R$ invariant while $N = x_1 y_2 - x_2 y_1$ is revealed
to be the Hamiltonian
\begin{equation}
   \RO{1}{{-}1}{0} : \qquad N \;\; = \;\;
   \frac{p_1^2 + q_1^2}{2} \; - \; \frac{p_2^2 + q_2^2}{2}
\end{equation}
of three coupled oscillators in $1{:}{-}1{:}0$~resonance.
An advantage of this point of view is that adding oscillators
$\RO{m_1}{m_2}{m_3}$ and $\RO{n_1}{n_2}{n_3}$ yields again an
oscillator $\RO{m_1+n_1}{m_2+n_2}{m_3+n_3}$.

\begin{remark}
\label{addingamultiple}
Adding the multiple~$\beta$ of~$N$
in~\eqref{originaldetunedhamiltonian} yielded a detuning of the
$1{:}1$~resonant oscillators in the subspace
$(x_3, y_3) = (q_3, p_3) = 0$ to frequencies
$\alpha + \beta + \delta$ and $\alpha - \beta + \delta$ (next
to~$-2 \alpha$) without breaking the symmetry generated by~$N$.
Note that adding a multiple~$\beta$ of
\begin{subequations}
\begin{align*}
   \frac{x_1^2 + y_1^2}{2} \; - \; \frac{x_2^2 + y_2^2}{2}
   \;\; & = \;\; - (q_1 p_2 - q_2 p_1)
\intertext{or of}
   x_1 x_2 \, + \, y_1 y_2 \;\; & = \;\; q_1 q_2 \, + \, p_1 p_2
\end{align*}
\end{subequations}
would yield the same detuning, but at the price of breaking the
symmetry generated by~$N$.
A general detuning of the $1{:}1$~subresonance has indeed co-dimension
three and would lead to all the phenomena detailed in~\cite{HH17}
concerning higher order terms in the normal form.
\end{remark}

\subsection{Isotropy}
\label{Isotropy}

\noindent
Let $\X{G} : \dot{F} = \{ F, G \}$ denote the Hamiltonian vector field
defined by a function~$G$ and $\varphi^G_t$ the corresponding flow.
Identify $\R^6 \cong \C^3$ by introducing complex co-ordinates
$z_j = p_j + \ii q_j$, $j=1,2,3$.
The flow $\varphi^N_t$ of~$\X{N}$ induces an $\setS^1$--action
on~$\C^3$.
For treating monodromy later on we prefer to have integer
periodicities, so let us define
\begin{displaymath}
   \T^1 \;\; := \;\; \R\quotient{\Z}
\end{displaymath}
(an $\setS^1$ with radius~$\Frac{1}{2 \pi}$) and use
$z = (z_1, z_2, z_3)$ to write the $\T^1$--action induced
by~$\varphi^N_t$ as
\begin{equation}
\label{axialsymmetry}
\begin{array}{rccl}
   \varphi^N : & \T^1 \times \C^3 & \longrightarrow & \C^3 \\
   & (t, z) &  \mapsto & \varphi^N_{2 \pi t}(z) \; = \;
   (\ee^{2 \pi \ii t} z_1 ,\, \ee^{-2 \pi \ii t} z_2,\, z_3)
\end{array}
\enspace.
\end{equation}
This action has trivial isotropy, except at $z_1 = z_2 = 0$ where the
isotropy subgroup is~$\T^1$.
Similarly, the flow $\varphi^L_t$ of~$\X{L}$ induces a $\T^1$--action
on~$\C^3$ given by
\begin{equation}
\label{oscillatorsymmetry}
\begin{array}{rccl}
   \varphi^L : & \T^1 \times \C^3 & \longrightarrow & \C^3 \\
   & (t, z) &  \mapsto & \varphi^L_{2 \pi t}(z) \; = \;
   (\ee^{2 \pi \ii t} z_1,\, \ee^{2 \pi \ii t} z_2,\,
   \ee^{-4 \pi \ii t} z_3)
\end{array}
\enspace.
\end{equation}
This action has non-trivial isotropies $\Z_2$ when $z_1 = z_2 = 0$
and $\T^1$ when $z_1 = z_2 = z_3 = 0$.
Combining the two commuting $\T^1$--actions $\varphi^N$ and~$\varphi^L$
one can define an action of $\T^2 = \T^1 \times \T^1$ on $\C^3$ given by
\begin{equation}
\label{T2-action-1}
\begin{array}{ccl}
   \T^2 \times \C^3 & \longrightarrow & \C^3 \\
   (s, t, z) &  \mapsto &
   \varphi^L_{2 \pi t} \circ \varphi^N_{2 \pi s}(z) \; = \;
   (\ee^{2 \pi \ii (s+t)} z_1,\, \ee^{2 \pi \ii (t-s)} z_2,\,
   \ee^{-4 \pi \ii t} z_3)
\end{array}
\enspace.
\end{equation}
A direct computation shows that the
$\T^2$--action~\eqref{T2-action-1} is not \emph{effective} as the
element $(s, t) = (\frac{1}{2}, \frac{1}{2}) \in \T^2$ acts as the
identity on all of~$\C^3$.

\begin{table}[tbp!]
\begin{center}
\begin{small}
\begin{tabular}{c|cccc}
set & elements & isotropy subgroup & $\Phi$--orbit & dynamics \\ \hline
$C_{123}$ & $z_1 = z_2 = z_3 = 0$ & $\T^2$ & point & equilibrium \\
$C_{12}$ & $z_1 = z_2 = 0 \ne z_3$ & $\{ (s, t) \mid t=0 \} \cong \T^1$
 & $\T^1$ & normal $3$--mode \\
$C_{13}$ & $z_1 = z_3 = 0 \ne z_2$ & $\{ (s, t) \mid s=0 \} \cong \T^1$
 & $\T^1$ & normal $2$--mode \\
$C_{23}$ & $z_2 = z_3 = 0 \ne z_1$ & $\{ (s, t) \mid s+t=0 \} \cong \T^1$
 & $\T^1$ & normal $1$--mode
\end{tabular}
\caption{
   Non-trivial isotropy groups of the
   $\T^2$--action~\eqref{T2actionPhi} and types of $\Phi$--orbits.
\label{tab:isotropy} }
\end{small}
\end{center}
\end{table}

We need a pair of generators of $\T^1$--actions for which the
combined $\T^2$--action is effective and coincides
with~\eqref{T2-action-1} projected to
$\T^2/(\frac{1}{2}, \frac{1}{2})$.
Such a pair is given by $N$ and~$J$ where the latter is defined as
\begin{equation}
   \RO{1}{0}{{-}1} : \qquad J \;\; = \;\; \frac{1}{2}(N + L)
   \enspace.
\end{equation}
The flows $\varphi^N_s$ and $\varphi^J_t$ on~$\R^6$ combine to the
effective $\T^2$--action
\begin{equation}
\label{T2actionPhi}
\begin{array}{rccc}
   \Phi : & \T^2 \times \C^3 & \longrightarrow & \C^3 \\
   & (s, t, z) &  \mapsto &
   (\ee^{2 \pi \ii (s+t)} z_1,\, \ee^{- 2 \pi \ii s} z_2,\,
   \ee^{-2 \pi \ii t} z_3)
\end{array}
\end{equation}
with momentum mapping $(N, J) : \R^6 \longrightarrow \R^2$.
Table~\ref{tab:isotropy} summarizes the isotropies of~$\Phi$ and the
topological types of the corresponding
$\Phi$--orbits.

\subsection{Reduction}
\label{Reduction}

\noindent
The reduction of the $\T^2$--symmetry is best performed by using a
Hilbert basis of the algebra of $\T^2$--invariant functions as
variables on the reduced phase space, see~\cite{CB97}.
As proven in~\cite{gws75, jnm77} this algebra is generated by
$\T^2$--invariant polynomials.

\begin{proposition}
\label{prop:T2inv}
The algebra $\mathcal{A}$ of invariant polynomials of the action $\Phi$
in~\eqref{T2actionPhi} is generated by the real polynomials $N$, $J$
and $R$ together with $X$ and $Y$ defined by
\begin{eqnarray*}
  X & = & \re (z_1 z_2 z_3)
          \;\; = \;\; p_1 p_2 p_3 \; - \; q_1 q_2 p_3
          \; - \; q_1 p_2 q_3 \; - \; p_1 q_2 q_3
    \enspace,
  \\
  Y & = & \im (z_1 z_2 z_3)
          \;\; = \;\; q_1 p_2 p_3 \; + \; p_1 q_2 p_3
          \; + \; p_1 p_2 q_3 \; - \; q_1 q_2 q_3
   \enspace.
\end{eqnarray*}
These satisfy the syzygy
\begin{equation}
\label{T2syz-J}
   S(N, J, R, X, Y) \;\; := \;\;
   X^2 \, + \, Y^2 \; - \; (R^2 - N^2) (R + N - 2 J) \;\; = \;\; 0
   \enspace,
\end{equation}
together with the inequality
$R \geq R_{\min} := \max(|N|, 2 J - N) \geq 0$.
These relations define a semi-algebraic variety in~$\R^5$.
\end{proposition}

\noindent
Note that $(N, L) \mapsto (N, J)$ has the inverse $L = 2 J - N$ whence
it is also possible to use the generators $N, L, R, X, Y$
of~$\mathcal{A}$ which satisfy
\begin{equation}
\label{T2syz-L}
   S(N, L, R, X, Y) \;\; = \;\;
   X^2 \, + \, Y^2 \; - \; (R^2 - N^2) (R - L) \;\; = \;\; 0
\end{equation}
and $R \geq R_{\min} = \max(|N|, L) \geq 0$.
In what follows we switch between these two sets of generators
of~$\mathcal{A}$ depending on which description is the most convenient.

\begin{table}[tb!]
\begin{center}
\begin{small}
\begin{tabular}{l|cccccc}
$\{ \downarrow , \rightarrow \}$ & $N$ & $L$ & $J$& $R$ & $X$ & $Y$
\\ \hline
$N$ & 0 & 0 & 0 & 0 & 0 & 0 \\
$L$ & 0 & 0 & 0 & 0 & 0 & 0 \\
$J$ & 0 & 0 & 0 & 0 & 0 & 0 \\
$R$ & 0 & 0 & 0 & 0 & $2Y$ & $-2X$ \\
$X$ & 0 & 0 & 0 & $-2Y$ & 0 & $N^2 + 2LR - 3R^2$ \\
$Y$ & 0 & 0 & 0 & $2X$ & $3R^2 - 2LR - N^2$ & 0
\end{tabular}
\caption{
   The Poisson bracket relations among the basic invariants.
\label{tab:brackets} }
\end{small}
\end{center}
\end{table}

For the Poisson structure we only need the Poisson bracket relations
given in table~\ref{tab:brackets}.
As expected, the symmetry generators $N$, $L$ and~$J$ are Casimir
functions and fixing their values to $2 \iota = \mu + \ell$, where
$\iota$~denotes the value of~$J$, yields the semi-algebraic surface
\begin{equation}\label{semialgvar}
   \cP_{\mu \ell} \;\; = \;\;
   \left\{ \; (R, X, Y) \in \R^3 \mitdE R \geq R_{\min},\,
   S_{\mu \ell}(R, X, Y) = 0 \; \right\} \enspace ,
\end{equation}
with Poisson structure
\begin{displaymath}
   \{ f, g \} \;\; = \;\; \langle \nabla f \times \nabla g \mid
   \nabla S_{\mu \ell} \rangle
   \enspace,
\end{displaymath}
where
\begin{equation*}
   S_{\mu \ell}(R, X, Y) \;\; = \;\;
   X^2 \, + \, Y^2 \; - \; (R^2 - \mu^2) (R - \ell)
   \enspace .
\end{equation*}
Any smooth (resp.\ polynomial) Hamiltonian function $H$ on~$\R^6$ that
is invariant under the $\T^2$--action~$\Phi$ can be expressed as a
smooth (resp.\ polynomial) function of the generators of~$\mathcal{A}$,
see~\cite{gws75, jnm77}.
Alternatively, $H$ gives rise to a function~$H_{\mu \ell}$ on the
reduced phase space~$P_{\mu \ell}$.

The semi-algebraic variety~$P_{\mu \ell}$ is a surface of revolution
about the $R$--axis.
The type of singularities of~$P_{\mu \ell}$ is given by the following
result and the different possibilities are depicted in \fref{fig:tip-3}.

\begin{proposition}
\label{prop:singtype}
The semi-algebraic variety~$\cP_{\mu \ell}$ given
by~\eqref{semialgvar} is smooth everywhere except possibly at its
`tip' point $(R, X, Y) = (R_{\min}, 0, 0)$.
The latter point is a cusp (or cuspidal singularity of order~$3$)
when $\mu = \ell = 0$; a cone (or conical singularity) when
$\ell = |\mu| > 0$ or $\mu = 0$, $\ell < 0$; and smooth in all other
cases.
\end{proposition}

\begin{figure}[t!]
\begin{center}
\begin{picture}(400,120)
   \put(0,0){\epsfig{file=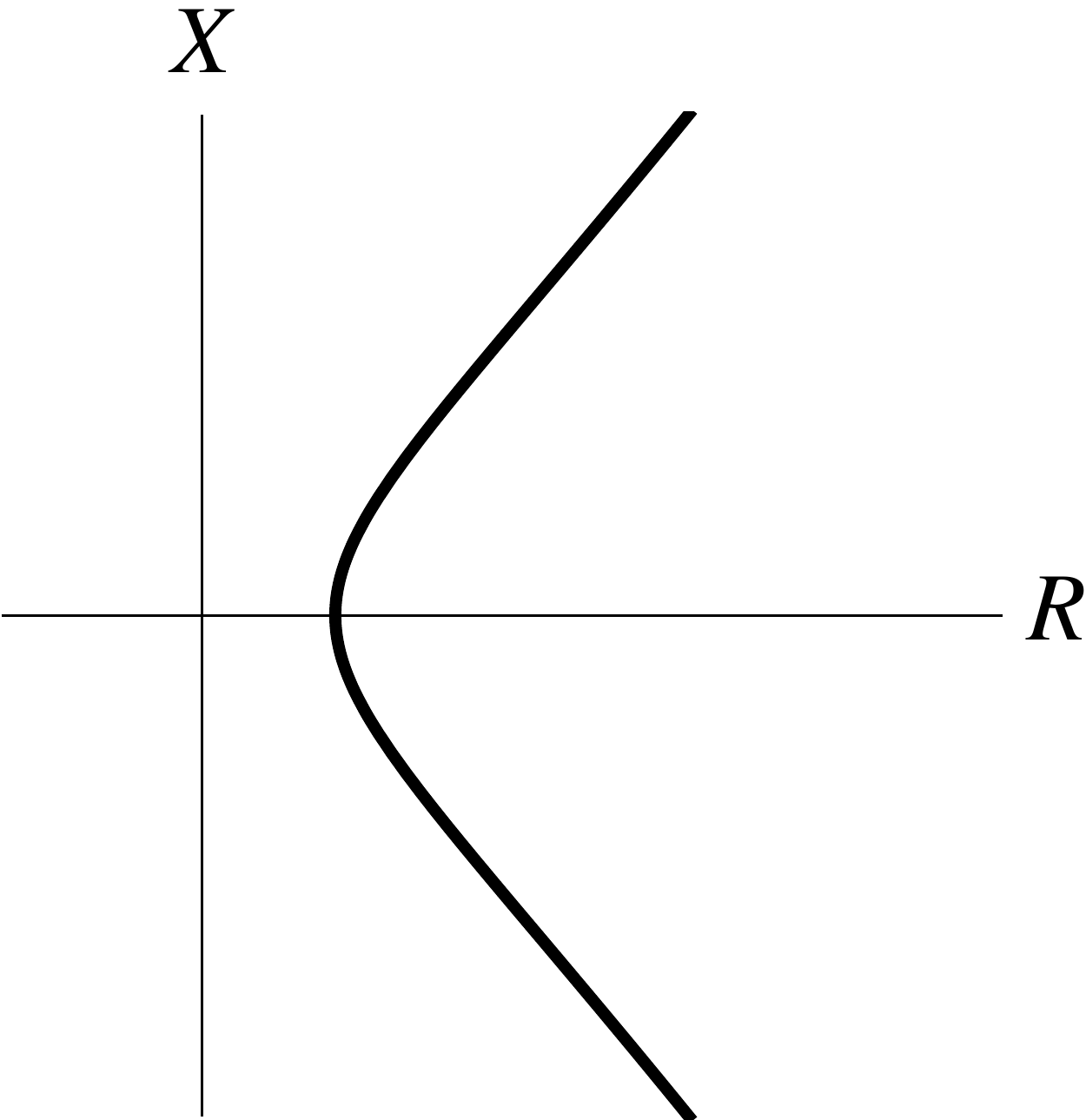, width=120pt}}
   \put(140,0){\epsfig{file=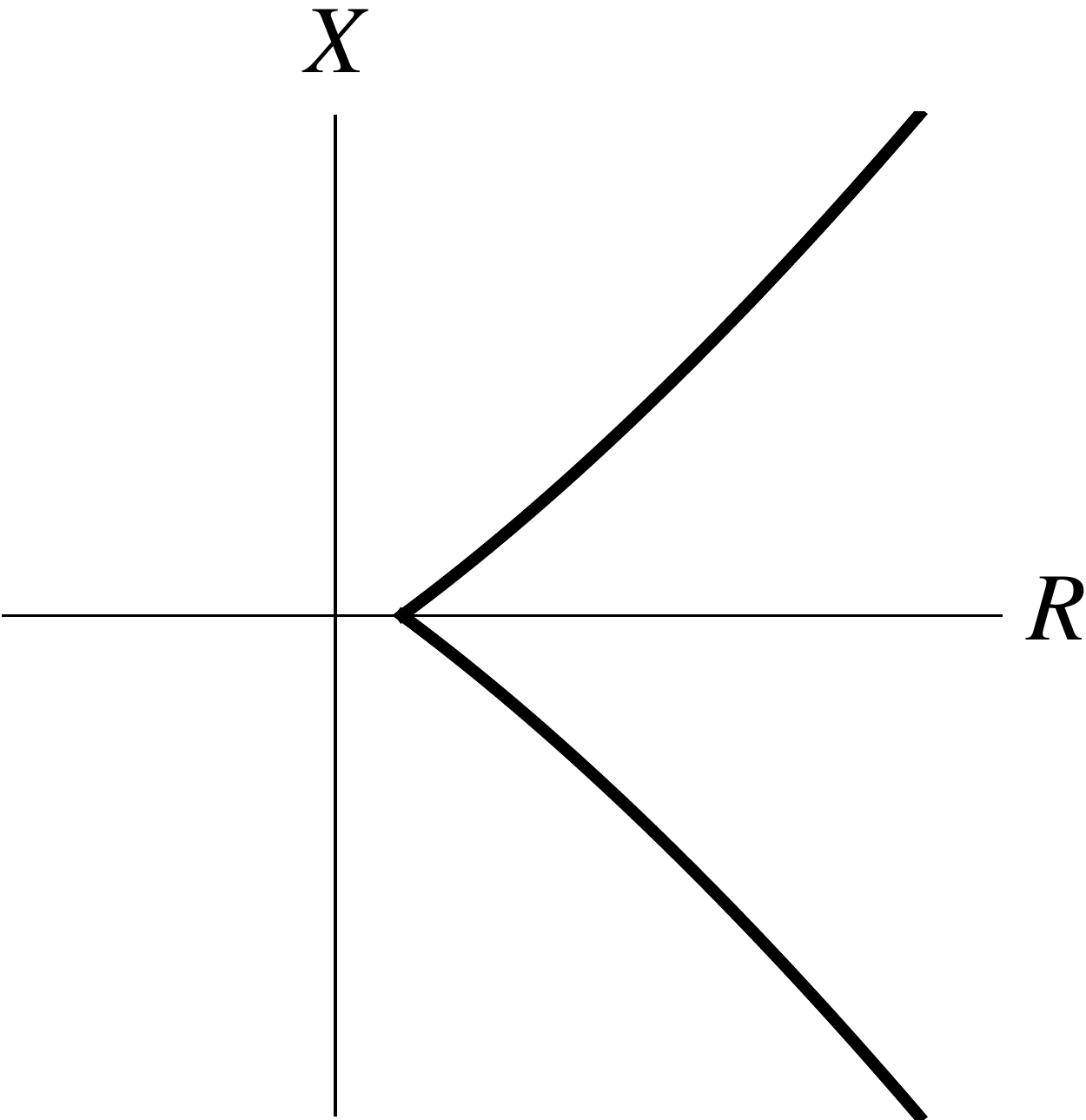, width=120pt}}
   \put(280,0){\epsfig{file=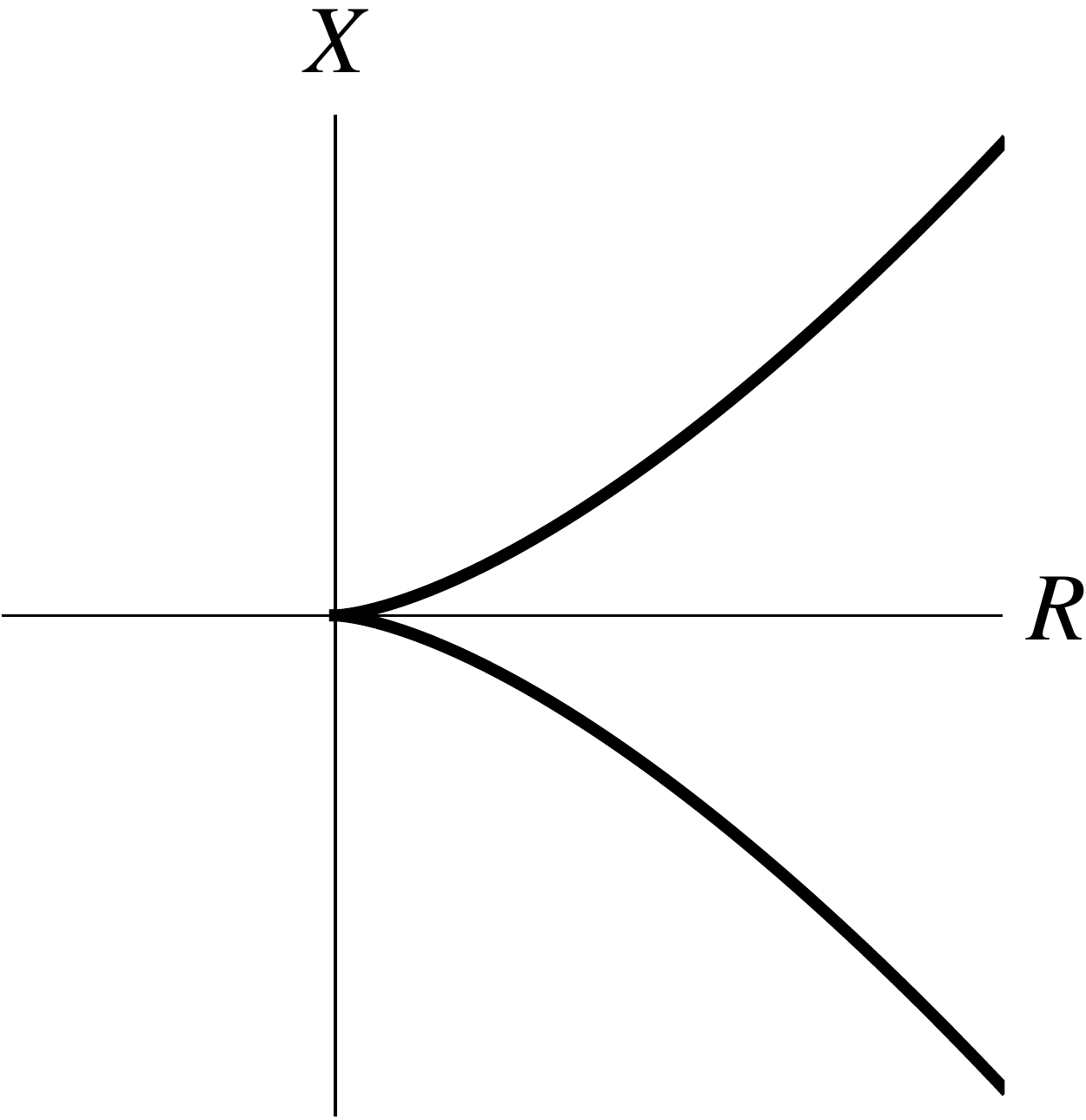, width=120pt}}
\end{picture}
\end{center}
\caption{
   The types of the singularities of the reduced spaces~$\cP_{\mu \ell}$
   (which are surfaces of revolution about the $R$--axis).
   From left to right: smooth, conical and cuspidal.
\label{fig:tip-3} }
\end{figure}

\begin{proof}
Since the semi-algebraic variety~$\cP_{\mu \ell}$ is a surface of
revolution it is sufficient to consider its section with the plane
$\{ Y = 0 \}$.
Then we have
\begin{equation*}
   X^2 \; = \; (R^2 - \mu^2) (R - \ell)
   \enspace , \quad 
   R \; \ge \; R_{\min} \; = \; \max(|\mu|,\ell) \; \ge \; 0
   \enspace .
\end{equation*}
Let $a_1 \ge a_2 \ge a_3$ denote the ordered values in the set
$\{\mu, -\mu, \ell\}$ and write
\begin{equation}
\label{cubiconecusp}
   X^2 \; = \; (R - a_1) (R - a_2) (R - a_3)
   \enspace , \quad
   R \; \ge \; R_{\min} \; = \; a_1 \; \ge \; 0
   \enspace .
\end{equation}
The type of singularity at $R = a_1$ of $X^2 = f(R)$ with
$f(a_1) = 0$ is determined by the lowest order term in the
Taylor expansion of $f(R)$ at $R = a_1$.
In particular, writing $X^2 = c (R-a_1)^k + h.o.t.$
we have a smooth curve when $k=1$, a conical singularity when
$k=2$, and a cuspidal singularity when $k \ge 3$.
Expand~\eqref{cubiconecusp} in the form
\begin{displaymath}
   X^2 \;\; = \;\; [(a_1 - a_2) (a_1 - a_3)] \, (R - a_1)
   \; + \; [2a_1 - a_2 - a_3] \, (R - a_1)^2
   \; + \; (R - a_1)^3
   \enspace .
\end{displaymath}
The last expression shows that we have a cusp when $a_1 = a_2 = a_3$
implying $\mu = \ell = 0$.
We have a cone when $a_1 = a_2 > a_3$ and one can check that this
gives the cases $\ell = |\mu| > 0$ and $\mu = 0$, $\ell < 0$.
The other cases lead to a smooth point.
\end{proof}

\noindent
In particular, there are three open regions in the $(\mu, \ell)$--plane,
each region being characterised by the (positive) value of~$R_{\min}$,
that is,
\begin{displaymath}
   R_{\min} \; = \; \mu \; > \; \max(-\mu, \ell)
   \quad \mbox{or} \quad
   R_{\min} \; = \; -\mu \; > \; \max(\mu, \ell)
   \quad \mbox{or} \quad
   R_{\min} \; = \; \ell \; > \; |\mu|
   \enspace .
\end{displaymath}
The three regions are separated by three half lines satisfying
$\mu = \pm \ell$ or $\mu = 0$.
Along these lines the tip is a cone while at $\mu = \ell = 0$ the tip
is a cusp.
The three half lines and their common limit point parametrising singular
tips are the images under the momentum mapping $(N, L)$ of the sets
$C_{23}$, $C_{13}$, $C_{12}$, $C_{123}$ in Table~\ref{tab:isotropy}
with non-trivial isotropies $\T^1$ and~$\T^2$ of the action~$\Phi$.
After reconstruction to three degrees of freedom these are the
elliptic equilibrium at the origin and its three normal modes.

\section{Dynamics}
\label{Dynamics}

\noindent
The most general $\T^2$--symmetric higher order terms
in~\eqref{originaldetunedhamiltonian} are
functions of $N$, $L$, $R$, $X$ and~$Y$.
As in~\eqref{normalizedhamiltonian} we normalize to order~$4$
in the original variables.
Then $X$ and~$Y$ appear only linearly and a rotation in the
$(X, Y)$--plane removes the $Y$--term, whence
\begin{equation}
\label{symmetricdetunedhamiltonian}
\begin{aligned}
   H^{\delta}_{N, L}(R, X, Y) \;\; & = \;\; \alpha L \, + \,
   \beta N \, + \, \delta R \; + \; X \; + \; \frac{\kappa}{2} R^2 \\
   & \phantom{=} \mbox{} + \; (\lambda_1 N + \lambda_2 L) R \; + \;
  \frac{\gamma_1}{2} N^2 + \gamma_2 N L + \frac{\gamma_3}{2} L^2
\end{aligned}
\end{equation}
is the most general choice in orders $3$ and~$4$.
The coefficient $1$ of~$X$ can be obtained scaling time or space.
Fixing the values $N = \mu$ and $L = \ell$ of the Casimirs we omit
the constant part
$\alpha \ell + \beta \mu + \frac{1}{2} \gamma_1 \mu^2
 + \gamma_2 \mu \ell + \frac{1}{2} \gamma_3 \ell^2$
and abbreviate
\begin{equation}
\label{linearparameterdependence}
   \lambda \;\; := \;\;
   \delta \; + \; \lambda_1 \mu \; + \; \lambda_2 \ell
\end{equation}
to obtain the reduced Hamiltonian
\begin{equation}
\label{reducedhamiltonian}
   \cH_{\lambda}(R, X, Y) \;\; = \;\;
   X \; + \; \lambda R \; + \; \frac{\kappa}{2} R^2
   \enspace ,
\end{equation}
whose energy level sets are the parabolic cylinders
\begin{displaymath}
   \cH_{\lambda}^{-1}(h) \;\; = \;\; \left\{ \; (R, X, Y) \in \R^3
   \mitdE \cH_{\lambda}(R, X, Y) = h \; \right\}
   \enspace .
\end{displaymath}
The intersections of these with the reduced phase space
$\cP_{\mu \ell} \subseteq \R^3$ yield the orbits of~$\cH_\lambda$.
The values $\mu$ and~$\ell$ of the Casimirs serve as internal (or
distinguished) parameters.
We have equilibria where the two surfaces touch
and where $\cH_{\lambda}^{-1}(h)$ contains a singular
point of~$\cP_{\mu \ell}$.
We fix $\kappa \neq 0$ and only vary~$\lambda$ as an external
parameter.

\begin{remark}
The coefficient $\kappa \ne 0$ in~\eqref{symmetricdetunedhamiltonian}
could be chosen as $\kappa = 1$ through the combined scaling
\begin{equation}
\label{kappa=1 scaling}
  (\lambda, \mu, \ell, R, X, Y, H) \;\; \mapsto \;\;
  (\kappa^{-1} \lambda, \kappa^{-2} \mu, \kappa^{-2} \ell,
  \kappa^{-2} R, \kappa^{-3} X, \kappa^{-3} Y, \kappa^{-3} H)
\end{equation}
of space and time (which leaves the coefficient of~$X$ equal to~$1$).
However, the dynamical consequence is that the basic frequency (or
common multiple) of the unperturbed $1{:}1{:}{-}2$~resonant oscillator,
the coefficient $\alpha$ of $L$ in~\eqref{symmetricdetunedhamiltonian},
gets scaled to $\kappa \alpha$ and can no longer be scaled to~$1$.
We therefore refrain from scaling the coefficient of $R$
in \eqref{symmetricdetunedhamiltonian} to~$\frac{1}{2}$ until
\eqref{finallydoscalekappa} in section~\ref{Critical values} below.
Note that for negative $\kappa$ the scaling~\eqref{kappa=1 scaling}
involves a reflection (combined with a scaling) of the
parameter~$\lambda$ and the quantities $X$, $Y$ and~$H$.
\end{remark}

\noindent
Our aim here is to determine the bifurcation diagram for
the Hamiltonian $\cH_\lambda$ in the parameter space
$(\delta, \mu, \ell)$.
The coefficients $\lambda_1$ and~$\lambda_2$  are fixed
at rather small values while the detuning~$\delta$ (and
hence~$\lambda$) is allowed to vary.
For fixed values of $\lambda_1$ and $\lambda_2$,
equation~\eqref{linearparameterdependence} defines a
diffeomorphism between $(\lambda, \mu, \ell)$--space and
$(\delta, \mu, \ell)$--space, allowing direct translation
of our findings between these two parameter spaces.

\subsection{Regular equilibria and their bifurcations}
\label{Bifurcationsofregularequilibria}

\noindent
Regular equilibria occur at smooth points of~$\cP_{\mu \ell}$
touching the parabolic cylinders~$\cH_{\lambda}^{-1}(h)$.
They correspond to invariant $2$--tori in three degrees
of freedom and their bifurcations are triggered by two of these
equilibria, or (all) three, meeting.
All tangent planes of the parabolic cylinder
$\cH_{\lambda}^{-1}(h)$ are parallel to the $Y$--axis while a tangent
plane of the surface of revolution~$\cP_{\mu \ell}$ can be parallel to the
$Y$--axis only at points $(R, X, Y)$ with $Y=0$.
Thus, the two surfaces can only touch at points in the $(Y{=}0)$--plane
and for this it is necessary and sufficient that the parabola
\begin{align}
   \cH_{\lambda}^{-1}(h) \cap \{ Y=0 \} \; : \qquad &
   X \;\; = \;\; X_1(R)
   \;\; = \;\; 
   h \; - \; \lambda R \; - \; \frac{\kappa}{2} R^2
\label{parabola11-2}
\intertext{has the same derivative $X' = X'(R)$ as the variety}
\label{phasespacesection11-2}
   \cP_{\mu \ell} \cap \{ Y=0 \} \; : \qquad &
   X^2 \;\; = \;\; X_2(R)^2 \;\; = \;\; (R^2 - \mu^2)(R - \ell)
   \enspace , \quad R \ge R_{\min}
\end{align}
where we choose
\begin{displaymath}
   X_2(R) \;\; = \;\; \sqrt{(R^2 - \mu^2)(R - \ell)}
\end{displaymath}
as the `upper' side.
Note that the value $h$ of the energy can always be adjusted to ensure
that the two derivatives become equal at a common point $(R, X)$ if
$R \geq R_{\min}$.
To compute the value of~$R$ we equate the slope
\begin{equation}
\label{slopeparabola}
   X_1^{\prime} \;\; = \;\; - \lambda \; - \; \kappa R
\end{equation}
of~\eqref{parabola11-2} with the slope of the
variety~\eqref{phasespacesection11-2}.
Adjusting~$h$ to actually have a common point of \eqref{parabola11-2}
and~\eqref{phasespacesection11-2} we can use (the square of)
\begin{equation}
\label{slopecubic}
   2 X_2 X_2^{\prime} \;\; = \;\;
   3 R^2 \; - \; 2 \ell R \; - \; \mu^2
\end{equation}
to obtain $R = R(\mu, \ell; \lambda, \kappa)$ from
\begin{equation}
\label{Rvaluesofequilibria}
\begin{array}{rcl}
   0 & = & (2 X_2 X_1')^2 \; - \; (2 X_2 X_2')^2  \\[3pt]
   & = &  4(\kappa R + \lambda)^2 (R - \ell) (R^2 - \mu^2) \; - \;
   (3 R^2 - 2 \ell R - \mu^2)^2
   \enspace .
\end{array}
\end{equation}
An alternative method of obtaining
equation~\eqref{Rvaluesofequilibria} that gives
the equilibria of the system is to search for
double roots
of the polynomial $F(R) = X_1^2(R) - X_2^2(R)$,
cf.\ equation~\eqref{F(R)} below.
The polynomial $F(R)$ is used
in \sref{bifurcation diagram kappa != 0}
to compute the bifurcations in the
system, see lemma~\ref{lemma bd char}.

\begin{proposition}
  Equation~\eqref{Rvaluesofequilibria} has at most three solutions in
  the interval $[R_{\min},\infty[$.
  There can be either one or three distinct solutions and in
  both cases the number of elliptic equilibria (centres)
  exceeds the number of hyperbolic equilibria (saddles) by~$1$.
\end{proposition}

\begin{proof}
Let $S(R)$ denote the fifth order polynomial in $R$ appearing
in~\eqref{Rvaluesofequilibria}.
Consider the root $R_* = \frac13 (\ell - (\ell^2 + 3 \mu^2)^{1/2})$
of the polynomial $3R^2 - 2\ell R - \mu^2$ in~\eqref{slopecubic}.
Then one checks that $R_* \le R_{\min}$ and $S(R_*) \ge 0$ while
$S(R_{\min}) \le 0$.
This implies that $S$ must have at least two real roots in
$\opin{{-}\infty}{R_{\min}}$, so it can have either $1$ or~$3$
distinct real roots in $[R_{\min}, \infty[$.

Multiple roots are characterised by $S'(R) = 0$ whence distinct
roots have $S'(R) \ne 0$ corresponding to elliptic equilibria if
$S'(R) > 0$ and to hyperbolic equilibria if $S'(R) < 0$.
The largest root of $S(R)$ must have $S'(R) > 0$ and is
thus elliptic while the two smaller ones (if they exist and
are distinct) alternate between hyperbolic and elliptic.
\end{proof}

\begin{remark}
\label{frequencyratio}
  The regular equilibria $(R, X, 0)$ of the reduced system yield
  invariant $2$--tori in three degrees of freedom and it depends on
  the ratio between the internal frequencies
  \begin{subequations}
    \label{internalfrequencies}
    \begin{align}
      \label{internalfrequency-N}
      \frac{\partial H}{\partial N} \;\; & = \;\; \beta \, - \, \alpha
      \; + \; (\gamma_1 - \gamma_2) \mu \, + \, (\gamma_2 - \gamma_3) \ell
      \; + \; (\lambda_1 - \lambda_2) R
      \intertext{and}
      \label{internalfrequency-J}
      \frac{\partial H}{\partial J} \;\; & = \;\;
      2 \left( \alpha \, + \, \gamma_2 \mu \, + \, \gamma_3 \ell
      \, + \, \lambda_2 R \right)
    \end{align}
  \end{subequations}
  (defined by the normal form $H = H^{\delta}_{NJ}$ obtained
  from~\eqref{symmetricdetunedhamiltonian} by replacing
  $L \mapsto 2 J - N$) whether the resulting trajectories are
  periodic or quasi-periodic.
\end{remark}

\noindent
For a centre and a saddle to meet and vanish in a centre-saddle
bifurcation the pa\-ra\-bo\-la~\eqref{parabola11-2} has to touch the
variety~\eqref{phasespacesection11-2} at an inflection point.
This is decided by the derivative $X_1^{\prime \prime} = - \kappa$
of~\eqref{slopeparabola} and the derivative
\begin{equation}
\label{secondslopecubic}
   X_2 X_2^{\prime \prime} \; + \; \left( X_2^{\prime} \right)^2 \;\; = \;\;
   3 R \; - \; \ell
\end{equation}
(of half) of~\eqref{slopecubic}.
Equating the two derivatives or, equivalently, finding the triple
roots of~$F(R)$ gives a polynomial expression whose roots correspond
to bifurcation points.
In particular, we obtain the polynomial
\begin{equation}
\label{Rvaluesofbifurcations}
\begin{array}{rcl}
   0 & = & \kappa^4 R^4 \; + \; \kappa^2 (4 \lambda \kappa - 7) R^3
   \; + \;
   3 (2 \lambda^2 \kappa^2 - 4 \lambda \kappa + \kappa^2 \ell + 3) R^2
   \\ & & \!\!\!\! \mbox{}
   + \; (4 \lambda^3 \kappa - 6 \lambda^2 + \kappa^2 \mu^2 +
   4 \lambda \kappa \ell - 6 \ell) R \; + \; \lambda^4 \, - \,
   \kappa^2 \mu^2 \ell \, + \, 2 \lambda^2 \ell \, + \, \ell^2
   \enspace.
\end{array}
\end{equation}
Inserting the solutions
$R = R(\mu, \ell; \lambda, \kappa) \geq R_{\min}$
into~\eqref{Rvaluesofequilibria} leads to three surfaces
$\lambda = \lambda(\mu, \ell)$ of centre-saddle bifurcations
\CS{k}, $k=1,2,3$ in parameter space (recall that we consider
$\kappa$ to be fixed).

This computation is done in \sref{bifurcation diagram kappa != 0} by
finding the triple roots of the polynomial $F(R)$ in~\eqref{F(R)},
see lemma~\ref{lemma bd char}.
Two surfaces parametrising centre-saddle bifurcations emanate from
each of the three curves of subcritical Hamiltonian Hopf
bifurcations \HHsub{k}, $k=1,2,3$ discussed in
\sref{HamiltonianHopfbifurcations} below, see
\fref{fig:3dpictureofbifurcationset}.
In particular, the surface~\CS{1} extends between \HHsub{2}
and~\HHsub{3}, the surface~\CS{2} extends between \HHsub{1}
and~\HHsub{3} and the surface~\CS{3} extends between \HHsub{1}
and~\HHsub{2}.
Each surface~\CS{k}, $k=1,2,3$ furthermore extends until a curve
segment~\Cusp{k} parametrising cusp bifurcations.
A fourth surface \CS{4} of centre-saddle bifurcations has as
boundary the union of the segments \Cusp{k}, $k=1,2,3$.

\begin{figure}[t!]
  \centering
  \includegraphics[width=0.48\textwidth]{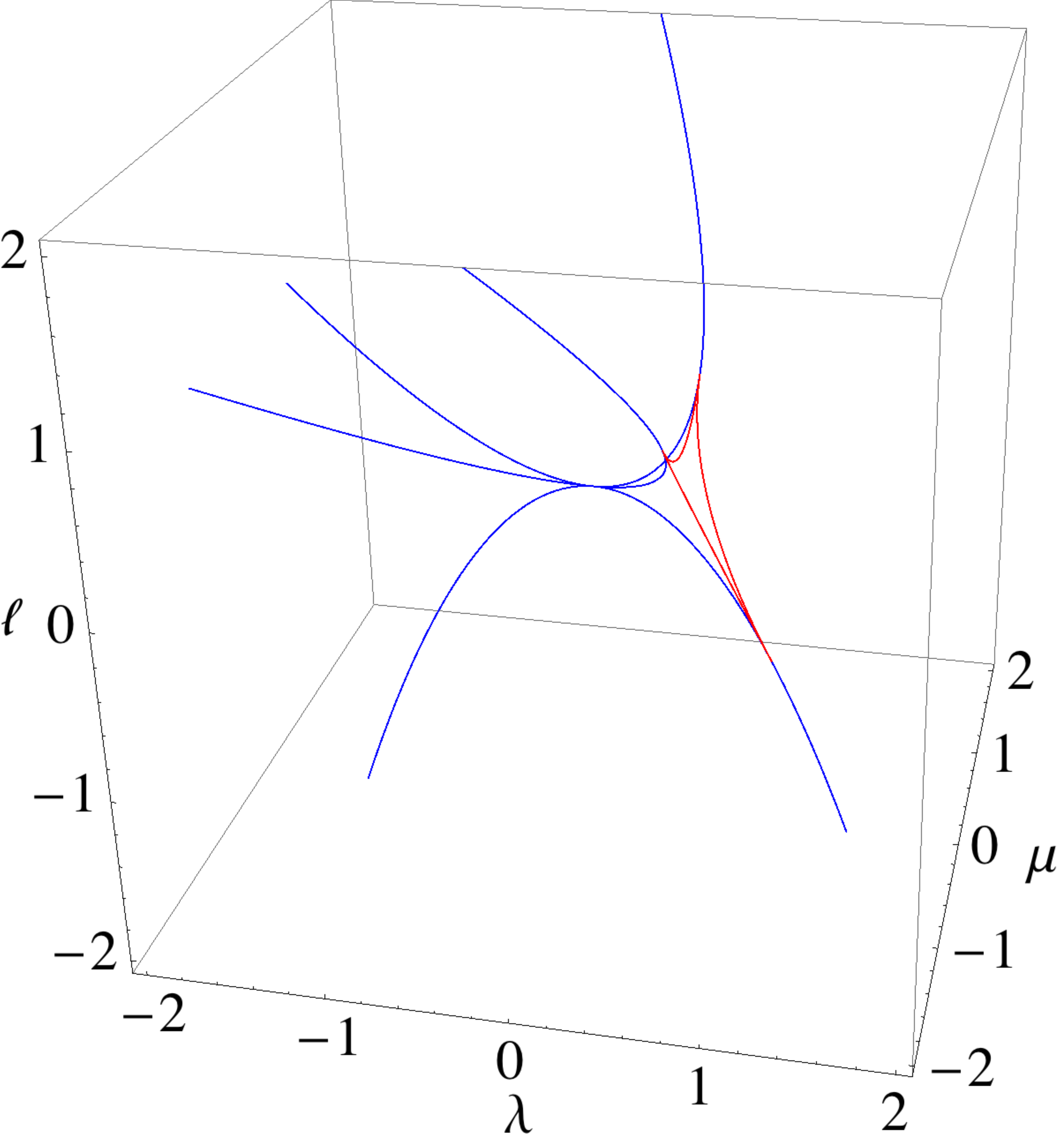}
  \hfill
  \includegraphics[width=0.48\textwidth]{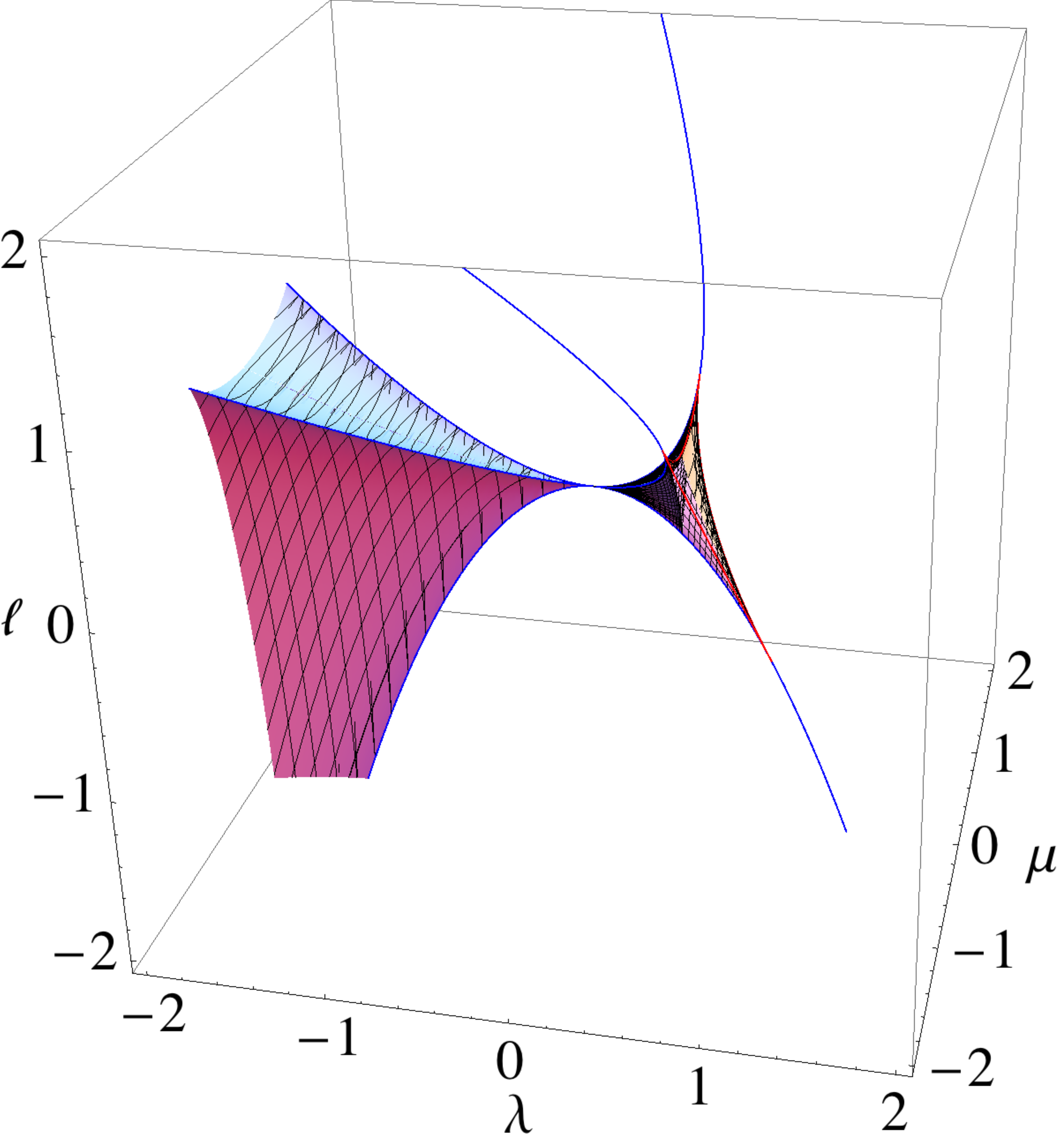}
  \caption{
    Bifurcation set for $\kappa = 1$.
    (a) Hamiltonian Hopf bifurcations (blue curves) 
    and cusp bifurcations (red curves).
    (b) Centre-saddle bifurcations are represented
    by the shown surfaces.
    The parts of the Hamiltonian Hopf bifurcation curves
    adjacent to centre-saddle bifurcations are
    subcritical Hamiltonian Hopf.
    The non-adjacent parts are supercritical.
    The transition takes place at the degenerate Hamiltonian
    Hopf bifurcations where the (red) cusp bifurcations emanate
    and the common point $(\lambda, \mu, \ell) = (0, 0, 0)$ of the
    three blue curves corresponds to the central equilibrium in
    $1{:}1{:}{-}2$~resonance.
  \label{fig:3dpictureofbifurcationset} }
\end{figure}

The non-degeneracy condition of a centre-saddle bifurcation requires
that the derivative of~\eqref{secondslopecubic} yields $X_2''' \neq 0$
(since the parabola~\eqref{parabola11-2} has zero third derivative).
Looking for $X_2''' = 0$ we find $\kappa(\lambda + \kappa R) = 1$.
The resulting
\begin{equation}
\label{Rvaluesofcusps}
   R \;\; = \;\; \frac{1}{\kappa^2} \; - \; \frac{\lambda}{\kappa}
\end{equation}
yields at the singular points $R=0$ and $R = \ell = |\mu|$
degenerate Hamiltonian Hopf bifurcations, see
sections~\ref{HamiltonianHopfbifurcations} and~\ref{degenerateHHb}
below.
The theory in~\cite{jcvdm96} predicts that two curves of cusp
bifurcations emanate from each degenerate Hamiltonian Hopf bifurcation.
The three degenerate Hamiltonian Hopf bifurcations form the vertices
of a curvilinear triangle with the families \Cusp{k}, $k=1,2,3$,
forming the corresponding edges.
The cusp bifurcations are non-degenerate.
Indeed, $X^{\prime \prime \prime}_2 = 0$ turns the second
derivative of~\eqref{secondslopecubic} into
$X_2 X_2^{(4)} + 3 (X^{\prime \prime}_2)^2 = 0$ whence
identifying $X_2^{\prime \prime}$ with
$X_1^{\prime \prime} \equiv -\kappa$ yields
\begin{displaymath}
   X_2^{(4)} \;\; = \;\; \frac{- 3\kappa^2}{X_2}
   \;\; \neq \;\; 0
\end{displaymath}
and the sign --- positive on the lower arc and negative on the upper
arc, i.e.\ the same sign as the second derivative of~$\cP_{\mu \ell}$
with respect to~$Y$ --- shows that this is the case of a cusp
bifurcation governed by the singularity~$A_3^+$, see~\cite{hh07}.
Note that the surfaces defined by $R \geq R_{\min}$,
\eqref{Rvaluesofequilibria}, \eqref{Rvaluesofbifurcations} and the
exclusion of the cusp lines indeed parametrise centre-saddle
bifurcations since the latter are given by exactly those parameter
values where the non-degeneracy condition is not satisfied.

\subsection{Hamiltonian Hopf bifurcations}
\label{HamiltonianHopfbifurcations}

\noindent
A \emph{singular} point $(R_{\min}, 0, 0)$ of~$\cP_{\mu \ell}$
(cone or cusp) is always an equilibrium.
We get bifurcations when $\cH_{\lambda}^{-1}(h^*)$, the parabolic
surface given by
\begin{equation}
\label{parabolaHHb}
   X \;\; = \;\; h^* \; - \; \lambda R \; - \; \frac{\kappa}{2} R^2
\end{equation}
(with
$h^* = \lambda R_{\min} + \frac{1}{2} \kappa R_{\min}^2
 = \lambda \ell + \frac{1}{2} \kappa \ell^2$
for $R_{\min} = \ell = \pm \mu$ and $h^* = 0$ for $R_{\min} = 0$),
touches $\cP_{\mu \ell}$ at the singular point, entering the
singular point with a tangent line that coincides with that of
the cone of~$\cP_{\mu \ell}$ (for the cusp this happens only when
$\lambda = \mu = \ell = 0$, i.e.\ at the $1{:}1{:}{-}2$~resonant
equilibrium).
Since the isotropy groups~$\T^1$ are not discrete we expect these
bifurcations to be Hamiltonian Hopf bifurcations.
The type of  intersection of $\cH_{\lambda}^{-1}(h^*)$, given
by~\eqref{parabolaHHb}, with $\cP_{\mu \ell}$ separates stable
equilibria $(R_{\min}, 0, 0)$ on~$\cP_{\mu \ell}$, where
$\cH_{\lambda}^{-1}(h^*)$ stays outside of $\cP_{\mu \ell}$,
from unstable equilibria, for which the intersection of
$\cH_{\lambda}^{-1}(h^*)$ with $\cP_{\mu \ell}$ yields their
stable$=$un\-stable manifolds in~$\cP_{\mu \ell}$.
In three degrees of freedom these correspond to elliptic and
hyperbolic periodic orbits; using the canonical equations of
motion we immediately see that these $\Phi$--orbits~$\T^1$
indeed do not consist of equilibria (they form the three
normal modes).

Correspondingly, for the supercritical type of the Hamiltonian Hopf
bifurcation the touching parabolic surface~$\cH_{\lambda}^{-1}(h^*)$
stays outside of $\cP_{\mu \ell}$ at the bifurcation parameter
(and the bifurcating equilibrium is dynamically stable), while for
the subcritical type $\cH_{\lambda}^{-1}(h^*)$ ~yields the
stable$=$un\-stable manifold of the bifurcating equilibrium
(which is therefore dynamically unstable).
The unstable periodic orbits resulting from the Hamiltonian Hopf
bifurcations largely determine the monodromy of the system,
discussed in \sref{Monodromy} below.

\begin{remark}
\label{lowordernormalinternal}
  Reducing only the $\T^1$--action~\eqref{oscillatorsymmetry} turns
  the normal modes into regular equilibria in two degrees of freedom
  that undergo Hamiltonian Hopf bifurcations.
  The Krein collision that triggers a Hamiltonian Hopf bifurcation
  and all other criteria~\cite{HM02} can be checked in one degree
  of freedom, but that the Krein collision does not occur at
  frequency~$0$ can only be verified in two degrees of freedom.
  However, the latter is merely needed for an approximate
  $\T^1$--symmetry that allows to reduce to one degree of freedom
  (after an additional normalization) and in the present situation
  we already have imposed the $\T^1$--symmetry~\eqref{axialsymmetry}
  as an exact symmetry.
  The whole bifurcation analysis thus takes place in one degree of
  freedom.
  Note that without imposing the $\T^1$--symmetry generated by~$N$
  we would even have to check that a normal mode undergoing a
  Hamiltonian Hopf bifurcation does not have its internal frequency,
  the period, in (low order) normal-internal resonance with the
  normal frequencies at the Krein collision.
\end{remark}

\noindent
The analysis in \sref{bifurcation diagram kappa != 0}
shows that there are three $1$--parameter families of subcritical
Hamiltonian Hopf bifurcations \HHsub{k}, $k = 1, 2, 3$ that
join with three $1$--parameter families of supercritical Hamiltonian
Hopf bifurcations \HHsup{k}, $k = 1, 2, 3$; see
\fref{fig:3dpictureofbifurcationset}.
Their parametrisation is given in
proposition~\ref{description of bifurcation set}.
As discussed in \sref{Bifurcationsofregularequilibria} there are
two families of centre-saddle bifurcations emanating from each
family of subcritical Hamiltonian Hopf bifurcations; the families
of supercritical Hamiltonian Hopf bifurcations are isolated and
meet the corresponding subcritical ones at three degenerate
Hamiltonian Hopf bifurcations.

\subsection{Bifurcation diagram for $\kappa \ne 0$}
\label{bifurcation diagram kappa != 0}

\noindent
In this section we give a complete parametrisation of the
bifurcation diagram (shown in \fref{fig:3dpictureofbifurcationset})
of the system passing through a $1{:}1{:}{-}2$~resonance.
As announced after having obtained
equation~\eqref{Rvaluesofequilibria}
we consider the polynomial function
\begin{equation}
\label{F(R)}
   F(R) \;\; = \;\; X_1^2(R) \; - \; X_2^2(R) \;\; = \;\;
   \left( h - \lambda R - \frac{\kappa}{2} R^2 \right)^2
   \; - \; (R^2 - \mu^2) (R - \ell) \enspace .
\end{equation}
Recall that $X_1(R)$ represents the energy level set
$\cH_{\lambda}^{-1}(h)$, see~\eqref{parabola11-2}, while
$X_2(R)$ and $-X_2(R)$ represent the upper and lower side
of $P_{\mu \ell} \cap \{ Y = 0 \}$ respectively,
see~\eqref{phasespacesection11-2}. 
Then we have the following characterisation of the bifurcation set of
the reduced Hamiltonian~\eqref{reducedhamiltonian}, which allows for a
complete and uniform approach to parametrising the bifurcation set.

\begin{lemma}\label{lemma bd char}
  The set of parameter values $(h,\lambda,\kappa,\mu,\ell)$ for
  which~\eqref{F(R)} has a triple root $R=a$ (that is,
  $F(a) = F'(a) = F''(a) = 0${\rm )} with
  $a \ge R_{\min} = \max(|\mu|, \ell)$, describes the set of
  centre-saddle and Hamiltonian Hopf bifurcations.
  In particular, supercritical Hamiltonian Hopf bifurcations are
  characterised by $a = R_{\min}$ with $F'''(a) > 0$ and
  subcritical Hamiltonian Hopf bifurcations by
  $a = R_{\min}$ with $F'''(a) < 0$, while centre-saddle bifurcations
  are characterised by $a > R_{\min}$ with $F'''(a) \ne 0$.
  Quadruple roots $R = a > R_{\min}$, that is, moreover $F'''(a) = 0$
  but $F^{(4)}(a) \ne 0$, correspond to cusp bifurcations and
  quadruple roots $R = a = R_{\min}$ correspond to degenerate
  Hamiltonian Hopf bifurcations.
\end{lemma}

\begin{proof}
Suppose that $R=a$ is a triple root of $F(R)$.
The relation $F(a) = 0$ gives $X_1(a) = \pm X_2(a)$.
The `$+$' sign corresponds to a common point of the energy level
set with the upper side of $P_{\mu \ell} \cap \{ Y = 0 \}$,
the `$-$' sign with the lower side.

Subsequently $F'(a) = 0$ gives $X_1(a)X_1'(a) = X_2(a)X_2'(a)$.
If $X_1(a) = X_2(a) = 0$ then the previous equation is satisfied.
If $X_1(a) = \pm X_2(a) \ne 0$ then $X_1'(a) = \pm X_2'(a)$, implying
that the energy level set touches the corresponding side
of $P_{\mu \ell} \cap \{ Y = 0 \}$.
Since we consider a triple root $a$ of~$F$ the last implication
$X_1'(a) = \pm X_2'(a)$ also follows if $X_1(a) = X_2(a) = 0$
since $F''(a) = 0$ gives
\begin{equation}
\label{F''(a)=0}
   (X_1')^2(a) \, - \, (X_2')^2(a) \; + \;
   X_1(a) X_1''(a) \, - \, X_2(a) X_2''(a)
   \;\; = \;\; 0
   \enspace .
\end{equation}
As $X_2(a) = 0$ implies $a = R_{\min}$ the energy level set not only
touches $P_{\mu \ell} \cap \{ Y = 0 \}$ at the point $(R_{\min},0)$
in this case, but furthermore the latter is singular, since for
non-singular points we have $X_2'(R_{\min}) = \infty$.
Dynamically this corresponds to a Hamiltonian Hopf bifurcation.
Finally
\begin{displaymath}
   F'''(R) \;\; = \;\; 2 X_1(R) X_1'''(R) \; + \; 6 X_1'(R) X_1''(R)
   \; - \; 2 X_2(R) X_2'''(R) \; - \; 6 X_2'(R) X_2''(R)
\end{displaymath}
and for $a = R_{\min}$ we get
\begin{displaymath}
   F'''(R_{\min}) \;\; = \;\; 6 X_1'(R_{\min}) X_1''(R_{\min})
   \; - \; 6 X_2'(R_{\min}) X_2''(R_{\min}) \enspace .
\end{displaymath}
If the Hamiltonian Hopf bifurcation takes place by a tangency at
the upper side, then we find that $F'''(R_{\min}) > 0$ implies
$X_1'' > X_2''$ so the level set~$\cH_{\lambda}^{-1}(h)$ touches
from outside and we have a supercritical Hamiltonian Hopf
bifurcation.
If the tangency is with the lower side, then $X_1'' < -X_2''$ so
the level set~$\cH_{\lambda}^{-1}(h)$ touches from outside and
we again have a supercritical Hamiltonian Hopf bifurcation.
Correspondingly, when $F'''(R_{\min}) < 0$ the level
set~$\cH_{\lambda}^{-1}(h)$ touches from inside and the
Hamiltonian Hopf bifurcation is subcritical.
In between, when $F'''(R_{\min}) = 0$ (while
$F^{(4)}(R_{\min}) \ne 0$) we have a degenerate Hamiltonian Hopf
bifurcation (for more details on this see \sref{degenerateHHb}
below).

If $X_1(a) = \pm X_2(a) \ne 0$ then, as we have earlier seen, we also
have $X_1'(a) = \pm X_2'(a)$, therefore from~\eqref{F''(a)=0} we
find $X_1''(a) = \pm X_2''(a)$.
These conditions determine an inflection point between the two curves
and therefore a centre-saddle bifurcation provided that
$X_1'''(a) \ne \pm X_2'''(a)$.
If the last non-degeneracy condition is not satisfied then
we have a cusp bifurcation.
This corresponds to $F'''(a) = 0$ (in addition to
$F(a) = F'(a) = F''(a) = 0$) and $F^{(4)} = 6 \kappa^2 \ne 0$.
\end{proof}

\noindent
Given lemma~\ref{lemma bd char} we can compute the bifurcation diagram
of the system by finding the triple roots of~$F$ that lie in
$[R_{\min},\infty[$.

\begin{sidewaystable}[htbp]
\small
\begin{tabularx}{\textwidth}{llXX}
\toprule
Type \quad & Parametrisation of $(\lambda,\mu,\ell)$ & Parameter ranges & Remarks \\
\midrule
\multicolumn{4}{l}{\textit{Centre-Saddle}} \\
\midrule
\CS{1} &
$(\lambda, -\mu_-(a, \lambda), \ell_-(a, \lambda))$ &
$\lambda < \tfrac12$ and $0 < a < 1-\lambda-\sqrt{1-2\lambda}$
or $\tfrac12 < \lambda < 1$, $0 < a < 1-\lambda$ &
\quad \newline Extended by continuity to $\lambda=\frac{1}{2}$ \\
\CS{2} &
$(\lambda, \mu_-(a, \lambda), \ell_-(a, \lambda))$ & same as for \CS{1} & \\
\CS{3} &
$(\lambda, \pm \mu_+(a, \lambda), \ell_+(a, \lambda))$ &
$\lambda < \tfrac12$,  $a_0(\lambda) < a < 1-\lambda-\sqrt{1-2\lambda}$
\quad & \\
\CS{4} &
$(\lambda, \pm \mu_-(a, \lambda), \ell_-(a, \lambda))$ &
$\tfrac12 < \lambda < 1$ and $1-\lambda < a < a_0(\lambda)$ & \\
\midrule
\multicolumn{4}{l}{\textit{Cusp}} \\
\midrule
\Cusp{1} &
$(\lambda,-(\lambda-\sqrt{2\lambda-1}),1-\lambda-\sqrt{2\lambda-1})$ &
$\tfrac12 < \lambda < 1$ &
$\partial\CS{1} \cap \partial\CS{4}$ \\
\Cusp{2} &
$(\lambda,\lambda-\sqrt{2\lambda-1},1-\lambda-\sqrt{2\lambda-1})$ &
$\tfrac12 < \lambda < 1$ &
$\partial\CS{2} \cap \partial\CS{4}$ \\
\Cusp{3} & $(\tfrac12,\mu,\tfrac14+\mu^2)$ &
$\lambda \equiv \tfrac12$, $-\tfrac12 < \mu < \tfrac12$ &
$\partial\CS{3} \cap \partial\CS{4}$ \\
\midrule
\multicolumn{4}{l}{\textit{Hamiltonian Hopf (subcritical)}} \\
\midrule
\HHsub{1} &
$(\lambda,1-\lambda-\sqrt{1-2\lambda},1-\lambda-\sqrt{1-2\lambda})$ &
$\lambda < \tfrac12$ &
$\partial\CS{2} \cap \partial\CS{3}$ \\
\HHsub{2} &
$(\lambda,-(1-\lambda-\sqrt{1-2\lambda}),1-\lambda-\sqrt{1-2\lambda})$
\quad &
$\lambda < \tfrac12$ &
$\partial\CS{1} \cap \partial\CS{3}$ \\
\HHsub{3} &
$(\lambda,0,-\lambda^2)$ &
$\lambda < 1$ &
$\partial\CS{1} \cap \partial\CS{2}$ \\
\midrule
\multicolumn{4}{l}{\textit{Hamiltonian Hopf (supercritical)}} \\
\midrule
\HHsup{1} &
$(\lambda,1-\lambda+\sqrt{1-2\lambda},1-\lambda+\sqrt{1-2\lambda})$ &
$\lambda < \tfrac12$ &  \\
\HHsup{2} &
$(\lambda,-(1-\lambda+\sqrt{1-2\lambda}),1-\lambda+\sqrt{1-2\lambda})$ &
$\lambda < \tfrac12$ &  \\
\HHsup{3} &
$(\lambda,0,-\lambda^2)$ &
$\lambda > 1$ &  \\
\midrule
\multicolumn{4}{l}{\textit{Hamiltonian Hopf (degenerate)}} \\
\midrule
\HHdeg{1} &
$(\tfrac{1}{2}, \tfrac{1}{2}, \tfrac{1}{2})$ &
A single point &
$\partial \Cusp{2} \cap \partial \Cusp{3}$ \\
\HHdeg{2} &
$(\tfrac{1}{2}, -\tfrac{1}{2}, \tfrac{1}{2})$ &
A single point &
$\partial \Cusp{1} \cap \partial \Cusp{3}$ \\
\HHdeg{3}
& $(1, 0, -1)$ &
A single point &
$\partial \Cusp{1} \cap \partial \Cusp{2}$ \\
\bottomrule
\end{tabularx}
\caption{
  Bifurcations for the $1{:}1{:}{-}2$~resonant system in the case
  $\kappa=1$.
  The quantities $\mu_\pm=\mu_\pm(a,\lambda)$ and
  $\ell_\pm=\ell_\pm(a,\lambda)$ are given
  in \eqref{general mu and ell} in
  appendix~\ref{proof of bifurcation set}, while $a_0(\lambda)$ is
  the unique non-negative root of the cubic polynomial $g(a)$ given
  in~\eqref{g(a)}.
\label{tbl:bifurcations} }
\end{sidewaystable}

\begin{proposition}\label{description of bifurcation set}
  The bifurcation diagram for $\kappa = 1$ consists of the
  $2$--parameter families of centre-saddle bifurcations, the
  $1$--parameter families of (non-degenerate) Hamiltonian Hopf
  and cusp bifurcations and the three degenerate Hamiltonian
  Hopf bifurcations given in table~\ref{tbl:bifurcations}.
\end{proposition}

\begin{figure}[htbp]
  \includegraphics[width=0.32\textwidth]{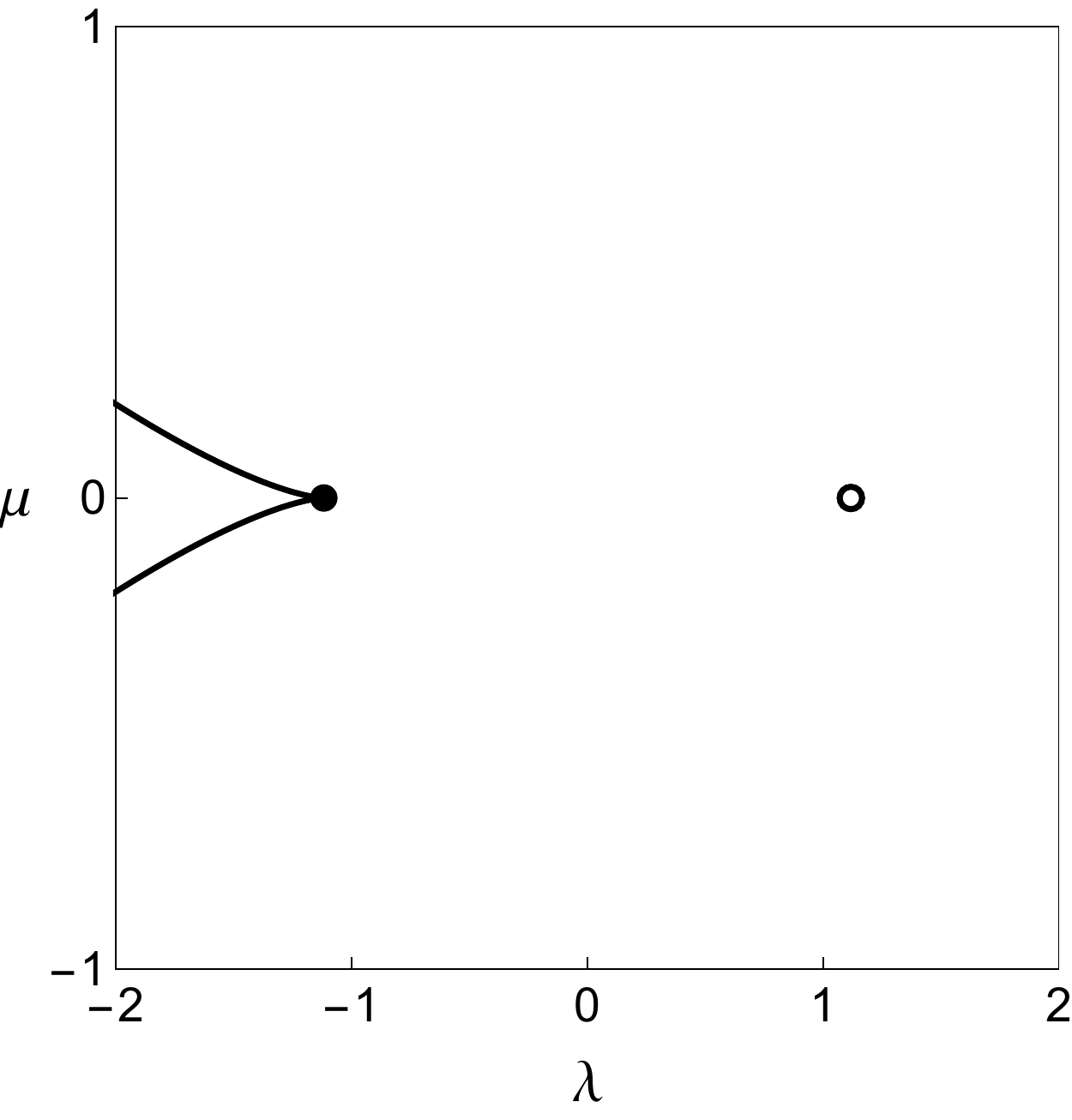}
  \hfill
  \includegraphics[width=0.32\textwidth]{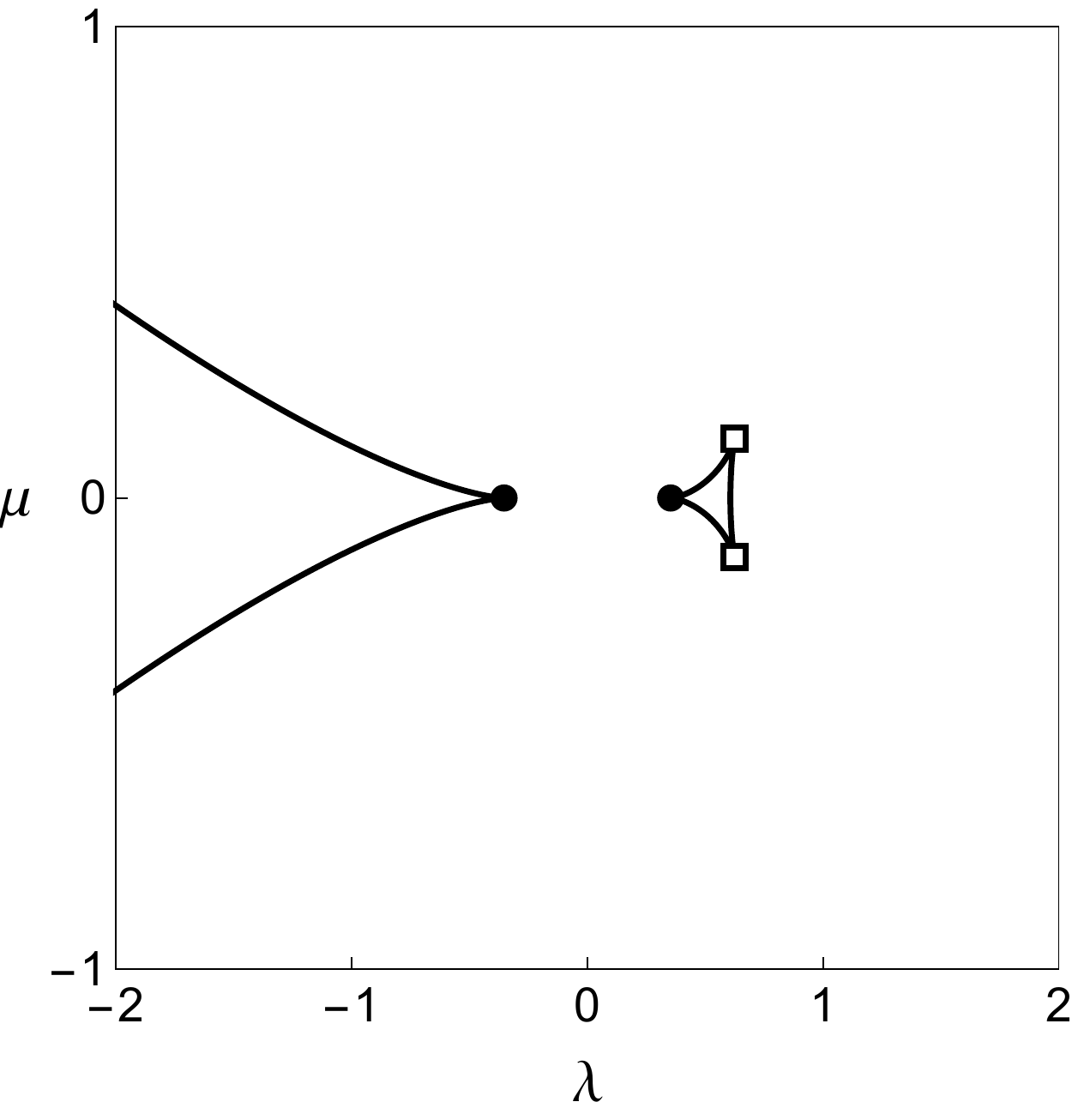}
  \hfill
  \includegraphics[width=0.32\textwidth]{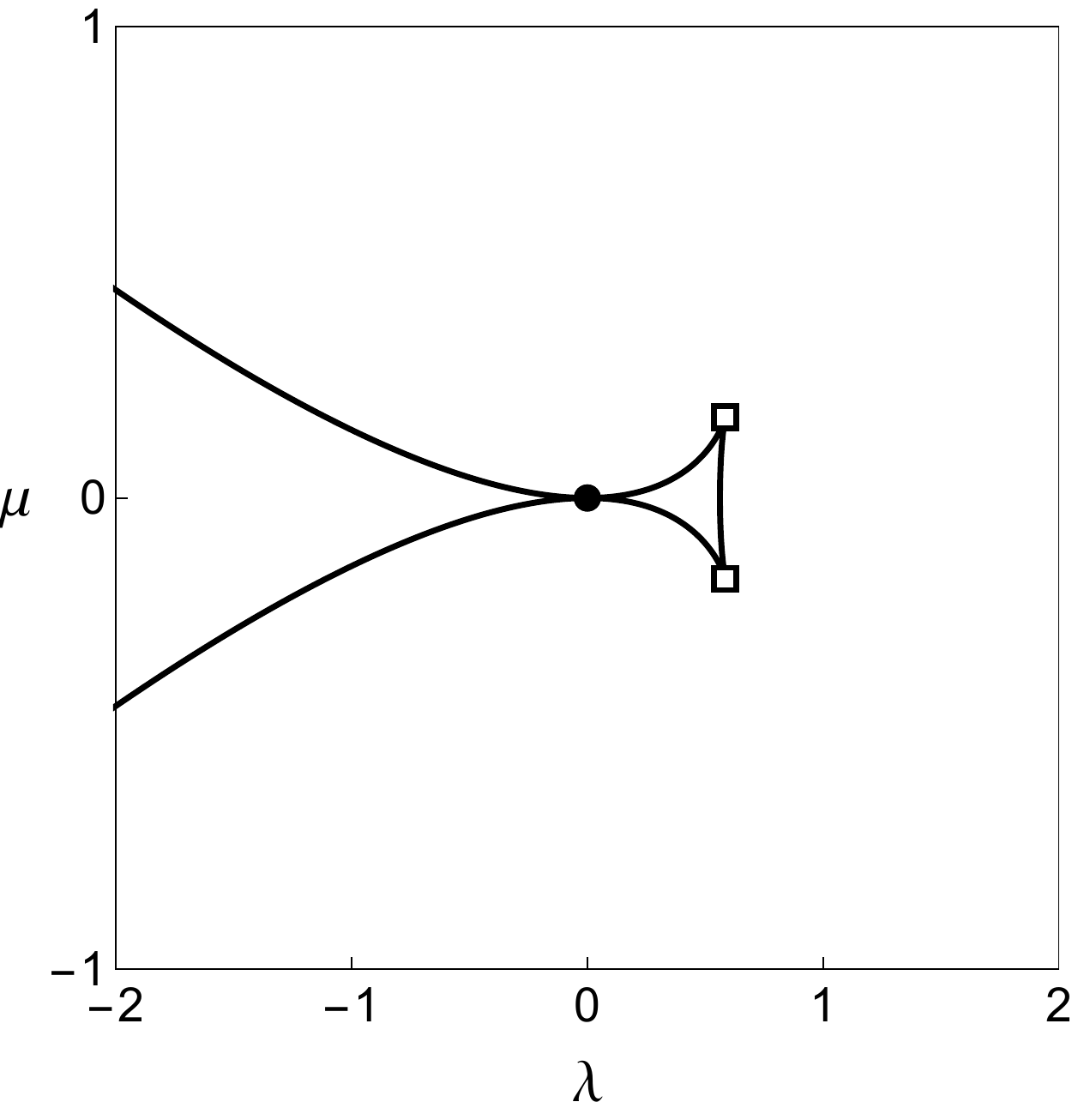}
  \\
  \includegraphics[width=0.32\textwidth]{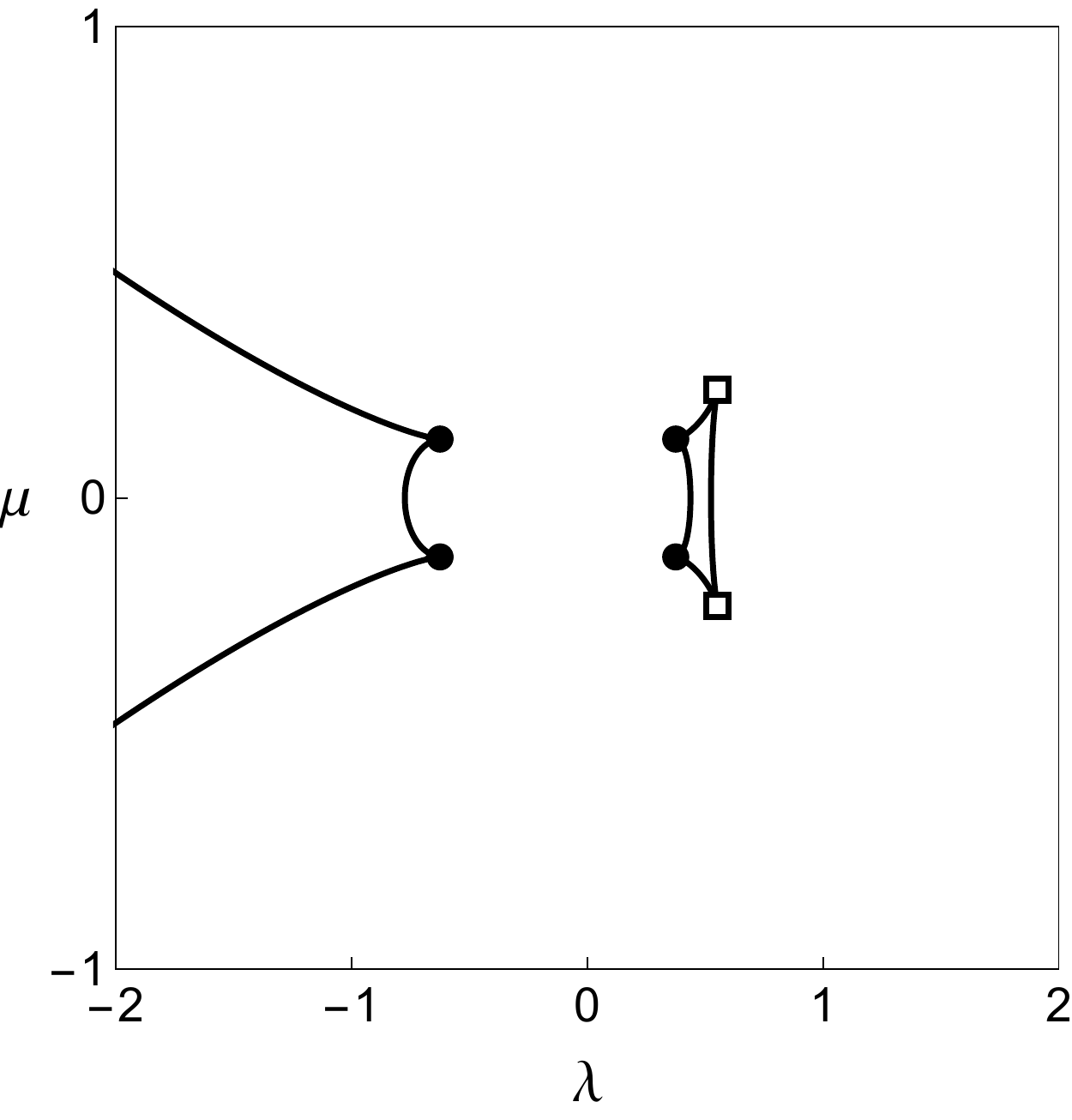}
  \hfill
  \includegraphics[width=0.32\textwidth]{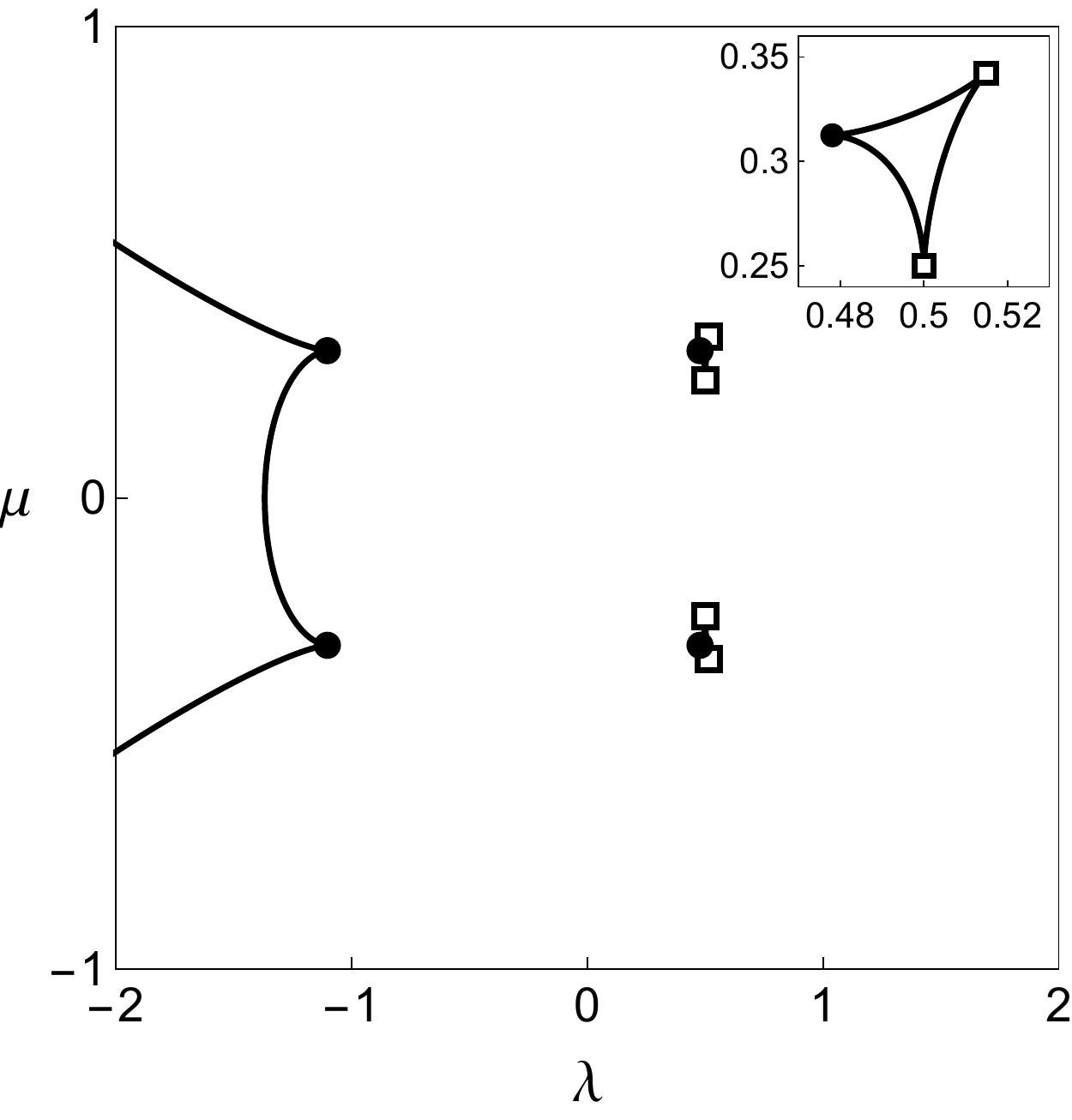}
  \hfill
  \includegraphics[width=0.32\textwidth]{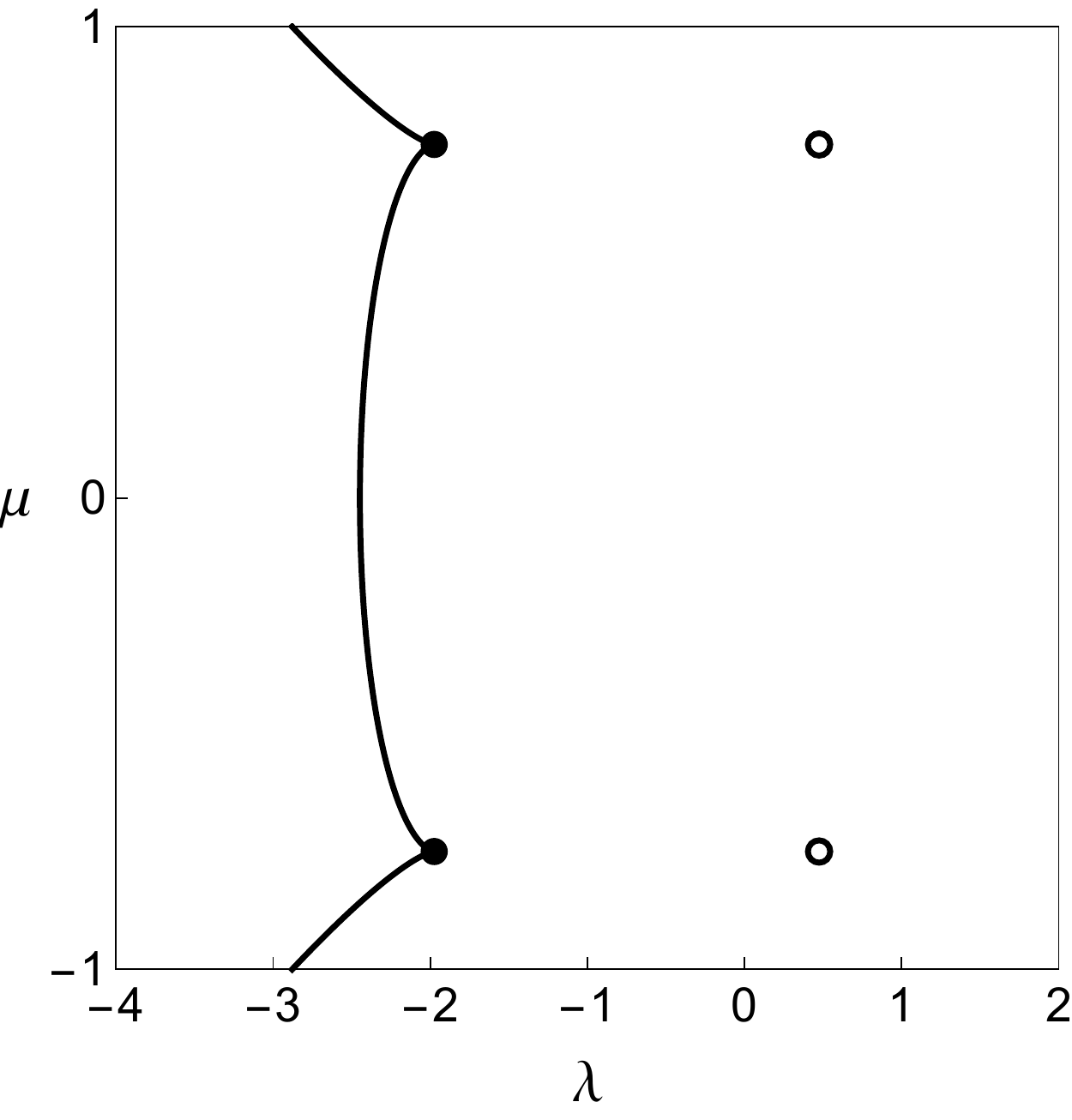}
  \caption{
    Bifurcation curves for $\kappa = 1$ and fixed values of~$\ell$
    (horizontal slices of the three-dimensional bifurcation diagram
    in \fref{fig:3dpictureofbifurcationset}).
    From top left:
    $\ell = -5/4$, $\ell = -1/8$, $\ell = 0$,
    $\ell = 1/8$, $\ell = 5/16$, and $\ell = 3/4$.
    The~$\bullet$ mark subcritical Hamiltonian Hopf bifurcations (for
    $\ell = 0$ the central equilibrium in $1{:}1{:}{-}2$~resonance),
    the~$\circ$ mark supercritical Hamiltonian Hopf bifurcations,
    and the~{\scriptsize $\square$} mark cusp bifurcations.
    In the panel for $\ell=5/16$ one of the smaller structures in
    the bifurcation diagram has been magnified in the inset.
  \label{fig:bif01} }
\end{figure}

\noindent
The proof of proposition~\ref{description of bifurcation set} is given
in appendix~\ref{proof of bifurcation set}, except for the degenerate
Hamiltonian Hopf bifurcation which we treat in \sref{degenerateHHb}
below.
The bifurcation set presented in table~\ref{tbl:bifurcations} includes
the centre-saddle and cusp bifurcations discussed
in~\sref{Bifurcationsofregularequilibria} and the Hamiltonian Hopf
bifurcations discussed in~\sref{HamiltonianHopfbifurcations}.
The bifurcation set in three-dimensional space
$(\lambda,\mu,\ell)$ is shown for $\kappa = 1$ in
figure~\ref{fig:3dpictureofbifurcationset}.
Successive horizontal sections of constant $\ell$ are shown
in figure~\ref{fig:bif01}.
Recall that for $\kappa \ne 0$ we can scale invariants and parameters to
obtain a system with $\kappa = 1$ and therefore the bifurcation diagram for
any $\kappa \ne 0$ can be obtained by inverting the scaling.

\subsection{The degenerate Hamiltonian Hopf bifurcation}
\label{degenerateHHb}

\noindent
Abbreviating
\begin{subequations}
\label{invariants}
\begin{align}
   T & \;\; = \;\; x_1 y_2 - x_2 y_1 \label{invariantsT}\\
   U & \;\; = \;\; \frac{x_1^2 + x_2^2}{2} \label{invariantsU}\\
   V & \;\; = \;\; \frac{y_1^2 + y_2^2}{2} \label{invariantsV}\\
   W & \;\; = \;\; x_1 y_1 + x_2 y_2 \label{invariantsW}
\end{align}
\end{subequations}
for a Hamiltonian system in two degrees of freedom, the standard
form of the (non-degenerate) Hamiltonian Hopf bifurcation reads as
\begin{equation}
\label{nondegHHb}
   H_{\nu} \;\; = \;\; T \; + \; U \; + \; \nu V \; + \; a V^2
   \enspace .
\end{equation}
The non-degeneracy condition is $a \neq 0$ and the sign of~$a$
distinguishes between the supercritical case $a > 0$ and the
subcritical case $a < 0$~; see~\cite{jcvdm85} for a proof.
In the degenerate case $a = 0$ terms of order higher than four
(in the original variables $x, y$) become important,
see~\cite{bri90, bri91, jcvdm96}.
Indeed, if the coefficient~$b$ of~$V^3$ is non-zero, then a
$C^{\infty}$--versal unfolding is given by
\begin{equation}
\label{degHHb}
   H_{\nu} \;\; = \;\; T \; + \; U \; + \; \nu_1 V \; + \;
   \frac{\nu_2}{2} V^2 \; + \; \nu_3 T V \; + \; b V^3
   \enspace .
\end{equation}
As expected, the coefficient of~$V^2$ has turned into the unfolding
parameter~$\nu_2$, but furthermore a third unfolding parameter~$\nu_3$
has emerged, the coefficient of the fourth order term~$TV$.
This $C^{\infty}$--modal parameter can be removed by passing to a
$C^0$--versal unfolding, again see~\cite{bri90, bri91, jcvdm96},
subject to $\nu_3 \notin \{ 0, \pm \sqrt{-b}, \pm \sqrt{3b} \}$.

For a general Hamiltonian with an equilibrium $0 \in \R^4$ in
$1{:}{-}1$~resonance, the standard forms \eqref{nondegHHb}
and~\eqref{degHHb} have to be achieved through normalization with
respect to~$T$ (and~$U$).
In our application to~\eqref{symmetricdetunedhamiltonian} the
Hamiltonian already is symmetric with respect to~$N$, which
amounts here to symmetry with respect to~$T$.
This alows to reduce to one degree of freedom and all other questions
can be answered using the reduced system;
see~\cite{HM02, HM05, hh07} for criteria concerning the
supercritical and subcritical Hamiltonian Hopf bifurcations.
To formulate similar criteria for the degenerate Hamiltonian Hopf
bifurcation we therefore reduce~\eqref{degHHb} to one degree of
freedom, using the invariants~\eqref{invariantsU}--\eqref{invariantsW}
of the $S^1$--action generated by~\eqref{invariantsT} as variables
subject to the syzygy
\begin{displaymath}
   2 U V \;\; = \;\; \frac{W^2 + T^2}{2}
\end{displaymath}
and the inequalities $U \geq 0$, $V \geq 0$.
Fixing $T = \theta$ this yields one sheet~$\cP_{\theta}$ of a
$2$--sheeted hyperboloid, a cone if $\theta = 0$, and the orbits
of the flow defined by~\eqref{degHHb} are given by the intersections
(within $\R^3 = \{ U, V, W \}$) with the parabolic cylinders
$\{ H_\nu = h \}$, i.e.\ the latter sets are flat in the $W$--direction.
Such parabolic cylinders can touch the reduced phase
space~$\cP_{\theta}$ only within the plane $\{ W = 0 \}$,
whence equilibria occur where the curves
\begin{subequations}
\label{withinW}
\begin{align}
   U & \;\; = \;\; 0  \label{phasespacewithinW}\\
   U & \;\; = \;\; h \; - \; \theta \; - \; \nu_1 V \; - \;
   \frac{\nu_2}{2} V^2 \; - \; \nu_3 T V \; - \; b V^3
   \label{energywithinW}
\end{align}
\end{subequations}
have the same derivative (adjusting the height~$h$
of~\eqref{energywithinW} appropriately).
Note that this formulation lends itself for a straightforward
generalization to $S^1$--symmetric Hamiltonian Hopf bifurcations
where none of the corresponding curves~\eqref{withinW} is expected
to be the co-ordinate axis, again compare
with~\cite{HM02, HM05, hh07}.

The point $V = 0$ corresponds to the singular tip
$(U, V, W) = (0, 0, 0)$ of the cone, so it is always an equilibrium.
It is where the Hamiltonian Hopf bifurcation takes place that
\eqref{phasespacewithinW} and~\eqref{energywithinW} have the same
derivative in $V = 0$, i.e.\ where the corresponding equilibrium
in~$\R^4$ has normal frequencies in $1{:}{-}1$~resonance.
This means that the difference function between \eqref{energywithinW}
and~\eqref{phasespacewithinW} --- which for the standard
form~\eqref{degHHb} is simply~\eqref{energywithinW} --- has a double
root at the singular point of the reduced phase space.
Correspondingly, the first derivative of the difference function at
$V = 0$ yields the first unfolding parameter~$\nu_1$.
The second derivative then yields the second unfoding
parameter~$\nu_2$, vanishing as well at the degenerate Hamiltonian
Hopf bifurcation and otherwise distinguishing between the
supercritical case (where the energy level set touches the cone from
the outside) and the subcritical case (where $\{ H_\nu = 0 \}$ touches
$\cP_0$ from the inside).
Note that the first derivative in fact yields $\nu_1 + \nu_3 \theta$
from which we not only obtain the unfolding parameter~$\nu_1$ by
taking $\theta = 0$, but also the modal parameter~$\nu_3$ by taking
the derivative with respect to~$\theta$.
Finally, the third derivative equals~$b$ and thus should be non-zero
to have a `non-degenerate' degenerate Hamiltonian Hopf bifurcation.

To check whether (and where) a degenerate Hamiltonian Hopf bifurcation
takes place in the normal form~\eqref{symmetricdetunedhamiltonian}
of the $1{:}1{:}{-}2$~resonance we merely have to take derivatives
of (the square root of) the curve~\eqref{phasespacesection11-2} and of
the curve~\eqref{parabola11-2} at $R = R_{\min} = \ell = |\mu| > 0$
and at $R = R_{\min} = \mu = 0$, $\ell < 0$.
As already noted in~\eqref{parabolaHHb}, the value $h$
of~$\cH_{\lambda}$ has to be adjusted to
$h^* = \lambda \ell + \frac{1}{2} \kappa \ell^2$ or $h^* = 0$,
respectively.
Let us first concentrate on $R_{\min} = 0$.
Then the cubic curve~\eqref{phasespacesection11-2} becomes
\begin{equation}
\label{ppsmuzero}
   X_2(R) \;\; = \;\; \pm R \sqrt{R - \ell}
\end{equation}
with derivatives $X_2^{\prime}(0) = \pm \sqrt{-\ell}$,
$X_2^{\prime \prime}(0) = \pm \sqrt{-\ell}^{-1}$ and
$X_2^{\prime \prime \prime}(0) = \mp \frac{3}{4} \sqrt{-\ell}^{-3}$,
while~\eqref{parabola11-2} has derivatives
$X_1^{\prime}(0) = - \lambda$, $X_1^{\prime \prime}(0) = - \kappa$
and $X_1^{\prime \prime \prime}(0) = 0$.
Equating the first derivatives yields
\begin{equation}
\label{firstderivativemuzero}
   - \lambda \;\; = \;\; \pm \sqrt{-\ell}
   \enspace ,
\end{equation}
whence the parabola~\eqref{parabola11-2} touches the
cubic~\eqref{phasespacesection11-2} at the `upper' side for
$\lambda < 0$ and at the `lower' side for $\lambda > 0$.
Equating the second derivatives yields
\begin{equation}
\label{secondderivativemuzero}
   - \kappa \;\; = \;\; \frac{\pm 1}{\sqrt{-\ell}}
   \enspace ,
\end{equation}
whence the Hamiltonian Hopf bifurcation is degenerate for
\begin{displaymath}
   (\lambda, \mu, \ell) \;\; = \;\;
   (\frac{1}{\kappa}, 0, \frac{-1}{\kappa^2})
   \enspace .
\end{displaymath}
For definiteness we restrict to a positive constant $\kappa > 0$, so
the degenerate Hamiltonian Hopf bifurcation takes place at the `lower'
side of the cubic~\eqref{phasespacesection11-2}.
We infer from~\eqref{firstderivativemuzero} that
$\frac{1}{\kappa} - \lambda$ can be used to unfold the Krein collision
and from~\eqref{linearparameterdependence} that in fact the detuning
$\frac{1}{\kappa} - \delta$ plays the r\^ole of the unfolding
parameter~$\nu_1$, while $-\lambda_1$ plays the r\^ole of the modal
parameter~$\nu_3$.
From~\eqref{secondderivativemuzero} we conclude that the second
derivative at $R = R_{\min} = 0$ of the difference function between
\eqref{ppsmuzero} and~\eqref{parabola11-2} reads as
\begin{displaymath}
   \frac{1}{\sqrt{-\ell}} \; - \; \kappa \;\; = \;\;
   \frac{\kappa^3}{2} (\ell + \frac{1}{\kappa^2}) \; + \;
   \frac{3 \kappa^5}{8} (\ell + \frac{1}{\kappa^2})^2 \;\; + \;\; \cdots
\end{displaymath}
whence the r\^ole of~$\nu_2$ is played by $\frac{1}{\kappa^2} + \ell$.
The third derivative $\frac{3}{4} \kappa^3$ is positive for
$\kappa > 0$ and in particular non-zero.
Note that the latter also follows from
$F^{(4)} \equiv 6 \kappa^2 \neq 0$,
see equation \eqref{F(R)}.

We now check the conditions at the other two possible values
of~$R_{\min}$.
For both(!) $R_{\min} = \ell = \pm \mu > 0$ the
cubic~\eqref{phasespacesection11-2} becomes
\begin{equation}
\label{ppsellmu}
   X_2(R) \;\; = \;\; \pm (R - \ell) \sqrt{R + \ell}
\end{equation}
with derivatives $X_2^{\prime}(\ell) = \pm \sqrt{2 \ell}$,
$X_2^{\prime \prime}(\ell) = \pm \sqrt{2 \ell}^{-1}$ and
$X_2^{\prime \prime \prime}(\ell) =
 \mp \frac{3}{4} \sqrt{2 \ell}^{-3}$,
while
$X_1^{\prime}(\ell) = - \lambda - \kappa \ell$,
$X_1^{\prime \prime}(\ell) = - \kappa$
and $X_1^{\prime \prime \prime}(\ell) = 0$.
Equating the first derivatives yields
\begin{equation}
\label{firstderivativeellmu}
   - \lambda \; - \; \kappa \ell \;\; = \;\; \pm \sqrt{2 \ell}
   \enspace ,
\end{equation}
whence the parabola~\eqref{parabola11-2} touches the
cubic~\eqref{phasespacesection11-2} at the `upper' side for
$\lambda < - \kappa \ell$ and at the `lower' side for
$\lambda > - \kappa \ell$.
Equating the second derivatives yields
\begin{equation}
\label{secondderivativeellmu}
   - \kappa \;\; = \;\; \frac{\pm 1}{\sqrt{2 \ell}}
   \enspace ,
\end{equation}
whence the Hamiltonian Hopf bifurcation is degenerate for
\begin{displaymath}
   (\lambda, \mu, \ell) \;\; = \;\;
   (\frac{1}{2 \kappa}, \frac{\pm 1}{2 \kappa^2}, \frac{1}{2 \kappa^2})
   \enspace .
\end{displaymath}
Keeping $\kappa > 0$, the degenerate Hamiltonian Hopf bifurcation
again takes place at the `lower' side of the
cubic~\eqref{phasespacesection11-2}.
We infer from~\eqref{firstderivativeellmu} that
$\frac{1}{2 \kappa} - \lambda$ can be used to unfold the Krein
collision where (again) the detuning $\frac{1}{2 \kappa} - \delta$
plays the r\^ole of the unfolding parameter~$\nu_1$, while
$-\lambda_1$ plays the r\^ole of the modal parameter~$\nu_3$.
From~\eqref{secondderivativeellmu} we conclude that the second
derivative at $R = R_{\min} = 0$ of the difference function between
\eqref{ppsellmu} and~\eqref{parabola11-2} reads as
\begin{displaymath}
   \frac{1}{\sqrt{2 \ell}} \; - \; \kappa \;\; = \;\;
   - \frac{\kappa^3}{2} (\ell - \frac{1}{2 \kappa^2}) \; + \;
   \frac{3 \kappa^5}{8} (\ell - \frac{1}{2 \kappa^2})^2 \;\; + \;\; \cdots
\end{displaymath}
whence the r\^ole of~$\nu_2$ can be played by any linear combination
of $\frac{1}{2 \kappa^2} - \ell$ and $\frac{1}{2 \kappa^2} \mp \mu$.
The third derivative $\frac{3}{4} \kappa^3$ is non-zero; this also
follows from $F^{(4)} \equiv 6 \kappa^2 \neq 0$.

Note that for all three degenerate Hamiltonian Hopf bifurcations the
Hamiltonian~\eqref{symmetricdetunedhamiltonian} does provide a full
$C^{\infty}$--versal unfolding, with modal parameter
$\nu_3 = -\lambda_1$.
However, putting $\lambda_1 = 0$ results in topological
co-dimension~$3$ instead of~$2$, see again~\cite{jcvdm96}.

\subsection{Bifurcation diagram for $\kappa = 0$}

\noindent
We briefly discuss here the case $\kappa = 0$.
Note that this case is degenerate:
higher order terms in $H_{N,L}$ may change the results obtained here.
Moreover, fibres may contain non-compact connected components.

\begin{table}[htbp]
  \centering
  \begin{tabularx}{1.0\linewidth}{lll}
    \toprule
    Family & Parametrisation $(\lambda,\mu,\ell)$ & Parameter ranges \\
    \midrule
    \multicolumn{3}{l}{\textit{Centre-Saddle}} \\
    \midrule
    $\CS{1}^{\kappa=0}$ &  $(\lambda,-\mu_-,3a-\lambda^2)$ & $\lambda \in \mathbb R_*$, $0 < a < \lambda^2/2$ \\ 
    $\CS{2}^{\kappa=0}$ &  $(\lambda,\mu_-,3a-\lambda^2)$ & $\lambda \in \mathbb R_*$, $0 < a < \lambda^2/2$ \\ 
    $\CS{3}^{\kappa=0}$ &  $(\lambda,\pm\mu_+,3a-\lambda^2)$ & $\lambda \in \mathbb R_*$, $4\lambda^2/9 < a < \lambda^2/2$ \\ 
    \midrule
    \multicolumn{3}{l}{\textit{Hamiltonian Hopf (subcritical)}} \\
    \midrule
    $(\HHsub{1})^{\kappa=0}$ & $(\lambda, \tfrac12 \lambda^2, \tfrac12 \lambda^2)$ & $\lambda \in \mathbb R_*$ \\
    $(\HHsub{2})^{\kappa=0}$ & $(\lambda, -\tfrac12 \lambda^2, \tfrac12 \lambda^2)$ & $\lambda \in \mathbb R_*$ \\
    $(\HHsub{3})^{\kappa=0}$ & $(\lambda, 0, -\lambda^2)$ & $\lambda \in \mathbb R_*$ \\
    \bottomrule
  \end{tabularx}
  \caption{
    Bifurcation families for $\kappa=0$.
    Here, the quantities $\mu_\pm$ are given by
    $\mu_\pm^2 = -3 a^2+6 a \lambda ^2
     - 2 \lambda^4 \mp 2 |\lambda| (\lambda^2-2a)^{3/2}$
    and can be obtained from equation~\eqref{general mu and ell}
    by setting $\kappa=0$.
  \label{tbl:bif-kappa-0} }
\end{table}

\begin{figure}[htbp]
  \centering
  \includegraphics[width=8cm]{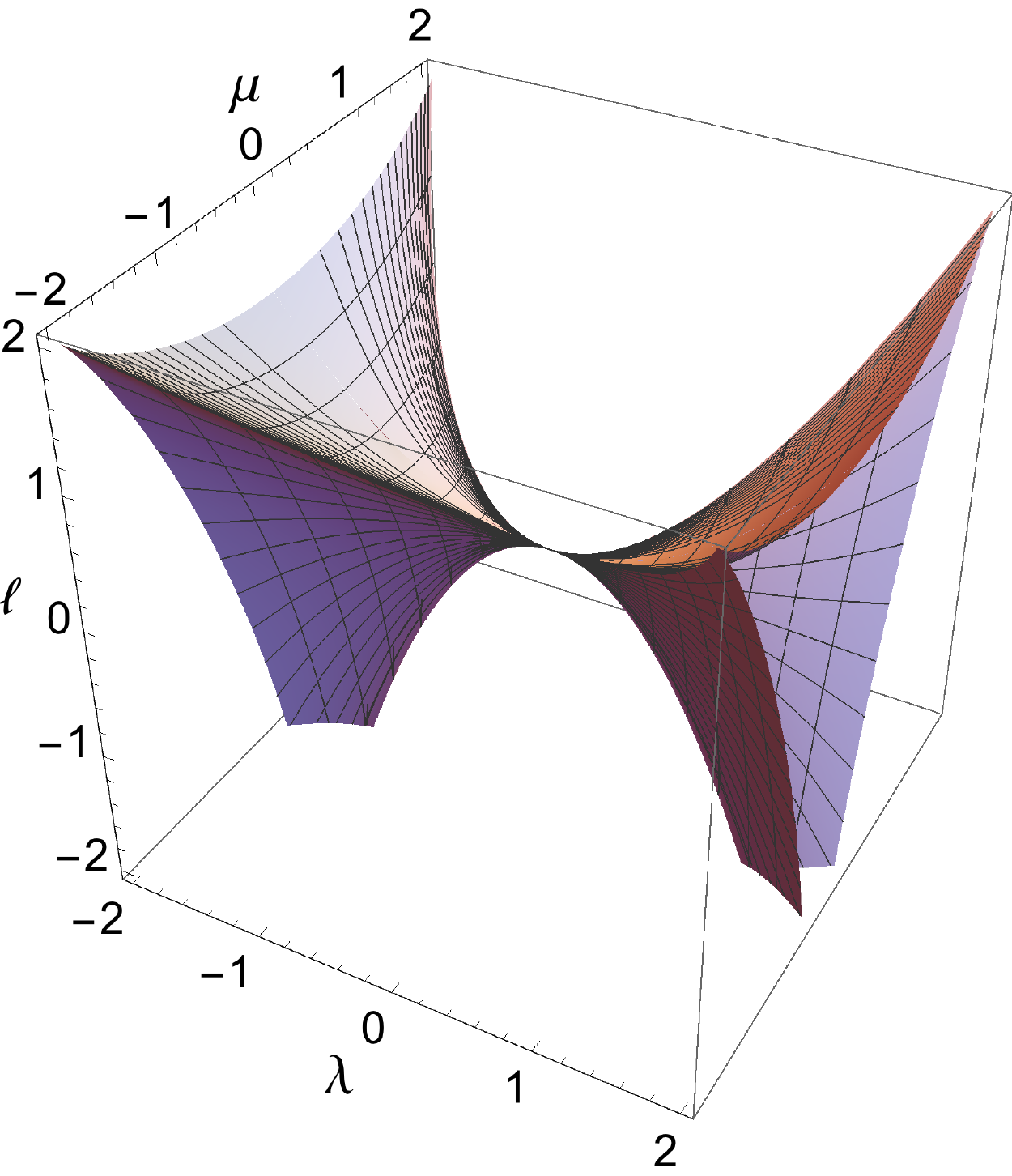}
  \caption{
    Bifurcation diagram for $\kappa=0$ in $(\lambda,\mu,\ell)$ space.
    One may think of this as obtained from
    \fref{fig:3dpictureofbifurcationset} by sending the triangle of
    cusp and degenerate Hamiltonian Hopf bifurcations from
    $\lambda = \frac{1}{2}$ to $\lambda = \infty$ (and the
    supercritical Hamiltonian Hopf bifurcations beyond).
  \label{fig:bif-kappa-0-3d} }
\end{figure}

\begin{figure}[htbp]
  \includegraphics[width=0.32\textwidth]{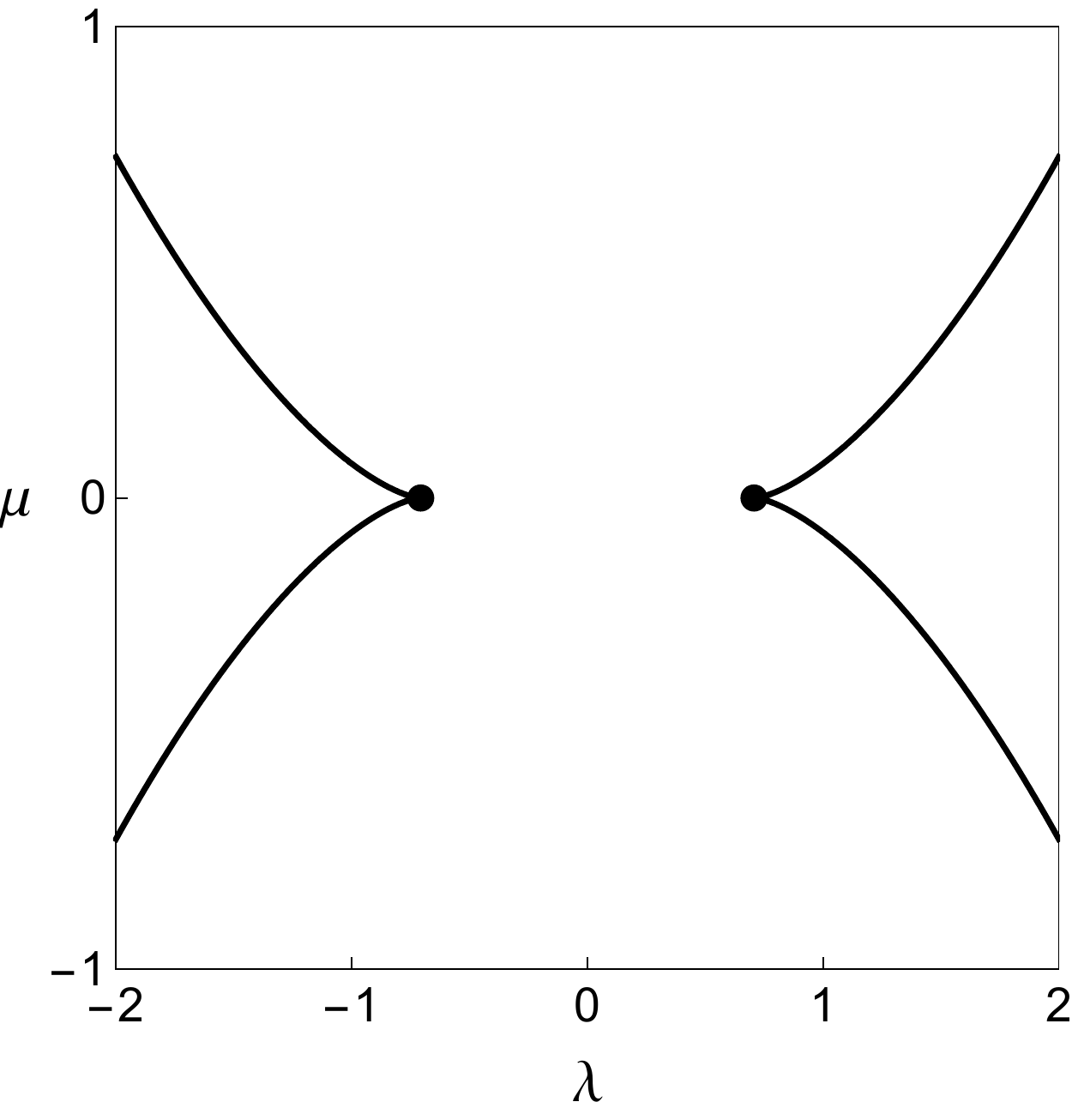}
  \hfill
  \includegraphics[width=0.32\textwidth]{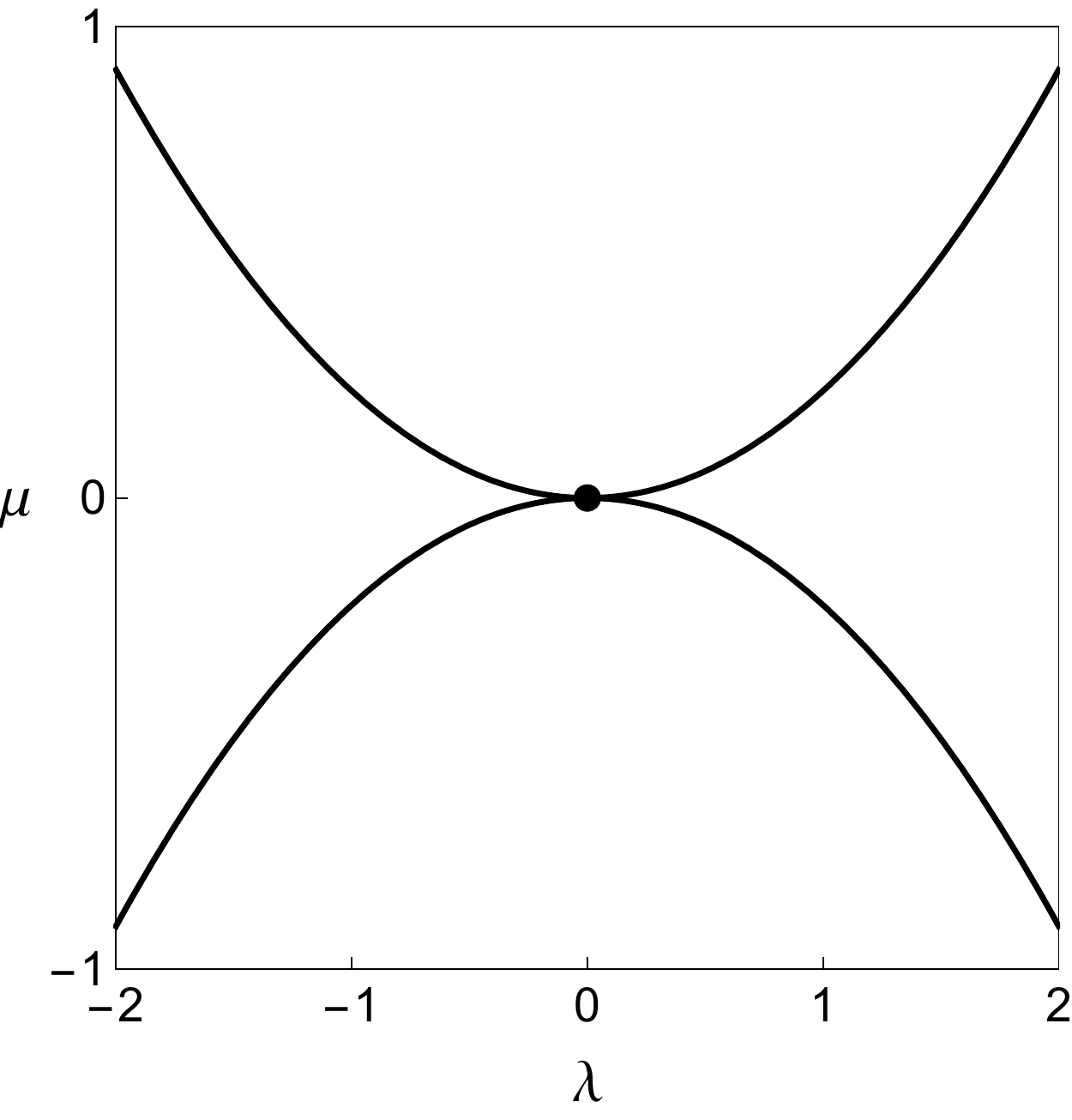}
  \hfill
  \includegraphics[width=0.32\textwidth]{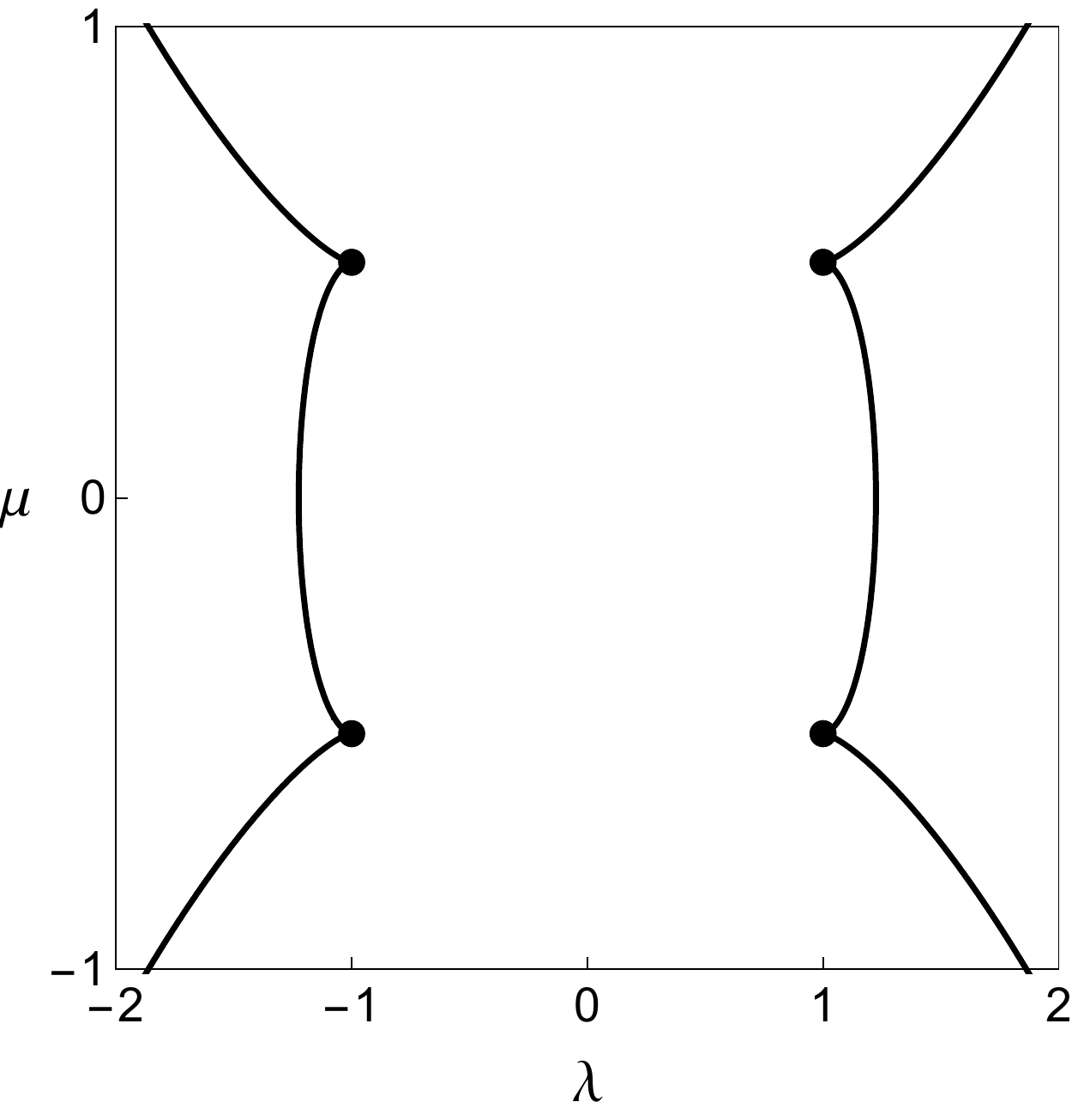}
  \caption{
    Bifurcation curves for $\kappa=0$ and fixed values
    of $\ell$ in $(\lambda,\mu)$--space.
    From left to right: $\ell = -\frac{1}{2}$, $\ell = 0$ and
    $\ell = \frac{1}{2}$.
    The~$\bullet$ mark subcritical Hamiltonian Hopf bifurcations
    (except for the one at $(\lambda, \mu, \ell) = (0, 0, 0)$ 
    corresponding to the central equilibrium in
    $1{:}1{:}{-}2$~resonance).
  \label{fig:bif0} }
\end{figure}

An analysis along the lines of \sref{bifurcation diagram kappa != 0}
shows that there are no cusp bifurcations or supercritical
Hamiltonian Hopf bifurcations.
We find three families of centre-saddle bifurcations
$\CS{k}^{\kappa=0}$, $k=1,2,3$, whose pairwise common boundaries
correspond to three families of subcritical Hamiltonian Hopf
bifurcations $(\HHsub{k})^{\kappa=0}$, $k=1,2,3$.
Their parametrisations are given in table~\ref{tbl:bif-kappa-0} and
they are depicted in figures \ref{fig:bif-kappa-0-3d}
and~\ref{fig:bif0}.
Note that the bifurcation diagram for $\kappa = 0$ can be obtained
from the one for $\kappa \ne 0$ by considering the limit of each
bifurcation family for $\kappa \to 0$ and ignoring the families that
exist only for $\lambda > \Frac{1}{2\kappa}$.

\section{Critical values of the energy-momentum mapping}
\label{Critical values}

\noindent
We now focus on the set of critical values~$\CV$ of the
\emph{energy-momentum mapping}
\begin{displaymath}
   \EM  \; : \;\; \R^6 \;\; \longrightarrow \;\; \R^3
\end{displaymath}
defined in~\eqref{energymomentummapping}, i.e.\ with components
$N$, $L$ and~$H$.
Note that the diffeomorphism
\begin{displaymath}
   (\mu, \ell, h) \;\; \mapsto \;\; (\mu, \frac{\ell + \mu}{2}, h)
\end{displaymath}
of~$\R^3$ maps $\CV$ to the set of critical values of
\begin{displaymath}
   (N, J, H)  \; : \;\; \R^6 \;\; \longrightarrow \;\; \R^3
   \enspace ;
\end{displaymath}
in fact the whole ramified torus bundle defined
by~\eqref{energymomentummapping} turns into the ramified torus
bundle defined by $(N, J, H)$.
It is only the way that the toral fibres are orbits of
\eqref{T2-action-1} and~\eqref{T2actionPhi} that is affected
by the global isotropy
$\Z_2 = \{ (0, 0), (\frac{1}{2}, \frac{1}{2}) \}$ --- the
$\T^2$--action~\eqref{T2-action-1} runs through regular fibres
twice.

The parameter $\lambda$ depends on the values $N = \mu$ and
$L = \ell$ through~\eqref{linearparameterdependence}, which
defines a diffeomorphism
\begin{displaymath}
   (\delta, \mu, \ell) \;\; \mapsto \;\; (\lambda, \mu, \ell)
\end{displaymath}
of~$\R^3$ that relates the bifurcation diagram detailed in
proposition~\ref{description of bifurcation set} to the
bifurcation diagram in terms of $(\delta, \mu, \ell)$.
While the latter is equal to the former when
$\lambda_1 = \lambda_2 = 0$ --- the situation we concentrate
on in section~\ref{sets of critical values} below --- it is
instructive to compare the intersections $\lambda = \delta$
in figures \ref{fig:cvlambda=0} and~\ref{cv lambda<0} below
with \fref{fig:introductorybifurcationdiagrams} where the
detuning~$\delta$ is fixed, but
$(\lambda_1, \lambda_2) \neq (0, 0)$
whence the vertical planes $\lambda = \delta$ get `tilted'
to the planes
$\lambda = \delta + \lambda_1 \mu + \lambda_2 \ell$
with constant~$\delta$.

Similarly, we can replace the normal form $H = H^{\delta}_{N, L}$
defined in~\eqref{symmetricdetunedhamiltonian} by the (simplified)
reduced Hamiltonian function $\cH_{\lambda}$
of~\eqref{reducedhamiltonian} expressed in co-ordinates $q_i, p_i$,
$i=1,2,3$ as the diffeomorphism
\begin{displaymath}
   \left( \mu, \ell, h \right) \;\; \mapsto \;\;
   \left( \mu, \ell, {\textstyle h - \alpha \ell - \beta \mu -
   \frac{1}{2} \gamma_1 \mu^2 - \gamma_2 \mu \ell -
   \frac{1}{2} \gamma_3 \ell^2} \right)
\end{displaymath}
of~$\R^3$ maps critical values to critical values.
In the sequel we therefore work with the energy-momentum mapping
\begin{equation}
\label{reducedenergymomentummapping}
   \EM  \;\; = \;\; (N, L, \cH_{\lambda}) \enspace ,
\end{equation}
the values of which we keep denoting by $(\mu, \ell, h) \in \R^3$.
We consider only the case $\kappa \ne 0$ and the
scaling~\eqref{kappa=1 scaling} allows us to restrict to
\begin{equation}
\label{finallydoscalekappa}
   \kappa \;\; = \;\; 1 \enspace ,
\end{equation}
the value which was also used for table~\ref{tbl:bifurcations} and
figures \ref{fig:3dpictureofbifurcationset} and~\ref{fig:bif01}.

\subsection{Amended bifurcation diagram}
\label{amended bd}

\noindent
The critical values of~\eqref{reducedenergymomentummapping} correspond
to values $(\mu, \ell, h)$ of the internal parameters and energy for
which the level set $\cH_{\lambda}^{-1}(h)$ of the reduced Hamiltonian
either has a tangency with the reduced space $P_{\mu \ell}$ or goes
through the singular point of~$\cP_{\mu \ell}$.
Recall from proposition~\ref{prop:singtype} that the reduced phase
space~$\cP_{\mu \ell}$ is singular at $(R_{\min}, 0, 0)$ for
$\ell = |\mu|$ or $\mu = 0$, $\ell \le 0$, where
$R_{\min} = \max(|\mu|, \ell)$.
From~\eqref{reducedhamiltonian} with $\kappa=1$ we find that the value
$h_c$ of the energy for which $\cH_{\lambda}^{-1}(h)$ goes through the
singular point is
\begin{displaymath}
   h_c \;\; = \;\; \lambda R_{\min} \; + \; \Frac{1}{2} R_{\min}^2
   \;\; = \;\;
   \begin{cases}
      \lambda \ell + \frac{1}{2} \ell^2 &
      \text{for $\mu = \pm \ell$, $\ell \ge 0$} \\
      0 & \text{for $\mu = 0$, $\ell \le 0$.}
   \end{cases}
\end{displaymath}
This gives three curves of critical values of $\EM$, each one
parametrised by~$\ell$.
Mirroring the notation for the normal modes in
table~\ref{tab:isotropy} we denote by $\CV_{23}$ the curve
corresponding to $\mu = \ell$, $\ell > 0$, parametrising the
normal $1$--mode, by $\CV_{13}$ the curve corresponding to
$\mu = -\ell$, $\ell > 0$, parametrising the
normal $2$--mode and by $\CV_{12}$ the curve corresponding
to $\mu = 0$, $\ell < 0$, parametrising the
normal $3$--mode, noting that $\EM(C_{ij}) = \CV_{ij}$ for
$ij=12,13,23$.
Depending on the topology of
$\cH_{\lambda}^{-1}(h_c) \cap \cP_{\mu \ell}$
(parts of) these curves could be be attached to a surface of
critical values or be transversally isolated, in the sense that
a neighbourhood of a point of such a curve contains only
critical values from the same curve.
We denote by $\CV_{ij}^0$ the subset of the curve~$\CV_{ij}$
where it is transversally isolated and we refer to such subset
as \emph{thread}.
Moreover, we denote by $\CV_{ij}^+ \subseteq \CV_{ij}$
the subset where $h_c > h_{\min}$; here $h_{\min}$ is the minimal
value of $\cH_\lambda$ for given $(\mu,\ell)$.
The minimum $h_{\min}$ of~$\cH_{\lambda}$ on~$\cP_{\mu \ell}$
depends continuously on $\mu$ and~$\ell$ whence necessarily
$\CV_{ij}^0 \subseteq \CV_{ij}^+$.

The topology of $\cH_{\lambda}^{-1}(h_c) \cap \cP_{\mu \ell}$, and
subsequently of the fibre $\EM^{-1}(\mu,\ell,h_c)$ depends
on the slope of $\cH_{\lambda}^{-1}(h_c)$ relative to the
slope of $\cP_{\mu \ell}$ at the singular point.
Recall from the discussion in \sref{HamiltonianHopfbifurcations} that
for each value of $(\mu, \ell)$ such that the reduced space
$\cP_{\mu \ell}$ has a singular point, there is an interval of values
of~$\lambda$ such that the connected component of
$\cH_{\lambda}^{-1}(h_c) \cap \cP_{\mu \ell}$ that goes through the
singular point is a topological (non-smooth) circle consisting of the
dynamically unstable singular point and its stable=unstable manifold.
We note here that such critical values $(\mu, \ell, h_c)$ that have
$\cH_{\lambda}^{-1}(h_c) \cap \{ Y=0 \}$ in the `interior' of
$\cP_{\mu \ell}$, compare with \fref{fig:tip-3}(middle), lie on the
thread~$\CV_{ij}^0$ for the corresponding $ij \in \{ 12, 13, 23 \}$
and thus in the interior of the image of~$\EM$.

\begin{figure}
  \centering
  \includegraphics[width=6cm]{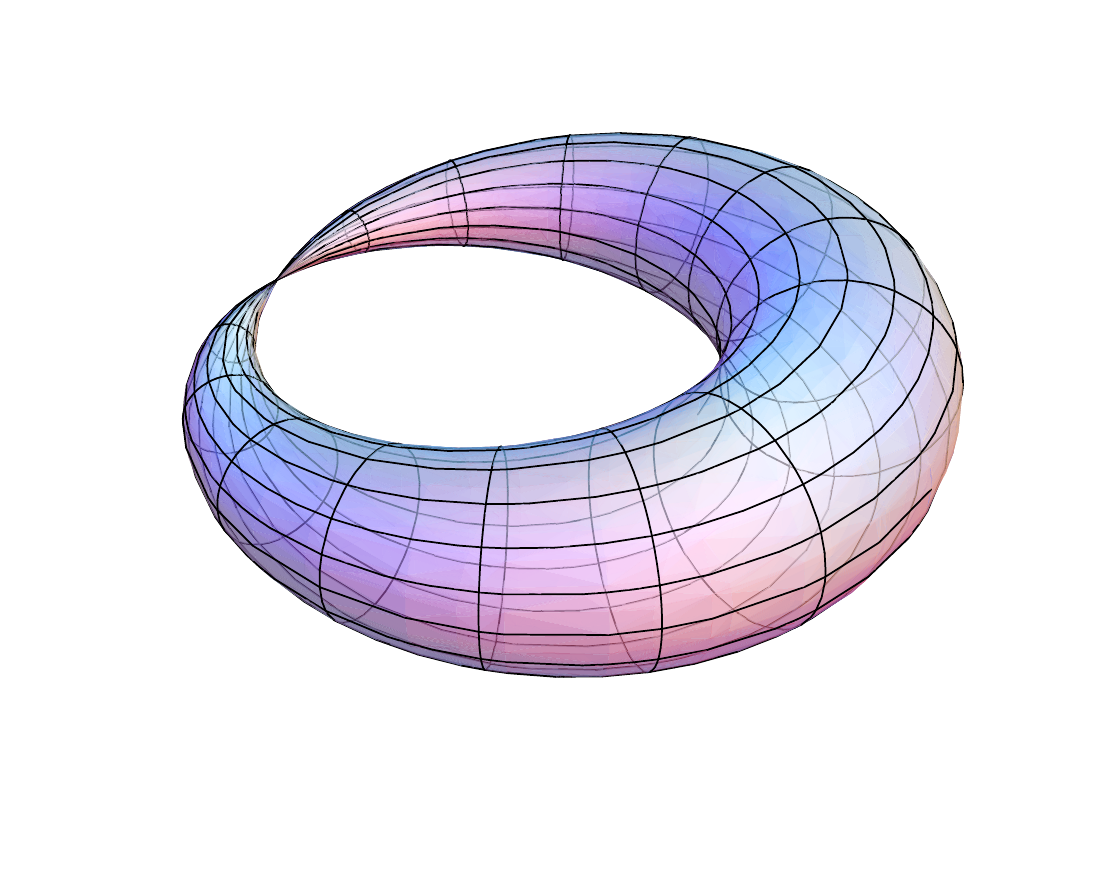}
  \caption{Two-dimensional pinched torus.
  \label{fig:pinchedtorus}}
\end{figure}

\begin{figure}
   \centering
   \includegraphics[width=6cm]{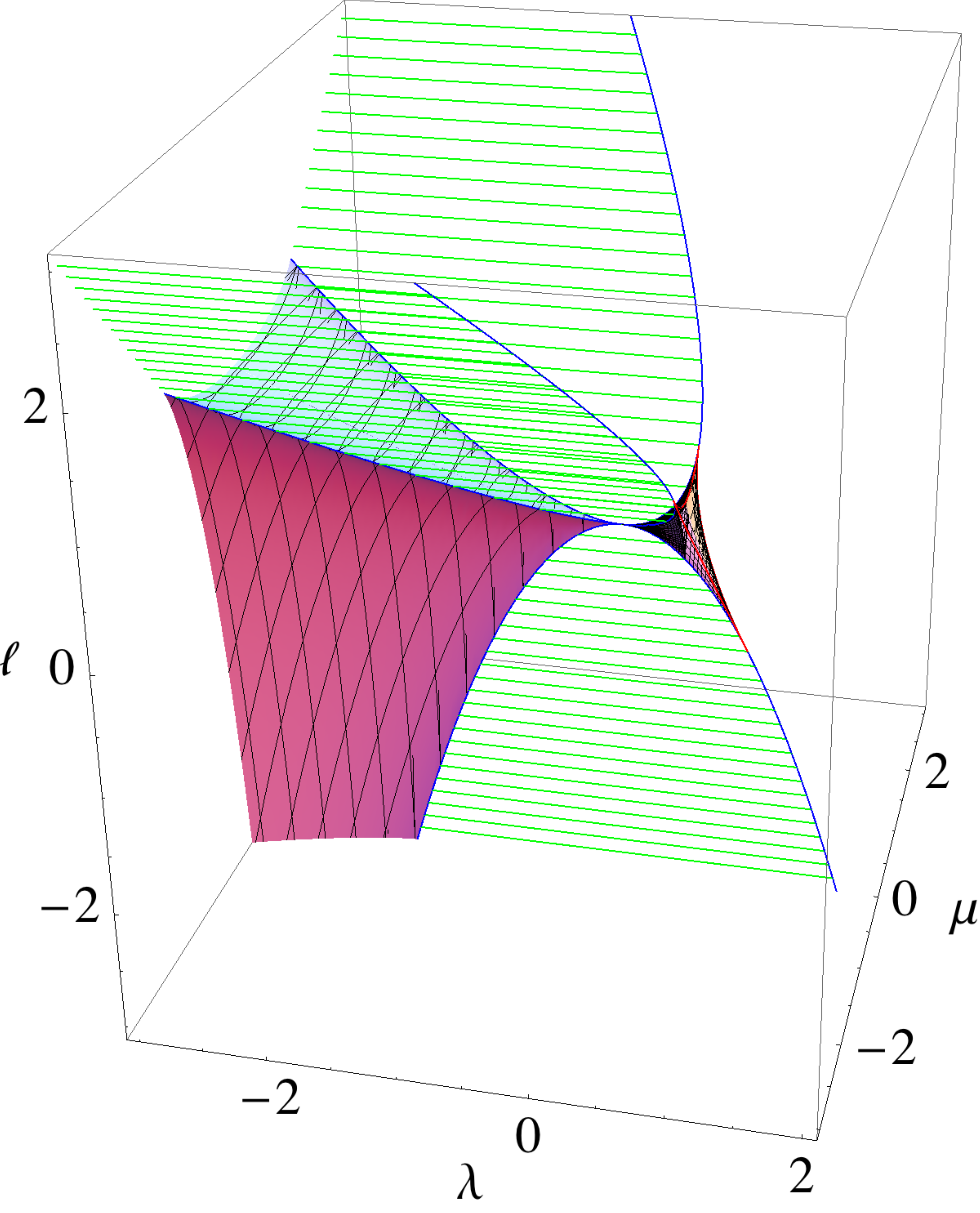}
   \caption{Amended bifurcation diagram.
     The horizontal lines represent intervals of values of
     $\lambda$ for which the fibre of $\cH_{\lambda}$ going through
     the singular point of $\cP_{\mu \ell}$ is a topological circle,
     i.e.\ the singular point is an unstable equilibrium.
   \label{fig:pictureofmu-ellcriticalvalues} }
\end{figure}

Reconstructing the $\T^2$--action~$\Phi$ over a topological
(non-smooth) circle $\cH^{-1}_{\lambda}(h_c) \cap \cP_{\mu \ell}$
gives that the resulting singular fibre is the Cartesian product
of a two-dimensional pinched torus, see
\fref{fig:pinchedtorus}, with a (smooth) torus~$\T^1$.
The latter is the normal mode that had been reduced to the singular
equilibrium on~$\cP_{\mu \ell}$ and the former constitutes its
stable=unstable manifold.
Specifically, the condition for
$\cH_{\lambda}^{-1}(h_c) \cap \cP_{\mu \ell}$
to be a topological circle is $|X_1'(R_{\min})| < |X_2'(R_{\min})|$.
Since
\begin{displaymath}
   F''(R_{\min}) \;\; = \;\; 2 \left(
   X_1'(R_{\min})^2 - X_2'(R_{\min})^2 \right)
\end{displaymath}
we get the equivalent condition $F''(R_{\min}) < 0$.
For fixed values $(\mu,\ell)$ such that $\cP_{\mu \ell}$ is singular
we consider the values of $\lambda$ such that
$\cH_{\lambda}^{-1}(h_c) \cap \cP_{\mu \ell}$ is a topological circle.
Evaluating for $h = h_c$ that
\begin{displaymath}
   F''(R_{\min}) \;\; = \;\; 2 \lambda^2 \; + \; 4 \lambda R_{\min}
   \; + \; 2 (\ell + R_{\min}^2 - 3 R_{\min}) \enspace ,
\end{displaymath}
we conclude that there is exactly one $\lambda$--interval where
$F''(R_{\min}) < 0$, lying between the two real roots of~$F''(R_{\min})$.
The endpoints of the $\lambda$--interval correspond to Hamiltonian
Hopf bifurcations and lemma~\ref{lemma bd char} allows us to
distinguish the supercritical from the subcritical ones.
In case $\mu = 0$, $\ell < 0$, the interval is
$-\sqrt{-\ell} \le \lambda \le \sqrt{-\ell}$.
For $-1 < \ell < 0$ both ends of the interval correspond to subcritical
Hamiltonian Hopf bifurcations.
For $\ell < -1$ the right end $\lambda = \sqrt{-\ell}$ corresponds to
a supercritical Hamiltonian Hopf bifurcation while the left end
$\lambda = -\sqrt{-\ell}$ corresponds to a subcritical Hamiltonian Hopf
bifurcation.
In case $\ell = |\mu| > 0$, the interval is
$-\sqrt{2\ell}-\ell \le \lambda \le \sqrt{2\ell}-\ell$.
For $0 < \ell < 1$ both ends of this interval correspond to subcritical
Hamiltonian Hopf bifurcations.
For $\ell > 1$ the right end $\lambda = \sqrt{2\ell}-\ell$ corresponds
to a supercritical Hamiltonian Hopf bifurcation while the left end
$\lambda = -\sqrt{2\ell}-\ell$ corresponds to a subcritical Hamiltonian
Hopf bifurcation.
These $\lambda$--intervals are represented by the horizontal lines in
\fref{fig:pictureofmu-ellcriticalvalues} for equally spaced
values of~$\ell$.
It is for these values that pinched tori occur --- not a bifurcation,
but a critical element for the description of the dynamics.

We call the diagram in \fref{fig:pictureofmu-ellcriticalvalues}
the \emph{amended bifurcation diagram}.
It combines the bifurcation diagram, discussed
in~\sref{bifurcation diagram kappa != 0}, and the values
$(\lambda,\mu,\ell)$ for which
$\cH_{\lambda}^{-1}(h_c) \cap \cP_{\mu \ell}$
is a topological (non-smooth) circle.
Each point in this diagram thus corresponds to a critical value of
the energy-momentum mapping~\eqref{reducedenergymomentummapping}
and the diagram is used as a starting point for deducing the structure
of the set of critical values of~$\EM$.

\subsection{Set of critical values of $\boldsymbol{\EM}$}
\label{sets of critical values}

\noindent
To obtain the set of critical values~$\CV$ of the energy-momentum
mapping~$\EM$ we consider different cases for the parameters
$\delta$, $\lambda_1$, $\lambda_2$ appearing in the Hamiltonian
function.
For fixed values of $\delta$, $\lambda_1$, $\lambda_2$ the relation
$\lambda = \delta + \lambda_1 \mu + \lambda_2 \ell$
in~\eqref{linearparameterdependence} defines an embedding of the
$(\mu, \ell)$--plane into the $(\lambda,\mu,\ell)$--space.
The intersection of this embedded plane with the amended bifurcation
diagram provides information that allows us to reconstruct a large
part of~$\CV$.
In this section, we determine $\CV$ for different choices of $\delta$
while fixing $\lambda_1 = \lambda_2 = 0$,
that is, we consider only vertical planes
$\lambda = \text{constant}$ in the $(\lambda,\mu,\ell)$--space.
The study of such vertical planes gives a complete description
of possible behaviors also for slightly tilted planes;
the only exception is the (degenerate) $\lambda = 1/(2\kappa)$
where a slight tilt qualitatively changes $\CV$.
For more strongly tilted planes, arguments similar to the ones
we use below allow to determine $\CV$ for any other choice of
$\delta$, $\lambda_1$, $\lambda_2$ and also for the case
$\kappa=0$.
Nevertheless, a complete description of all possible cases is beyond
the aim of this paper.

\subsubsection{Case $\boldsymbol{\delta = \lambda_1 = \lambda_2 = 0}$}
\label{sec:CVCase1}

\noindent
In this `undetuned' case the $(\mu, \ell)$--plane embeds as the
$(\lambda{=}0)$--plane in the $(\lambda,\mu,\ell)$--space.
We can deduce a large part of $\CV$ by checking the intersection
of the $(\lambda{=}0)$--plane with the amended bifurcation diagram
in figure~\ref{fig:pictureofmu-ellcriticalvalues}.
The plane $\lambda = 0$ intersects the three
surfaces of $h$--isolated critical values along the lines
$\mu = \pm \ell$, $0 \le \ell \le 2$ and $\mu = 0, \ell \le 0$,
see \fref{fig:partialcvlambda=0}(left).
At $\ell = 2$ the two lines $\mu = \pm \ell$ end at supercritical
Hamiltonian Hopf bifurcations while the line $\mu = 0$ extends
indefinitely.
These three lines give three curves of critical values of~$\EM$,
where the value of~$h_c$ has to also be taken into account as
discussed in \sref{amended bd}.
In particular, the curves $\CV_{23}$ and $\CV_{13}$ have parts
$\CV_{23}^0 = \CV_{23}^+$ and $\CV_{13}^0 = \CV_{13}^+$ that end
at supercritical Hamiltonian Hopf bifurcations while
$\CV_{12} = \CV_{12}^0$ extends indefinitely.

The parts $\CV_{ij} \backslash \CV_{ij}^+$, $ij = 23, 13$ have
$h_c = h_{\min}$ --- here the normal $1$--mode respectively the
normal $2$--mode is stable as
$\cH_0^{-1}(h_{\min}) \cap \cP_{\mu \ell}$
is a single point, the singular point reduced from the normal mode.
The rest of $\CV$ consists of values where the energy level touches
the reduced space from outside, in a regular point
of~$\cP_{\mu \ell}$.
This implies that for such a value the energy~$\cH_0$ is at minimum
for given $(\mu, \ell)$ and the set of such values
$(\mu,\ell,h_{\min})$ yields the boundary of the image of~$\EM$.
We denote by $\BV \subseteq \CV$ the surface
$\{ (\mu,\ell,h_{\min}) : (\mu,\ell) \in \R^2 \}$, note that $\BV$
also contains the lines $\CV_{ij} \backslash \CV_{ij}^+$,
$ij = 23, 13$.
The set $\CV$ of critical values is shown in
\fref{fig:cvlambda=0}(right).

\begin{figure}
  \centering
  \includegraphics[width=6cm]{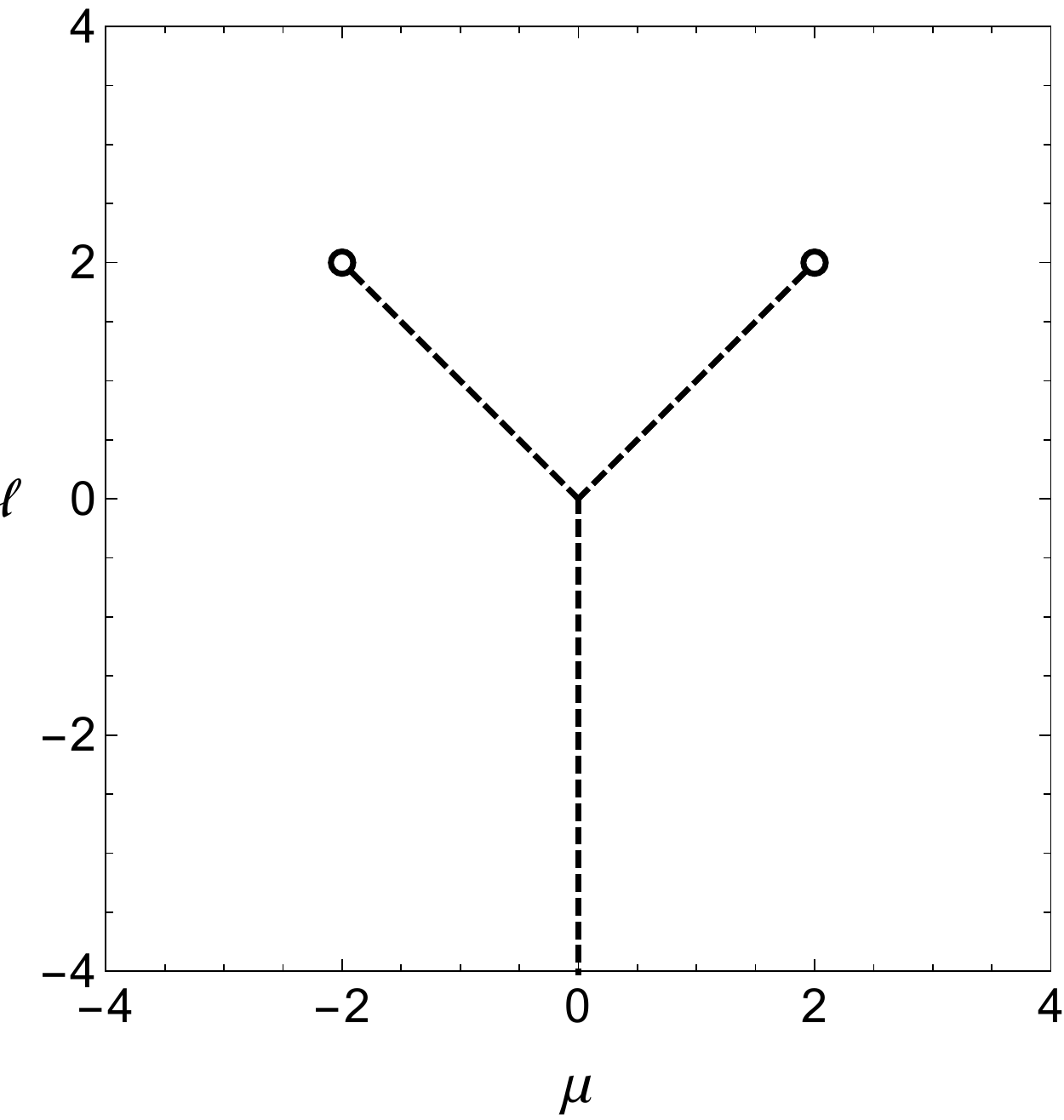}
  \hspace{1cm}
  \includegraphics[width=7cm]{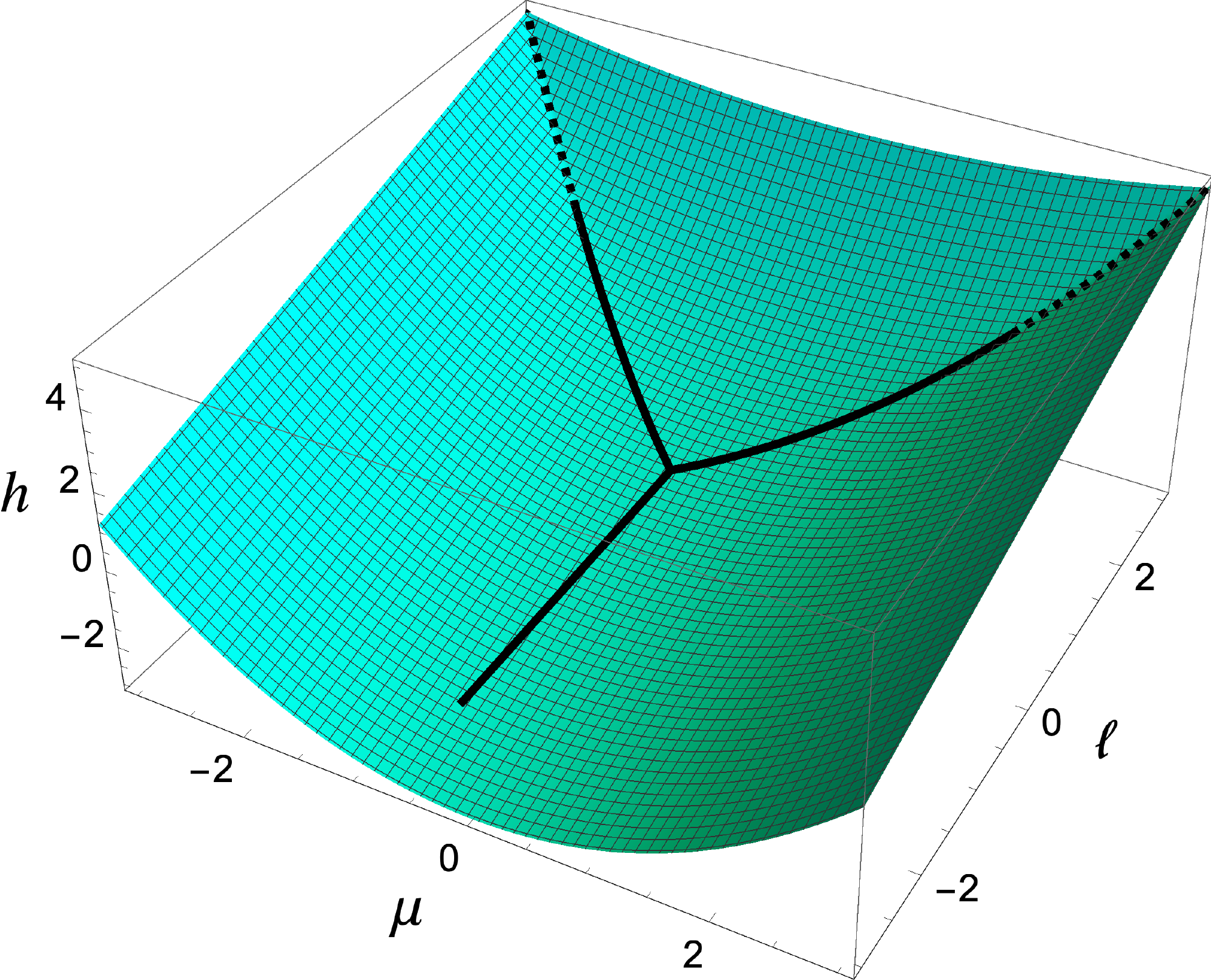}
  \caption{
    Left: intersection of the plane $\lambda = 0$ with the amended
    bifurcation diagram.
    The dashed lines parametrise families of unstable periodic
    orbits, the $\circ$ stand for supercritical Hamiltonian Hopf
    bifurcations.
    Right: set $\CV$ of critical values of $\EM$ for
    $\delta = \lambda_1 = \lambda_2 = 0$.
    The three solid lines lie above the surface and at the two
    dashed curves the otherwise smooth surface~$\BV$ has two
    creases; the transition is at the two supercritical Hamiltonian
    Hopf bifurcations.
    Values in the interior of each solid curve correspond to the
    Cartesian product of a two-dimensional pinched torus with~$\T^1$;
    values on the dashed lines to~$\T^1$; values on the surface
    to~$\T^2$; regular values above the surface to~$\T^3$.
  \label{fig:partialcvlambda=0} 
  \label{fig:cvlambda=0} }
\end{figure}

The set~$\Reg$ of regular values $(\mu,\ell,h)$ of~$\EM$
parametrises the Lagrangean~$\T^3$ in phase space.
Critical values on $\BV$ parametrise smooth $\T^2$ in phase space
which are also $\Phi$--orbits, provided that they do not belong to
$\CV_{13}$ or $\CV_{23}$.
Critical values on $\CV_{13} \cap \BV$ or $\CV_{23} \cap \BV$ lift
to~$\T^1$.
Critical values along the three threads of critical values
$\CV_{ij}^0$ correspond to singular fibres, the Cartesian product
of a two-dimensional pinched torus, see \fref{fig:pinchedtorus},
with a~$\T^1$, the latter associated to a circle action that acts
freely on the fibre.
Note that there is no globally defined circle action arising from a
linear combination of $\X{N}$ and~$\X{L}$ that acts freely on fibres
over all three curves of critical values.
Finally, the critical value $(\mu,\ell,h)=(0,0,0)$, where the three
threads meet, corresponds to a singular fibre, which can be described
as a~$\T^3$ where a $\T^2$~orbit has been `pinched' to a point.

\begin{remark}
It is the undetuned case where the central equilibrium is in
$1{:}1{:}{-}2$~resonance and the level set $\cH_0^{-1}(0)$ passes
through the cuspidal singularity of~$\cP_{00}$, yielding a
topological (non-smooth) circle.
Hence, in three degrees of freedom each regular point gets a~$\T^2$
attached and all these form the stable$=$unstable manifold of the
central equilibrium, which has isotropy~$\T^2$
under~\eqref{oscillatorsymmetry}, revealing the
$1{:}1{:}{-}2$~resonant equilibrium to be unstable (despite being
linearly stable).
This is reminiscent of both the `phantom kiss' at a periodic orbit
undergoing a $1{:}3$~normal-internal resonance, compare
with~\cite{AM78, hh07}, and the normal $1{:}{-}2$~resonance in
two degrees of freedom, compare with~\cite{AO84, mlb87}.
\end{remark}

\subsubsection{Case $\boldsymbol{\delta < 0,\,  \lambda_1 = \lambda_2 = 0}$}
\label{sec:CVCase2}

\noindent
Adding a small detuning $\delta \ne 0$ qualitatively modifies the
set of critical values only in a neighbourhood of the origin.
We first consider the case $\delta < 0$, $\lambda_1=\lambda_2=0$.
Here the $(\mu,\ell)$--plane embeds in $(\lambda,\mu,\ell)$--space
as the plane $\lambda = \delta < 0$, and all such planes have
qualitatively the same intersection with the amended bifurcation
diagram.

\begin{figure}
  \centering
  \includegraphics[width=6cm]{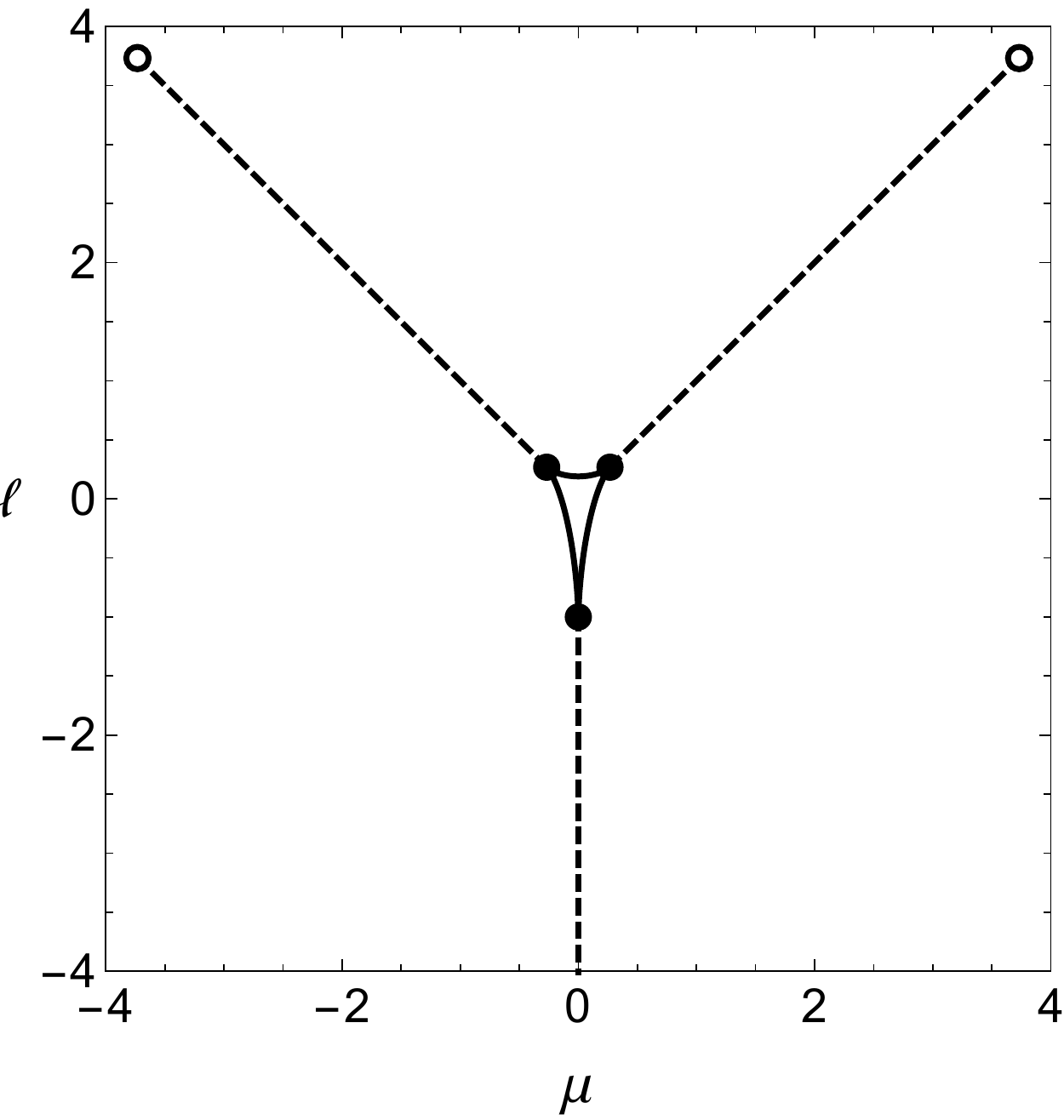}
  \hspace{1cm}
  \includegraphics[width=7cm]{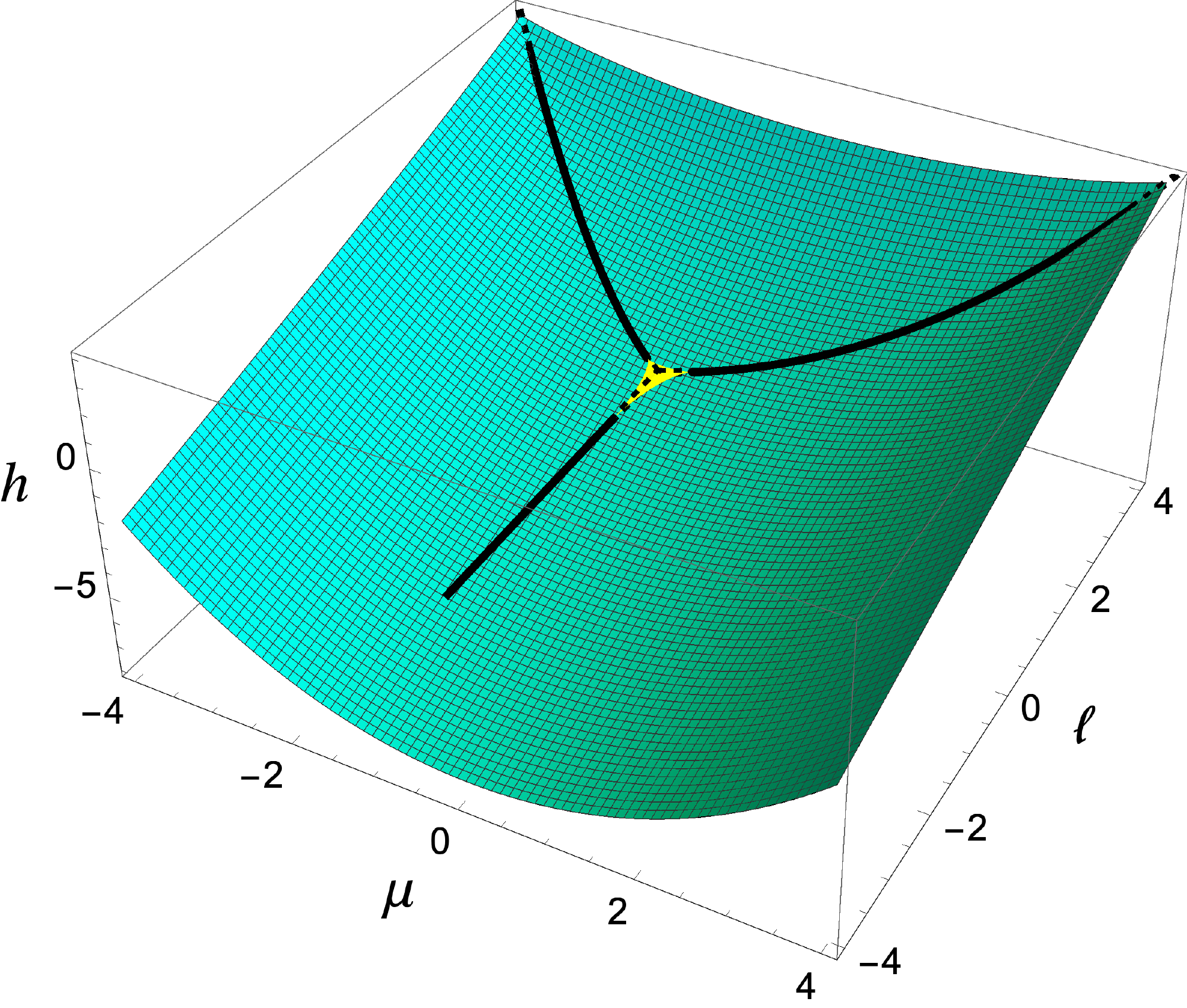}
  \caption{
    Left: intersection of the plane $\lambda = -1$ with the
    amended bifurcation diagram.
    The solid lines parametrise families of centre-saddle
    bifurcations, the $\bullet$ stand for subcritical Hamiltonian
    Hopf bifurcations, and
    the $\circ$ stand for supercritical Hamiltonian Hopf
    bifurcations.
    Right: set of critical values $\CV$ for $\delta=-1$,
    $\lambda_1=\lambda_2=0$.
    For an enlargement of the central region see
    \fref{cv detail lambda<0}.
  \label{cv lambda<0} }
\end{figure}

\begin{figure}
   \centering
   \includegraphics[width=7cm]{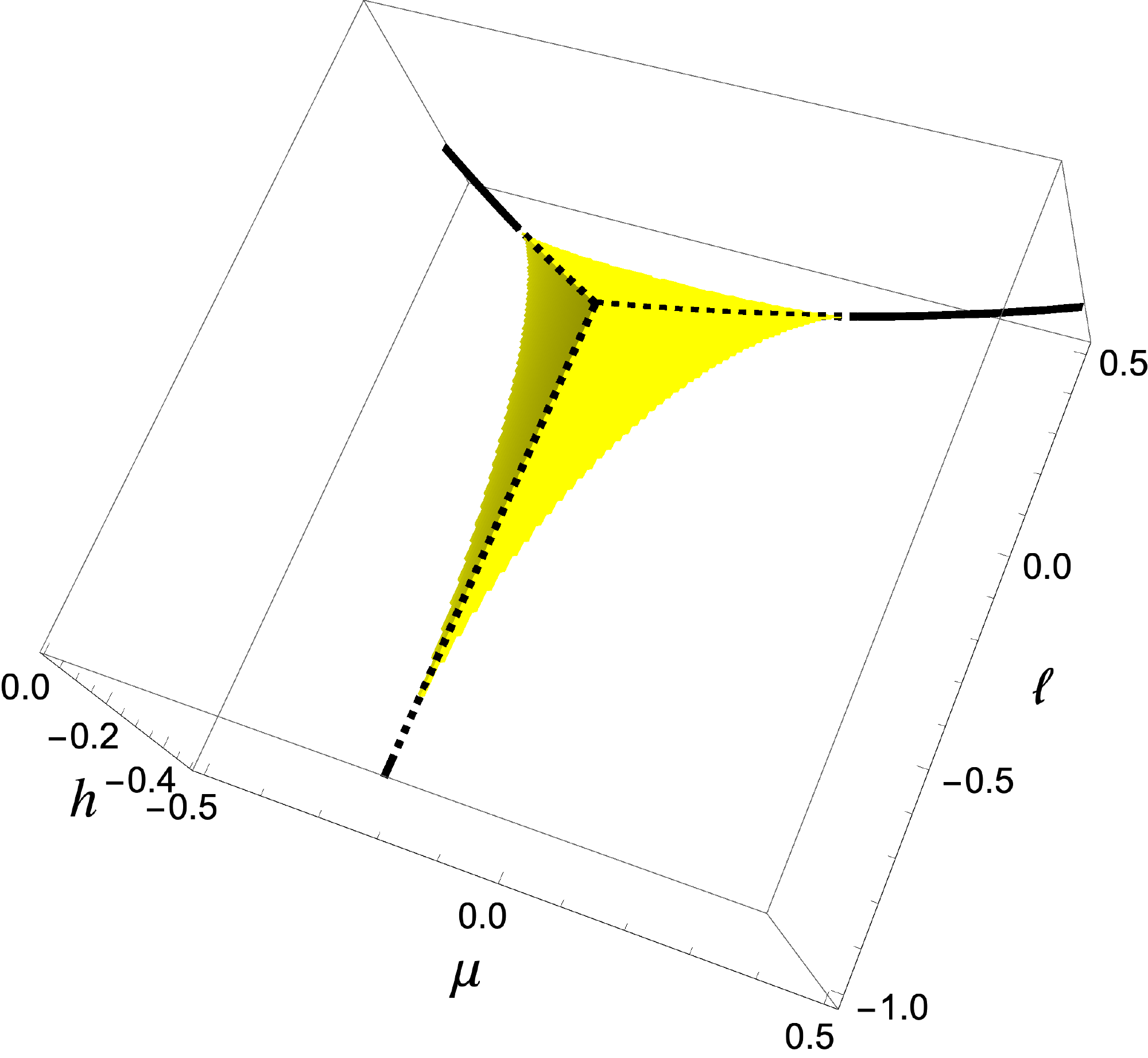}
   \caption{
     Detail of $\CV$ for $\lambda=-1$.
   \label{cv detail lambda<0}}
\end{figure}
 
A plane $\lambda = \delta$ intersects the amended bifurcation
diagram along three straight lines and a curvilinear
triangle~$\DV$, see \fref{cv lambda<0}(left).
Each of the straight lines joins $\DV$ at a vertex
corresponding to a subcritical Hamiltonian Hopf bifurcation.
The two straight lines parametrised by $\mu = \pm \ell$ with
$1-\lambda-\sqrt{1-2\lambda} < \ell < 1-\lambda+\sqrt{1-2\lambda}$
end at supercritical Hamiltonian Hopf bifurcations.
The straight line parametrised by $\ell < -\lambda^2$ at $\mu = 0$
extends indefinitely.
The edges of~$\DV$ are the intersections with the surfaces of
centre-saddle bifurcations.

The set of critical values~$\CV$ is depicted in
\fref{cv lambda<0}(right).
Away from the origin, the description of $\CV$ from
\sref{sec:CVCase1} can be repeated verbatim for this case.
However, the situation is different near the origin,
see \fref{cv detail lambda<0}.
In the set of critical values we note the appearance of a
tetrahedral surface~$\TV$ of critical values, corresponding
to the appearance of~$\DV$ in the intersection of the plane
$\lambda = \delta$ with the amended bifurcation diagram.
The surface~$\TV$ separates the set~$\Reg$ of regular values
into two connected components: $\Reg'$ outside $\TV$ and
$\Reg''$ inside~$\TV$.
Values in $\Reg'$ lift to~$\T^3$.
However, values $v = (\mu,\ell,h) \in \Reg''$ correspond to the
union of two disjoint $\T^3$ in phase space.
We denote by $\T^3_A(v)$ and $\T^3_B(v)$ the two resulting disjoint
families of tori parametrised by $v \in \Reg''$.

The curvilinear triangle $\DV$ of centre-saddle bifurcations
embeds in $\TV$ as the union of three curved edges~$\DV'$
that splits the ``upper'' and ``lower'' parts of~$\TV$.
The ``upper'' part of $\TV$ is visible in \fref{cv detail lambda<0}.
It consists of three faces that we denote by $\FV_+$, $\FV_-$
and~$\FV_0$, respectively, and the three straight lines
$\CV_{ij}^+ \backslash \CV_{ij}^0$ of stable normal modes.
To be precise, $\FV_0$ has $\ell > |\mu|$ and lies between
$\CV_{13}$, $\CV_{23}$ and $\DV'$ while $\FV_+$ has
$\mu > \max(\ell, 0)$ and lies between $\CV_{12}$, $\CV_{23}$
and~$\DV'$ and $\FV_-$ has $\mu < \min(-\ell, 0)$ and lies between
$\CV_{12}$, $\CV_{13}$ and~$\DV'$.
We denote by $\FV_e^\circ$ the union of the faces $\FV_{\plminul}$
and by $\FV_e$ the union of $\FV_e^\circ$ with the pairwise common
topological boundaries of $\FV_{\plminul}$ (the dashed lines in
\fref{cv detail lambda<0}), so $\FV_e$ constitutes the ``upper''
part of~$\TV$.
The ``lower'' part consists of a single face $\FV_h$ with
topological boundary $\DV'$.
Note that in all four cases we do not consider the topological
boundary of a face to belong to the face.

Values on the ``upper'' part~$\FV_e^\circ$ lift to the disjoint
union of a $\T^2$ and a~$\T^3$.
Here, one of the families $\T^3_{A,B}(v)$ in $v \in \Reg''$
shrinks down to a $\T^2$ as $v$ approaches~$\FV_e^\circ$.
This allows us to `globally' distinguish the two families:
we take $\T^3_B(v)$ to be the family that shrinks down to~$\T^2$
and $\T^3_A(v)$ the family that can be smoothly continued outside
of~$\Reg''$ to the regular $\T^3$--fibration over~$\Reg'$.
Dynamically, the lower dimensional invariant torus~$\T^2$ is elliptic.
The pairwise common boundaries of $\FV_{\plminul}$ meet at the origin.
The origin lifts to the disjoint union of a point, the central
equilibrium, and a~$\T^3$ in phase space.
Other values on these pairwise common boundaries lift to the
disjoint union of a normal mode (a~$\T^1$) and a~$\T^3$.

\begin{figure}
  \centering
  \includegraphics[width=6cm]{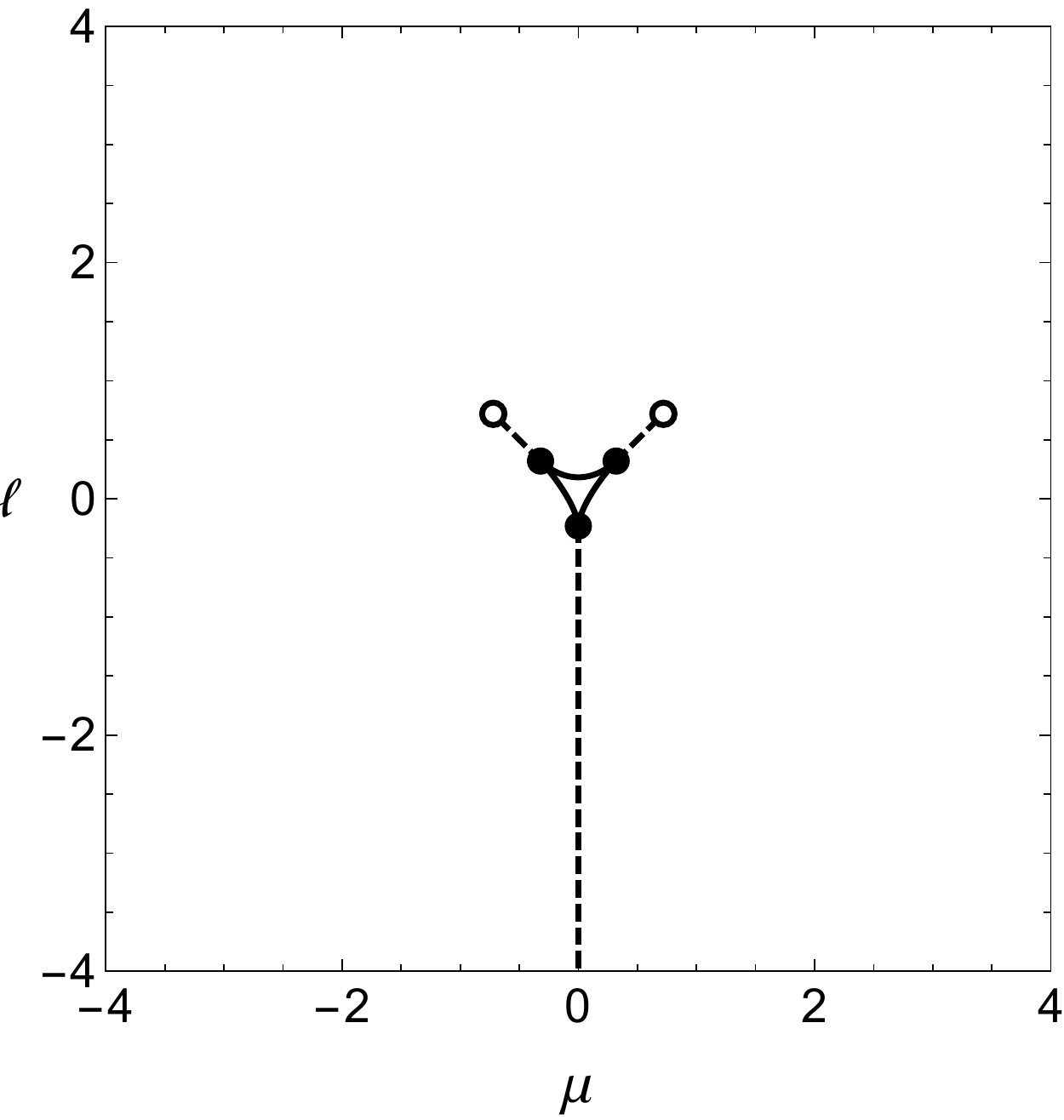}
  \hspace{1cm}
  \includegraphics[width=7cm]{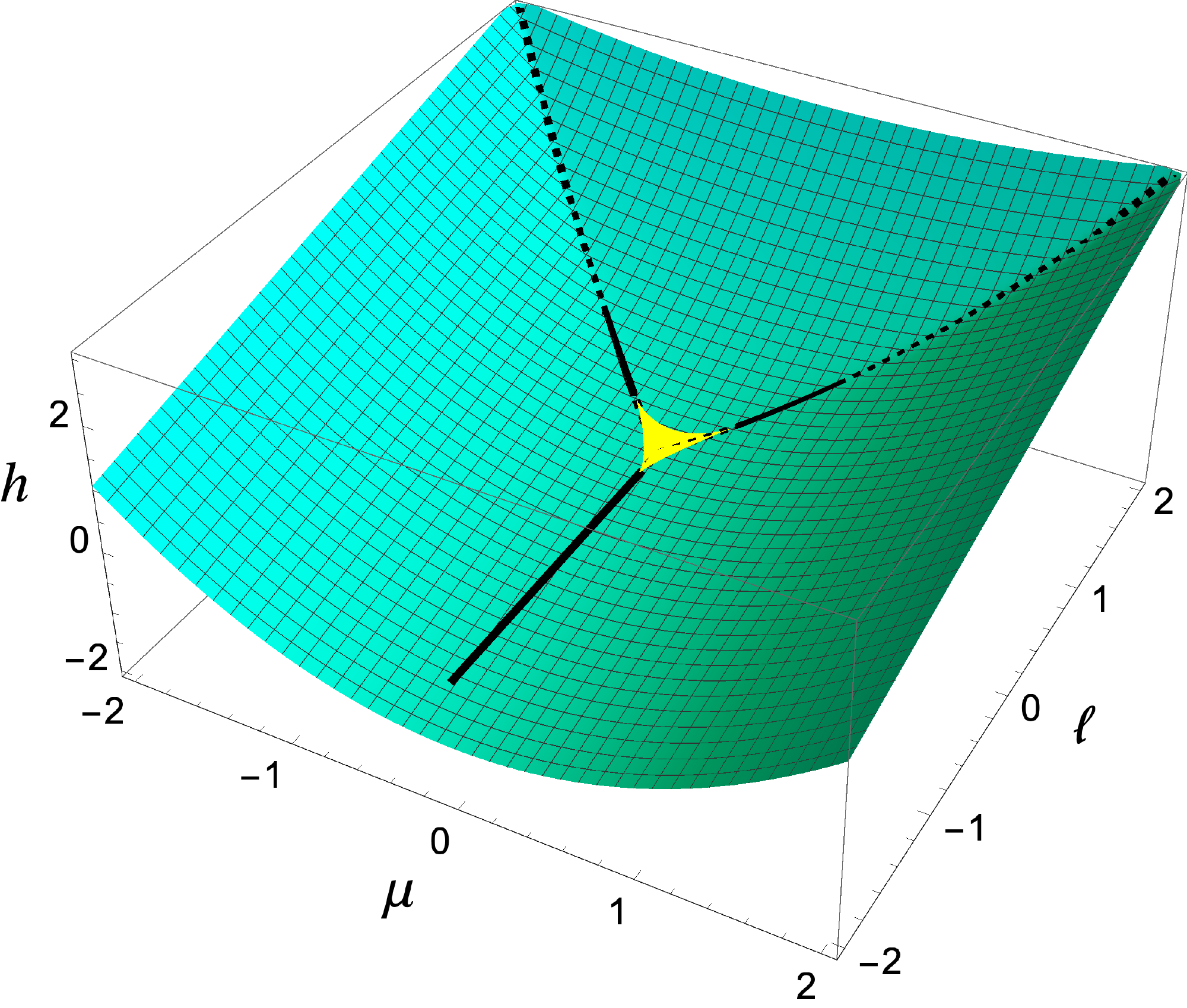}
  \caption{
    Intersection of the plane $\lambda = 0.48$ with the amended
    bifurcation diagram and the corresponding set of critical
    values~$\CV$.
  \label{cv 0<lambda<1/2} }
\end{figure}

Values on the ``lower'' part~$\FV_h$ lift to a connected singular
fibre.
The singular fibre is the Cartesian product of a figure~eight with
a~$\T^2$, that is, it corresponds to the two disjoint families
$\T^3_A(v)$ and $\T^3_B(v)$ in $\Reg''$ getting glued together
along a common~$\T^2$.
Dynamically, this lower dimensional invariant torus~$\T^2$ is
hyperbolic and the two glued $\T^3$ correspond to its
stable=unstable manifold.
The face~$\FV_h$ meets each one of the faces~$\FV_{\plminul}$ along
a family of centre-saddle bifurcations (an edge of~$\DV'$), where
the hyperbolic $\T^2$ from~$\FV_h$ and the elliptic $\T^2$
from~$\FV_e^\circ$ meet and disappear.
Moving on $\FV_h$ toward an edge of~$\DV'$, the component~$\T^3_B$
(glued with $\T^3_A$ to form the corresponding singular fibre)
shrinks and then disappears at the edge of~$\DV'$.

\begin{remark}
\label{passagethrough}
It is instructive to have a look at the quantitative changes that
occur as $\delta$ increases towards~$0$ before the qualitative
change at the case $\delta = 0$ discussed before.
Indeed for $\delta \nearrow 0$ the tetrahedral surface~$\TV$ gets
smaller and smaller, shrinking to the critical value
$(\mu, \ell, h) = (0, 0, 0)$ where the three threads that we have
seen to exist for $\delta = 0$ meet.
For $\delta > 0$ another tetrahedral set~$\TV^{\prime}$ emerges,
`flipped upside-down' compared to~$\TV$.
Thus, while passing through the $1{:}1{:}{-}2$~resonance, the three
straight lines $\CV_{ij}^+ \backslash \CV_{ij}^0$ parametrising
the stable normal modes of the elliptic central equilibrium shrink
down (and re-grow) as the central equilibrium momentarily loses its
stability at the $1{:}1{:}{-}2$~resonance itself.
\end{remark}

\subsubsection{Cases $\boldsymbol{\delta > 0,\,  \lambda_1 = \lambda_2 = 0}$}
\label{sec:CVCases345}

\noindent
We now turn our attention to the case $\delta > 0$.
Here the $(\mu,\ell)$--plane embeds in $(\lambda,\mu,\ell)$--space
as the plane $\lambda = \delta > 0$ and we must consider three
subcases depending on the value of~$\delta$.
Indeed, when comparing with the case $\delta < 0$ we see that the
planes $\lambda = \delta > 0$ intersect the amended bifurcation
diagram in five different ways.
We omit the transitional cases $\delta = \frac{1}{2}$, $\delta = 1$
and concentrate on the three open intervals $\opin{0}{\frac{1}{2}}$,
$\opin{\frac{1}{2}}{1}$ and $\opin{1}{\infty}$ of small,
intermediate and large positive detuning~$\delta$.

\begin{figure}
  \centering
  \includegraphics[width=6cm]{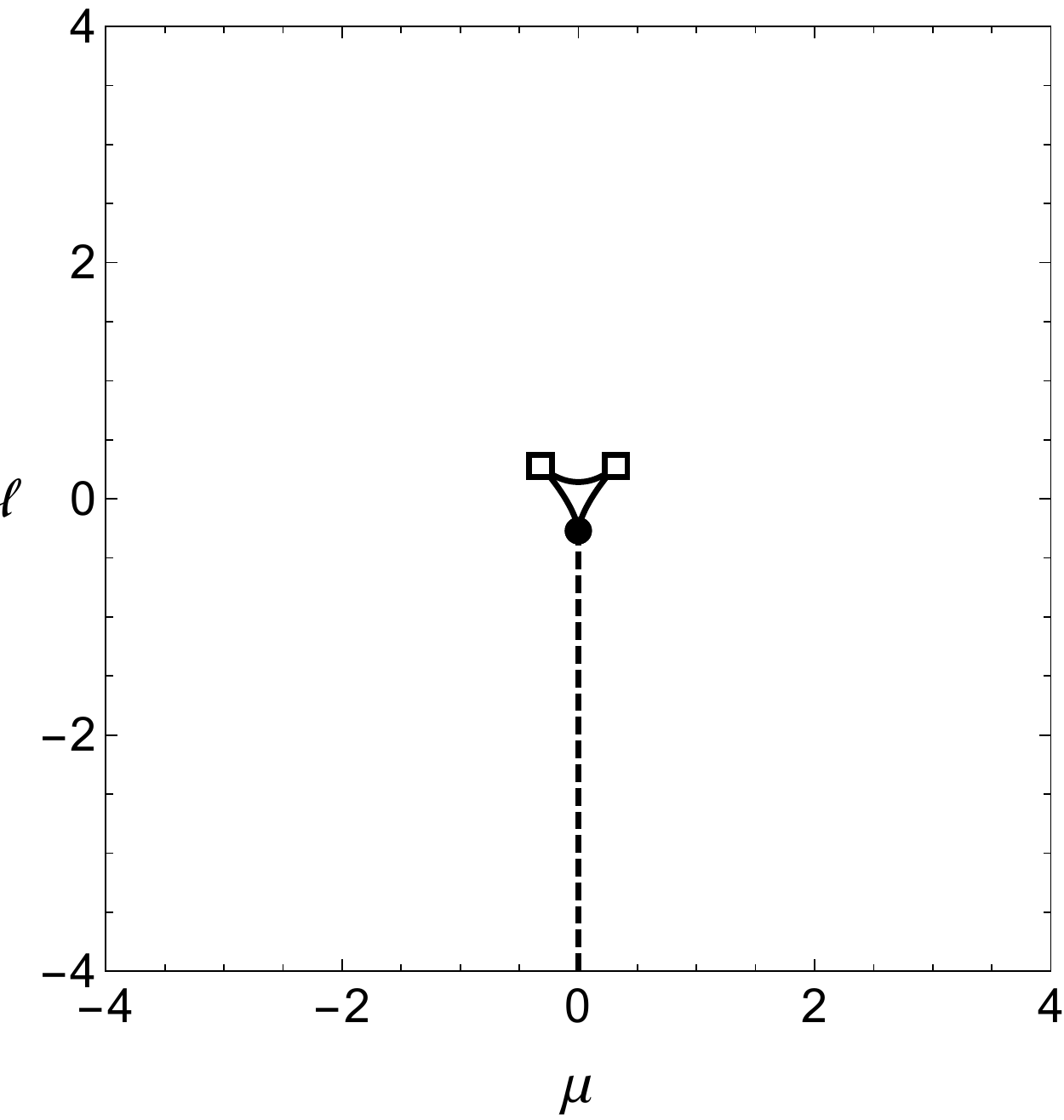}
  \hspace{1cm}
  \includegraphics[width=7cm]{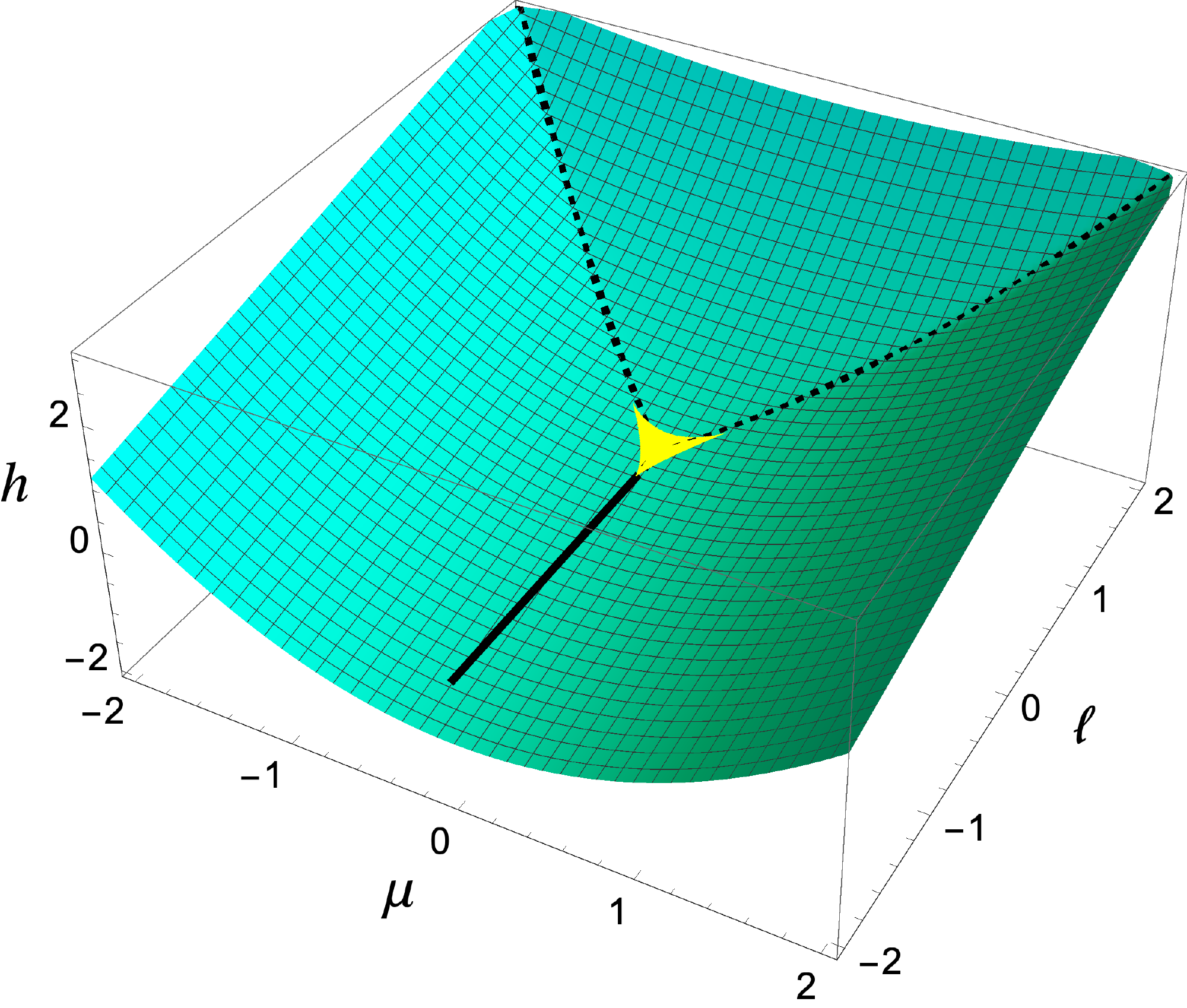}
  \caption{
    Left: intersection of the plane $\lambda = 0.52$ with the
    amended bifurcation diagram.
    The~{\scriptsize $\square$} stand for cusp bifurcations and the~$\bullet$ marks
    the subcritical Hamiltonian Hopf bifurcation.
    Right: corresponding set of critical values $\CV$ of~$\EM$.
    For an enlargement of the central region see
    \fref{cv detail 1/2<lambda<1}.
  \label{cv 1/2<lambda<1} }
\end{figure}

\begin{figure}
   \centering
   \includegraphics[width=7cm]{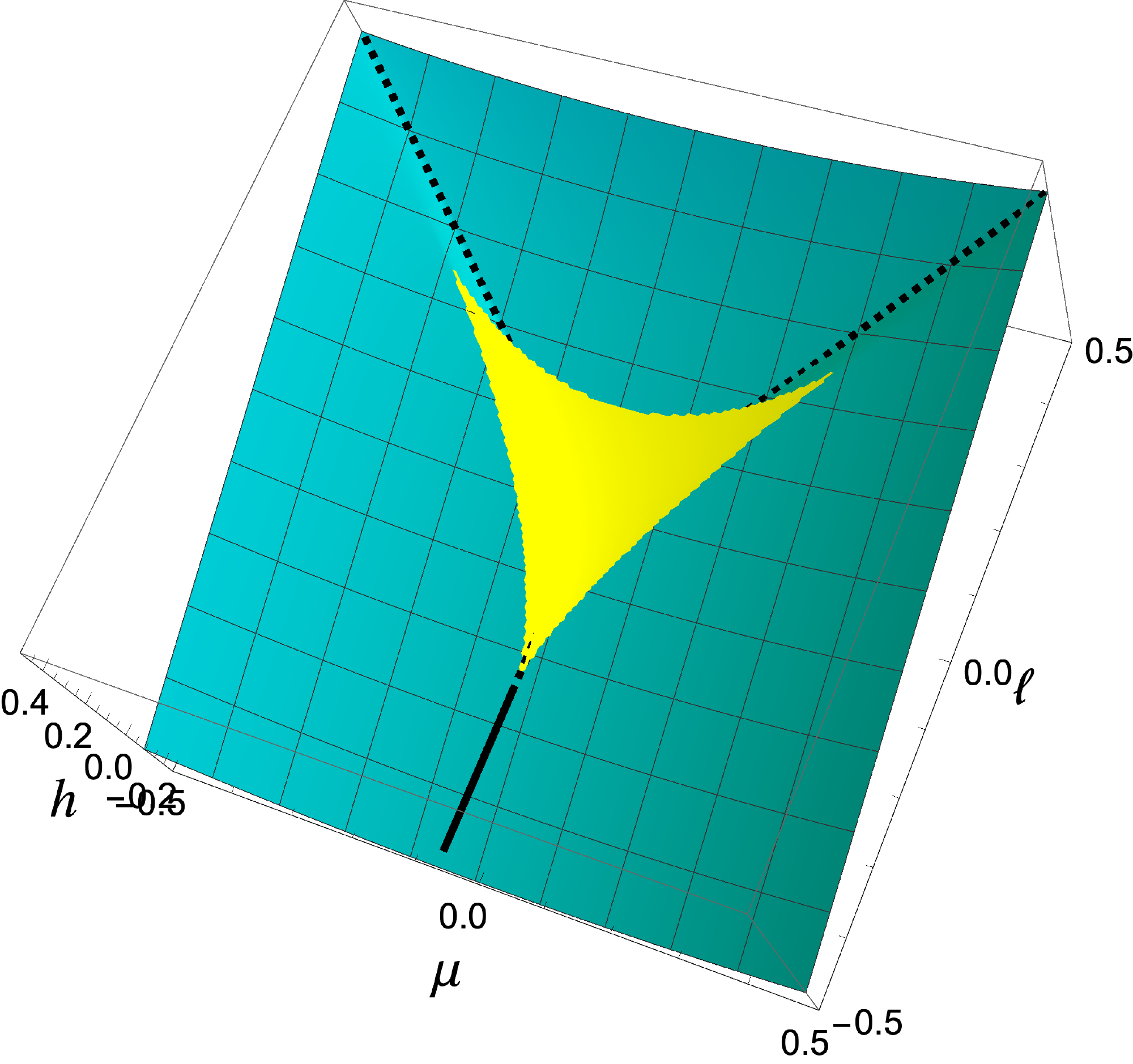}
   \hspace{1cm}
   \includegraphics[width=7cm]{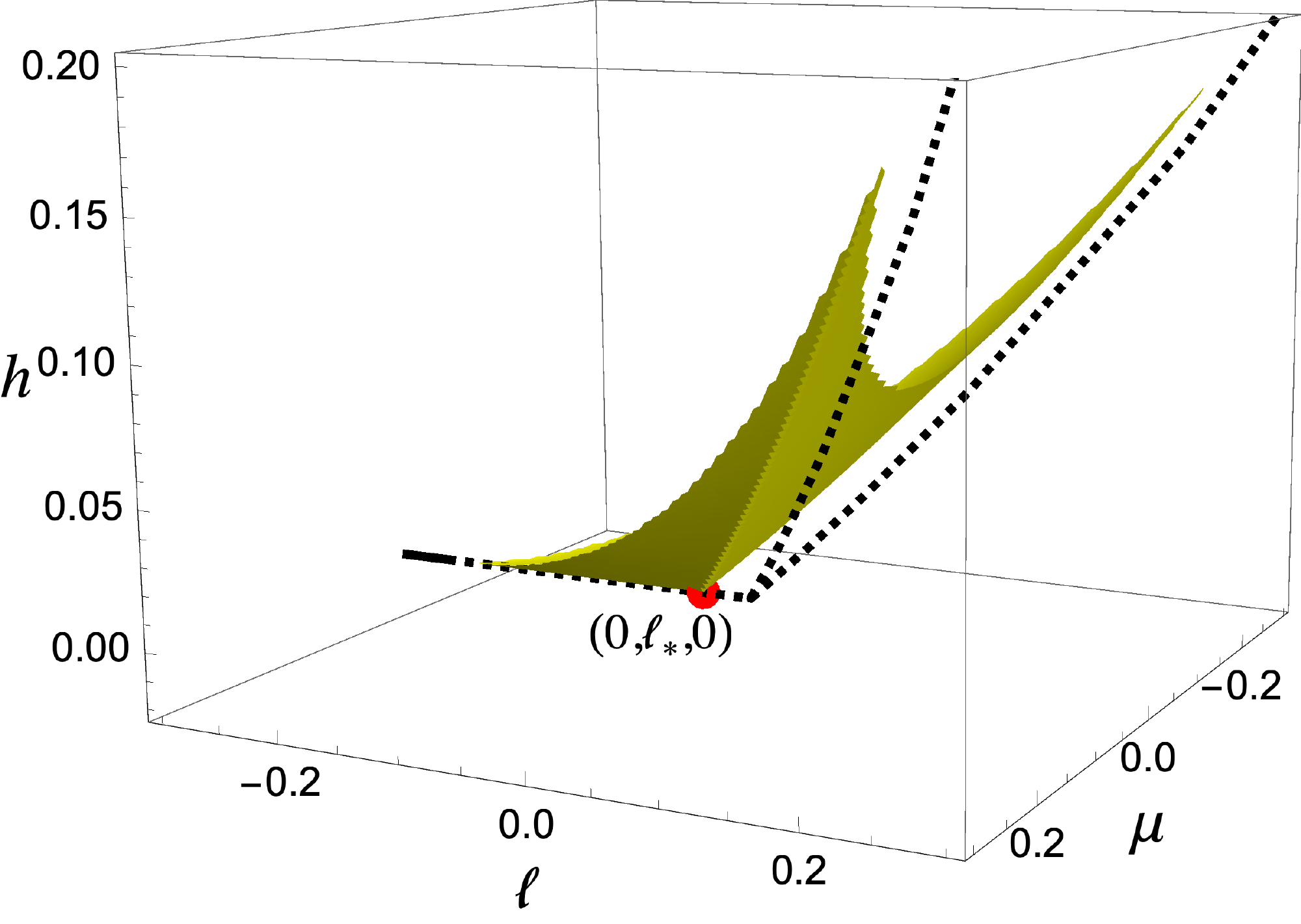}
   \caption{Two views of $\CV$ for $\lambda=0.52$ focusing on $\TV''$.
   At the right panel we are not showing the whole set $\BV$ so as to make more prominent the features of $\TV''$ --- we are drawing however the curves $\CV_{12}$, $\CV_{23}$, $\CV_{13}$ and we note that $\CV_{13}$, $\CV_{23}$, and the part of $\CV_{12}$ between the origin and $(0,\ell_*,0)$ are contained in $\BV$.
   Note that $\TV''$ touches $\BV$ only along the `creases' $\mathcal L_+$ and $\mathcal L_-$ of $\TV''$ that start at $(0,\ell_*,0)$ and end at the two topmost `horns'.
   \label{cv detail 1/2<lambda<1} }
\end{figure}

\paragraph{Case $\boldsymbol{0 < \delta < \frac{1}{2}}$.}
The intersection of the plane $\lambda = \delta$ with the amended
bifurcation diagram as depicted in \fref{cv 0<lambda<1/2}(left) is
qualitatively the same as in the case $\delta < 0$.
The set of critical values~$\CV$ as depicted in
\fref{cv 0<lambda<1/2}(right) is also similar, with the only
difference that the tetrahedral set~$\TV^{\prime}$ of critical
values has been flipped upside-down, see remark~\ref{passagethrough}.
Therefore, it is now the ``upper'' part where we find the
face~$\FV_h$ with hyperbolic~$\T^2$, while the ``lower''
part~$\FV_e$ consists of three creases parametrising the normal modes
joined by three faces~$\FV_{\plminul}$ where we find elliptic~$\T^2$.
The discussion in \sref{sec:CVCase2} carries over
{\it mutatis mutandi}.

\begin{remark}
It is instructive to have a look at the quantitative changes that
occur as $\delta$ increases towards~$\frac{1}{2}$.
Indeed, exactly at $\delta = \frac{1}{2}$ the two vertices of the
tetrahedral surface~$\TV^{\prime}$ correspond to degenerate
Hamiltonian Hopf bifurcations.
Moreover, between these two vertices extends a whole curve
of cusp bifurcations~$\Cusp{3}$.
This is probably the most important difference to considering
`tilted' planes with $(\lambda_1, \lambda_2) \neq (0, 0)$ which
e.g.\ for small $\lambda_2 > 0$ (and $\lambda_1 = 0$) `first'
contain two degenerate Hamiltonian Hopf bifurcations (for an
exceptional value $\delta^* < \frac{1}{2}$) and (increasing
$\delta$ slightly above~$\delta^*$) contain four isolated values
of cusp bifurcations.
For $\lambda_1 \neq 0$ the two degenerate Hamiltonian Hopf
bifurcations lie in different tilted planes.

For $\delta \nearrow \frac{1}{2}$ (with $\lambda_1 = \lambda_2 = 0$)
the face $\FV_0$ of~$\TV^{\prime}$ moves closer and closer to the
surface~$\BV$ of minimal energy values, until at
$\delta = \frac{1}{2}$ the intersection $\TV^{\prime} \cap \BV$
consists of~$\FV_0$, of the values
$(\mu, \ell, h) = (\frac{1}{2}, \pm \frac{1}{2}, \frac{1}{2})$ of
degenerate Hamiltonian Hopf bifurcations and
$(\mu, \ell, h) = (0, 0, 0)$ of the central equilibrium, and of the
parts of $\CV_{13}$, $\CV_{23}$ that extend between the central
equilibrium and one of the degenerate Hamiltonian Hopf bifurcations.
For $\delta > \frac{1}{2}$ we denote the tetrahedral surface, which
no longer contains the origin $(\mu, \ell, h) = (0, 0, 0)$,
by~$\TV^{\prime \prime}$ and leave the description of its position
with respect to~$\BV$ to case $\frac{1}{2} < \delta < 1$ below.
\end{remark}

\paragraph{Case $\boldsymbol{\frac{1}{2} < \delta < 1}$.}
In the intersection of the plane $\lambda = \delta$ with the amended
bifurcation diagram, depicted in \fref{cv 1/2<lambda<1}(left), the
two line segments along $\mu = \pm \ell$ parametrising unstable
normal modes have disappeared.
Their endpoints, which are supercritical and subcritical Hamiltonian
Hopf bifurcations, met at $\delta = \frac{1}{2}$ and for
$\delta > \frac{1}{2}$ the corresponding vertices of the curvilinear
triangle~$\DV$ stand for cusp bifurcations.

The changes in the set of critical values, depicted in
\fref{cv 1/2<lambda<1}(right), reflect the changes in the
intersection with the amended bifurcation diagram.
We focus on the modified part of the set of critical values depicted
in \fref{cv detail 1/2<lambda<1}.
In this case we `again' have a tetrahedral
surface~$\TV^{\prime \prime}$, but its properties are different
from the cases described before.
The same structure also appears in a detuned $1{:}1{:}2$~resonance
and has been described in~\cite{Sadovskii2007}. 
Following the previously introduced terminology, we denote by
$\Reg'$ and~$\Reg''$ the two connected components of~$\Reg$, outside
and inside of~$\TV^{\prime \prime}$, respectively.
Values in $\Reg'$ lift to~$\T^3$.
Values in $\Reg''$ lift to the disjoint union of two~$\T^3$ and
therefore we can again consider two disjoint families $\T^3_A(v)$
and~$\T^3_B(v)$ for $v \in \Reg''$.

We denote by $\FV_h$ the ``upper'' part of $\TV^{\prime \prime}$
with topological boundary the embedding $\DV'$ of $\DV$
to~$\TV^{\prime \prime}$.
Values in $\FV_h$ lift to singular fibres where the families
$\T^3_A(v)$ and $\T^3_B(v)$ intersect along a hyperbolic
lower dimensional invariant torus~$\T^2$.
Two of the vertices of~$\DV'$ correspond to cusp bifurcations,
and they are part of the intersection
$\TV^{\prime \prime} \cap \BV$, while the remaining vertex
(with $\mu = 0$, not on~$\BV$) corresponds to a subcritical
Hamiltonian Hopf bifurcation where the thread~$\CV_{12}^0$
parametrising pinched tori turns into the crease
$\CV_{12}^+ \backslash \CV_{12}^0$ of the ``lower'' part~$\FV_e$
of the tetrahedron, parametrising the now stable normal $3$--mode.

The ``lower'' part $\FV_e$ of~$\TV^{\prime \prime}$ has again
three faces.
Using the same conventions as in \sref{sec:CVCase2}, we denote these
faces by~$\FV_{\plminul}$ and we denote by~$\FV_e^\circ$ their union,
noting that $\FV_e^\circ \cap \BV = \emptyset$.
Values on $\FV_e^\circ$ lift to the disjoint union of a $\T^2$ and
a~$\T^3$.
The difference with previous cases is that it is not the same
$\T^3$--family in $\Reg''$ that shrinks down to~$\T^2$ at each of
these faces whence we no longer can make a `global' choice.
Specifically, let $\T^3_A(v)$ be the family that shrinks down
to~$\T^2$ at the face~$\FV_+$.
Then, the same family $\T^3_A(v)$ shrinks down to $\T^2$ at~$\FV_-$.
However, it is the family $\T^3_B(v)$ that shrinks down to $\T^2$
at~$\FV_0$.
Values on the common topological boundary of $\FV_-$ and~$\FV_+$
(a subset of $\CV_{12}$) lift to the disjoint union of the stable
normal $3$--mode (a smooth~$\T^1$) and a~$\T^3$.

Thus, after emanating from the central equilibrium, the normal
$3$--mode parametrised by $\CV_{12} \backslash \CV_{12}^+$ first
is stable with minimal energy $h_{\min} = 0$ until the value
\begin{displaymath}
   (\mu, \ell, h) \; = \; (0, \ell^*, 0) \;\; \in \;\;
   \CV_{12} \cap \TV^{\prime \prime} \cap \BV
   \enspace , \quad
   \ell^* \in \opin{{-}\lambda^2}{0} \enspace ,
\end{displaymath}
then is parametrised by
$\CV_{12}^+ \backslash \CV_{12}^0 \subseteq \TV^{\prime \prime}$
and loses its stability in the subcritical Hamiltonian Hopf
bifurcation at $\TV^{\prime \prime} \cap \partial \CV_{12}^0$.
The intersection $\TV^{\prime \prime} \cap \BV$ is formed by
$(0, \ell^*, 0)$ and the two values of cusp bifurcations
together with the curve segments $\mathcal{L}_+$
and~$\mathcal{L}_-$ where $\FV_+$ and~$\FV_-$, respectively,
meet with~$\FV_0$.
Values on~$\mathcal{L}_{\pm}$ lift to the disjoint union
$\T^2_A(v) \dot{\cup} \T^2_B(v)$ of two~$\T^2$, one being the
limit of the family $\T^3_B(v)$ at~$\FV_0$ and one being the
limit of the family $\T^3_A(v)$ at~$\FV_{\plmin}$. 
Note that the family~$\T^3_A(v)$ can be smoothly continued through
the face~$\FV_0$ (where $\T^3_B(v)$ shrinks down) to the regular
$\T^3$--fibration over~$\Reg'$, while the family~$\T^3_B(v)$ can be
smoothly continued through $\FV_+$ or~$\FV_-$ (where $\T^3_A(v)$
shrinks down).

In the complement of the normal modes parametrised by~$\CV_{ij}$
the surface~$\BV$ parametrises invariant $2$--tori with minimal
energy.
In particular, the two $\T^2$ parametrised by~$\mathcal{L}_{\pm}$
have the same energy.
Outside of~$\mathcal{L}_{\pm}$ the surface~$\BV$ parametrises a
single family of~$\T^2$.
To the side of~$\mathcal{L}_{\pm}$ `covered' by~$\FV_0$ this is
$\T^2_A(v)$ and to the sides `covered' by~$\FV_{\plmin}$ this is
$\T^2_B(v)$.
In a similar way, when passing through $(0, \ell^*, 0)$ the
$\T^2$--family surrounding the normal $3$--mode is~$\T^2_A(v)$,
while $\T^2_B(v)$ consists of $2$--tori even for critical values
$v = (\mu, \ell, h)$ with $\mu = 0$.
At the cusp value the two families $\T^2_A(v)$ and~$\T^2_B(v)$,
$v \in \mathcal{L}_+$ (or $v \in \mathcal{L}_-$) coincide with
the hyperbolic $\T^2$ parametrised by~$\FV_h$ as the figure~eight
yielding the stable$=$unstable manifolds of the latter has
shrunk to a point.

\begin{figure}
  \centering
  \includegraphics[width=6cm]{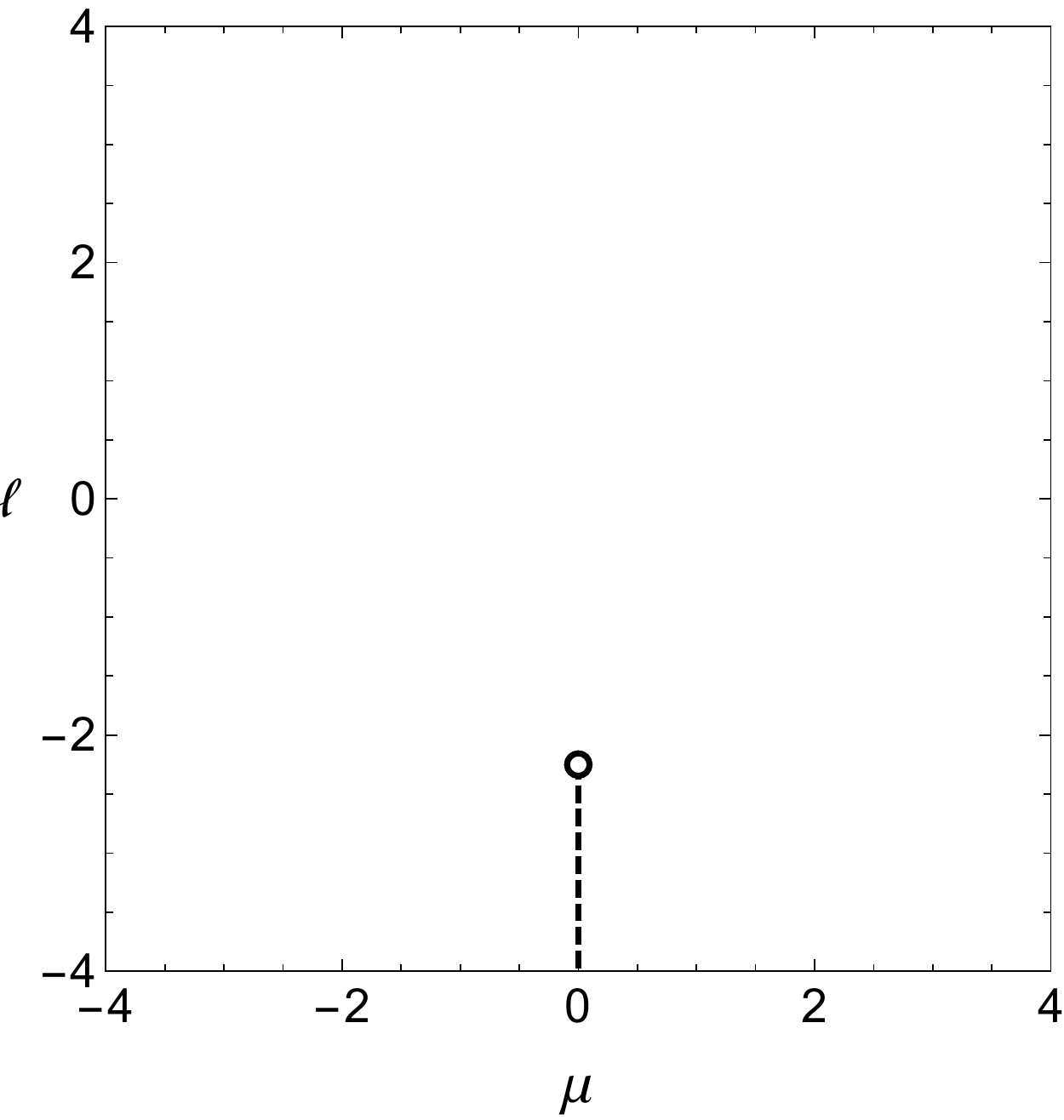}
  \hspace{1cm}
  \includegraphics[width=7cm]{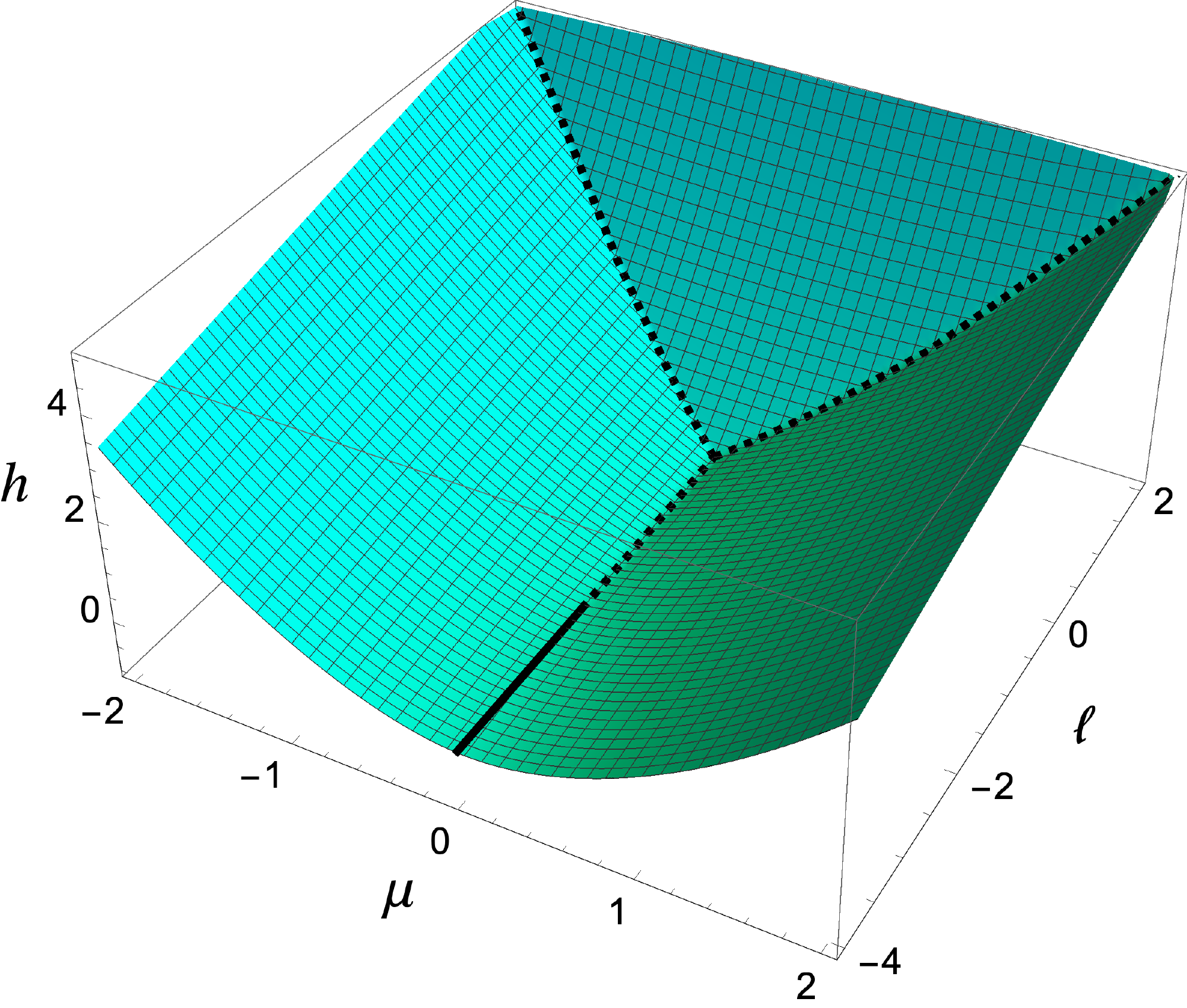}
  \caption{
    Intersection of the plane $\lambda = 1.5$ with the amended
    bifurcation diagram and the corresponding set of critical
    values~$\CV$.
  \label{cv 1<lambda} }
\end{figure}

To better understand the placement of~$\TV^{\prime \prime}$ with
respect to the rest of~$\CV$ and, in particular, to~$\BV$, consider
the straight line $\CV_{12}$ with $(\mu,\ell,h) = (0,\ell,0)$,
$\ell < 0$. 
Then place the tetrahedral structure in such a way that the edge
between $\FV_+$ and~$\FV_-$ is a subset of~$\CV_{12}$ and does not
extend to the origin (where all~$C_{ij}$ meet).
The edges~$\mathcal{L}_{\plmin}$ should then be glued to~$\BV$ while
being distinct from $\CV_{13}$ and $\CV_{23}$, which also belong
to~$\BV$.
Note that $\BV$ is not smooth along~$\mathcal{L}_{\plmin}$, along
$\CV_{13}$ and~$\CV_{23}$, and along the part of $\CV_{12}$ that
belongs on~$\BV$.
The lack of smoothness along $\mathcal{L}_{\plmin}$
and~$\CV_{13,23}$ allows $\TV^{\prime \prime}$ to be glued
with~$\BV$ in the indicated way.

In fact, locally around~$\mathcal{L}_+$ it seems better to view
the surfaces not as part of $\BV$ and~$\TV^{\prime \prime}$, but
as two smooth surfaces parametrising $\T^2_A(v)$ and~$\T^2_B(v)$,
respectively, that intersect along~$\mathcal{L}_+$ (similar
for~$\mathcal{L}_-$).
In the same way, the surface parametrising~$\T^2_B(v)$ is pierced
through by~$\CV_{12}$, at $(0, \ell^*, 0)$ the short line
$(\mu, \ell, h) = (0, \ell, 0)$ does not detach smoothly.

\begin{remark}
Again the transitional case $\delta = 1$ helps to explain the
structure of the set~$\CV$ of critical values.
When passing through $\delta = 1$ the tetrahedral
surface~$\TV^{\prime \prime}$ with its non-empty intersection
with~$\BV$ shrinks down to a single value where $\CV_{12}$
detaches from~$\BV$.
At $\delta = 1$ this value corresponds to a degenerate
Hamiltonian Hopf bifurcation and the set $\CV$ of critical values
already looks like the one depicted in \fref{cv 1<lambda}.
Note that this does not indicate a different type of degenerate
Hamiltonian Hopf bifurcation --- strongly `tilted' planes have the
same kind of transition when passing through one of the other two
degenerate Hamiltonian Hopf bifurcations.
\end{remark}

\paragraph{Case $\boldsymbol{1 < \delta}$.}
For large positive detuning $\delta > 1$ the intersection of the
plane $\lambda = \delta$ with the amended bifurcation diagram
attains its simplest form, see \fref{cv 1<lambda}(left).
The intersection consists only of the straight half line $\mu = 0$,
$\ell \le -\lambda^2$, which starts at a supercritical Hamiltonian
Hopf bifurcation and extends indefinitely.

The set $\CV$ of critical values depicted in
\fref{cv 1<lambda}(right) consists of the surface~$\BV$ where values
lift to~$\T^2$, except along the curves $\CV_{ij}$, $ij=12,13,23$.
In particular, critical values along
$\mathrm{int}(\CV_{ij} \cap \BV)$, $ij=12,13,23$ (dashed lines
in \fref{cv 1<lambda}) correspond to stable normal modes and
where these three meet we have the elliptic equilibrium.
Critical values on $\CV_{12}^0 = \CV_{12}^+$ (the half line
$(0,\ell,0)$, $\ell<-\lambda^2$) correspond to the Cartesian
product of a~$\T^1$ (the normal $3$--mode) and a pinched torus
(its stable=unstable manifold).
The normal $3$--mode loses its stability in a  supercritical
Hamiltonian Hopf bifurcation where the half line~$\CV_{12}$
detaches from~$\BV$ to become $\CV_{12}^0 = \CV_{12}^+$.

\begin{remark}
For $\kappa = 0$ the fibres of the energy-momentum mapping need
not be compact.
Correspondingly, not only the cases with $\delta \geq \frac{1}{2}$
are lost, but also the interpretation for small detuning
$\delta < \frac{1}{2}$ (i.e.\ the passage through $\delta = 0$)
changes.
What remains are the tetrahedral sets $\TV$ and~$\TV^{\prime}$
shrinking down to the origin $(\mu, \ell, h) = (0, 0, 0)$ as
$\delta \nearrow 0$ and $\delta \searrow 0$, respecitvely.
Regular values $v \in \Reg''$ parametrise $\T^3_B(v)$ that shrink
down to elliptic $\T^2_B(v)$, $v \in \FV_e^{\circ}$ surrounding
the three normal modes and to hyperbolic $\T^2_B(v)$, $v \in \FV_h$
which at~$\DV'$ meet the elliptic~$\T^2$ in quasi-periodic
centre-saddle bifurcations.
When passing through the $1{:}1{:}{-}2$~resonance, $\TV$ shrinks
down to the origin and re-emerges as~$\TV^{\prime}$ flipped
upside-down.

In comparison to this, the `final stage' $\delta > 1$ depicted in
\fref{cv 1<lambda} also has three faces of~$\BV$ parametrising
elliptic~$\T^2$, but the `hyperbolic face' has vanished along
with two of the subcritical Hamiltonian Hopf bifurcations, while
the third Hamiltonian Hopf bifurcation has become supercritical.
In particular, $\kappa \neq 0$ provides not only for compact level
sets $\T^3_A(v)$, $v$ a regular value of~$\EM$, but also for
critical values with minimal energy (maximal energy if $\kappa < 0$).
When $\kappa = 0$ higher order terms in the normal
form~\eqref{symmetricdetunedhamiltonian} are needed to decide
what other invariant sets are parametrised by regular values
of~$\EM$ inside and outside of $\TV$ and~$\TV^{\prime}$.
\end{remark}

\section{Monodromy}
\label{Monodromy}

\noindent
In this section we determine the monodromy for the system defined by
the normalized Hamiltonian~$H^{\delta}_{N, L}$ given
in~\eqref{symmetricdetunedhamiltonian}.
We briefly recall some basic facts about monodromy in integrable
Hamiltonian systems as they apply to the system under study here.

Consider a closed path~$\gamma$ in the set $\Reg$ of regular values
of the energy-momentum mapping~$\EM$.
Then $\EM^{-1}(\gamma)$ is a, possibly non-trivial, $\setT^3$--fibre
bundle over~$\gamma$.
The non-triviality of the bundle can be expressed through the
glueing mapping $\psi : \setT^3 \longrightarrow \setT^3$ of the bundle
over the circle~$\gamma$.
The mapping $\psi$ induces a mapping $\psi_*$ on
$H_1(\setT^3) \simeq \setZ^3$ which, fixing a basis of $H_1(\setT^3)$,
can be written as a matrix of~$\SL(3,\setZ)$
which is called the \emph{monodromy matrix} of the bundle.
The monodromy matrix depends only on the homotopy class~$[\gamma]$ of
a path $\gamma$ in~$\Reg$ and we thus denote the monodromy matrix
by~$M_{[\gamma]}$.
The mapping
\begin{displaymath}
   M \; : \;\; \pi_1(\Reg) \;\; \longrightarrow \;\; \SL(3, \setZ)
\end{displaymath}
that assigns to each homotopy class $[\gamma]$ of~$\Reg$ the
corresponding monodromy matrix~$M_{[\gamma]}$ is called the
\emph{monodromy mapping} of the system.

Note that not the energy-momentum mapping~$\EM$ itself, but its
set~$\Reg$ of regular values is the main ingredient of the monodromy
mapping~$M$.
This allows us to use any diffeomorphism of~$\R^3$ to transform~$\Reg$
into a form suitable for our considerations; in particular we may use
any of the alternatives of $\EM = (N, L, H^{\delta}_{N, L})$
discussed at the beginning of \sref{Critical values}.
Thus, we again replace $H^{\delta}_{N, L}$ by~$\cH_{\lambda}$ as
defined in~\eqref{reducedhamiltonian} but considered as a function
$\cH_{\lambda} : \R^6 \longrightarrow \R$ and now also replace $L$ by
$J = \frac{1}{2}(N+L)$, i.e.\ we work with
\begin{displaymath}
   \EM  \;\; = \;\; (N, J, \cH_{\lambda})
\end{displaymath}
where $\lambda$ depends on the values $\mu$ of~$N$ and $\iota$ of~$J$
through
\begin{displaymath}
   \lambda \;\; = \;\; \delta \; + \; (\lambda_1 - \lambda_2) \mu
   \; + \; 2 \lambda_2 \iota
\end{displaymath}
which replaces~\eqref{linearparameterdependence}.
Because of the existence of the effective global $\T^2$--action~$\Phi$
generated by $\X{N}$ and~$\X{J}$, we choose as a basis for expressing
the monodromy matrix a basis of $H_1(\T^3) \simeq \Z^3$
given by the homology cycles $g_N$, $g_J$ represented by periodic
orbits of $\X{N}$ and $\X{J}$ on~$\T^3$ and completed by any other
homology cycle~$g$ with the property that $(g_N,g_J,g)$ form a basis.

\begin{remark}
  The reason we need to consider here the effective action generated
  by $\X{N}$ and~$\X{J}$, rather than the non-effective action
  generated by $\X{N}$ and~$\X{L}$, manifests here.
  The cycles $g_N$ and~$g_L$ generated by the flows of $\X{N}$
  and~$\X{L}$ cannot be combined with any other cycle to form a basis
  of~$H_1(\T^3)$.
  In particular, $g_J$ cannot be expressed as an integer linear
  combination of $g_N$ and~$g_L$, since the definition of~$J$ implies
  $g_L + g_N = 2 g_J$.
\end{remark}

\noindent
Then the monodromy matrix for~$[\gamma]$ can be written in such basis
as
\begin{displaymath}
  M_{[\gamma]} \;\; = \;\;
  \begin{pmatrix}
    1 & 0 & m^{[\gamma]}_N \\
    0 & 1 & m^{[\gamma]}_J \\
    0 & 0 & 1
  \end{pmatrix}
  \enspace .
\end{displaymath}
The given monodromy matrix signifies that parallel transport of $g_N$
and $g_J$ along $\gamma$ gives $g_N$ and $g_J$ respectively, while
parallel transport of $g$ gives
$g + m^{[\gamma]}_N g_N + m^{[\gamma]}_J g_J$.

\begin{remark}
  The cycle $g$ completing the basis $(g_N,g_J,g)$ is not uniquely
  determined: any cycle $g' = g + k_N g_N + k_J g_J$ also defines a
  basis of~$H_1(\T^3)$ with the same orientation.
  However, different choices of~$g$ do not affect the monodromy matrix
  since, if $g$ is parallel transported along~$\gamma$ to some cycle
  $\hat{g} = g + m^{[\gamma]}_N g_N + m^{[\gamma]}_J g_J$, then $g'$
  is parallel transported to
  $\hat{g}' = g^{\prime} + m^{[\gamma]}_N g_N + m^{[\gamma]}_J g_J$
  (because $g_N$ and $g_J$ do
  \emph{not} change under parallel transport).
\end{remark}

\noindent
We define the \emph{monodromy vector} for $[\gamma]$ by
\begin{displaymath}
   \vec{m}^{[\gamma]} \;\; = \;\; (m^{[\gamma]}_N, m^{[\gamma]}_J)
   \;\; \in \;\; \Z^2 \enspace .
\end{displaymath}
The mapping
\begin{displaymath}
   \vec{m} \; = \; (m_N,m_J) \;\; \mapsto \;\; M(\vec{m}) \; = \;
   \begin{pmatrix}
      1 & 0 & m_N \\
      0 & 1 & m_J \\
      0 & 0 & 1
   \end{pmatrix}
\end{displaymath}
exhibits a group isomorphism between $(\Z^2,+,\vec{0})$ and the
subgroup of $(\SL(3,\Z),\cdot,I)$ that is the image of~$\pi_1(\Reg)$
under the monodromy mapping.
In particular,
\begin{displaymath}
  M(\vec{0}) \; = \; I
  \enspace , \quad
  M(\vec{m}_1) M(\vec{m}_2) \; = \; M(\vec{m}_1 + \vec{m}_2)
  \enspace , \quad
  M(\vec{m})^{-1} \; = \; M(-\vec{m})
  \enspace .
\end{displaymath}
Therefore, the existence of the global $\T^2$--action~$\Phi$ implies
that, even though $\pi_1(\Reg)$ may not be abelian, its image under
the monodromy mapping is an abelian subgroup of~$\SL(3,\Z)$.
Note that if $\pi_1(\Reg)$ is generated by $k$ not necessarily
commuting paths $[\gamma_1], \dots, [\gamma_k]$, then we can expand
\begin{displaymath}
   [\gamma] \;\;\ = \;\;
   \prod_{j=1}^{K} \, \prod_{i=1}^k \; [\gamma_i]^{a_{i,j}}
\end{displaymath}
into $K$~factors, i.e.\ with $\sum_{i=1}^k a_{i,j} = 1$,
$j = 1, \ldots, K$ and the corresponding monodromy
vector~$\vec{m}^{[\gamma]}$ is given by
\begin{displaymath}
   \vec{m}^{[\gamma]} \;\; = \;\; \sum_{i=1}^k a_i \vec{m}^{[\gamma_i]}
   \enspace ,
\end{displaymath}
where $a_i = \sum_{j=1}^K a_{i,j}$.
It is therefore sufficient to compute the monodromy vector for the
generators of~$\pi_1(\Reg)$.
We do this computation in \sref{monodromy computation} below using
the approach of~\cite{Efstathiou2017a}.

\subsection{Computation of monodromy}
\label{monodromy computation}

\noindent
We compute monodromy for the cases for which the sets $\CV$ of
critical values were described in \sref{sets of critical values}.
Recall that the system has always three curves of critical values
parametrising the three normal modes $\CV_{23}$, $\CV_{13}$
and~$\CV_{12}$.
The monodromy of the ramified torus bundle defined by the
energy-momentum mapping~$\EM$ is largely determined by these
three curves and whether parts of them are transversally isolated
in the image of~$\EM$ or whether they are embedded in some
two-dimensional surface~$\BV$ of critical values.
These two possibilities also largely determine~$\pi_1(\Reg)$.
In all cases considered in \sref{sets of critical values}, the
fundamental group~$\pi_1(\Reg)$ is non-trivial although its
structure is not always the same.
Moreover, in some cases monodromy can be meaningfully defined for
loops~$\gamma$ that contain critical values of~$\EM$,
see~\cite{EB13} for more details.

\begin{remark}
  The monodromy of a similar $n$--degree-of-freedom Hamiltonian
  system where $n$~threads of critical values join at the origin
  has been studied in~\cite[Example~1.2]{Gross1998}.
  The monodromy of that system is determined using a convenient
  representation where monodromy acts trivially to $n-2$ cycles
  generating the $\setT^n$--fibre homology.
\end{remark}

\subsubsection{Case $\boldsymbol{\delta = \lambda_1 = \lambda_2 = 0}$}
\label{sec:MonCase1}

\noindent
We first determine the monodromy mapping for the resonant system
without detuning, cf.\ \sref{sec:CVCase1}.
Part of this computation has been given in~\cite{Efstathiou2017a},
with the exception of an overall ``sign'' of the monodromy vector
(which depends on a careful choice of orientations and was outside
the scope of~\cite{Efstathiou2017a}).
In this case $\pi_1(\Reg)$ is isomorphic to the free product
$\setZ * \setZ$ and is generated by two closed paths: one
path~$[\gamma_1]$ encircling the thread~$\CV_{23}^0$ and one
path~$[\gamma_2]$ encircling the thread~$\CV_{13}^0$.
The threads $\CV_{23}^0$ and~$\CV_{13}^0$ are oriented so
that they start at infinity and point to the origin.
Then the paths $[\gamma_1]$ and~$[\gamma_2]$ are oriented so
that they follow the right-hand rule with respect to $\CV_{23}^0$
and~$\CV_{13}^0$, respectively.
Note that $[\gamma_1]$ and~$[\gamma_2]$ do not commute.
We further define $[\gamma_3] = [\gamma_1]^{-1}\cdot[\gamma_2]^{-1}$.
Such a homotopy class is represented by a path that
encircles~$\CV_{12}^0$.
Following the same orientation conventions, $[\gamma_3]$ is positively
oriented.
Since $[\gamma_1] \cdot [\gamma_2] \cdot [\gamma_3] = 1$ we conclude that
\begin{equation}
\label{eq:relation monodromy vectors}
   \vec{m}^{[\gamma_1]} \; + \; \vec{m}^{[\gamma_2]}
   \; + \; \vec{m}^{[\gamma_3]} \;\; = \;\; \vec{0}
   \enspace .
\end{equation}
Given equation~\eqref{eq:relation monodromy vectors}, we here first
compute $\vec{m}^{[\gamma_2]}$ and $\vec{m}^{[\gamma_3]}$, from which
we then deduce~$\vec{m}^{[\gamma_1]}$.

To compute $\vec{m}^{[\gamma_2]}$ we consider a specific
representative of~$[\gamma_2]$ on a plane $N = \mu < 0$.
Using $(J, \cH_{\lambda})$ as co-ordinates on the plane $N = \mu < 0$,
the thread~$\CV_{13}^0$ intersects the plane at
$(J, \cH_{\lambda}) = (0,\frac12\mu^2) =: c_2$.
Then $[\gamma_2]$ can be represented by a circle~$C_2$ that winds once
counterclockwise around $c_2$ with respect to the oriented
co-ordinates~$(J, \cH_{\lambda})$ on the $N = \mu < 0$ plane and which
bounds a disk~$U_2$ on this plane.
The disk~$U_2$ contains one $\Phi$--orbit with non-trivial
isotropy~$\T^1_J$ generated by~$\X{J}$.
Following~\cite{Efstathiou2017a} we note that in the basis
$(\X{N}, \X{J})$ we can write $\X{J} = (0,1)$ and therefore the
corresponding monodromy vector should be
$\vec{m}^{[\gamma_2]} = \pm (0,1)$, where the sign must be determined.
For a positively oriented (counter-clockwise) path on the oriented
$(J, \cH_{\lambda})$--plane the sign is~$+1$ if the $\Phi$--orbit with
non-trivial isotropy is \emph{positive} in the sense
of~\cite{Efstathiou2017a} and is~$-1$ otherwise.
The $\T^1_J$--action acts on the reduced space
\begin{displaymath}
   N^{-1}(\mu<0)\quotient{\T^1_N} \;\; \simeq \;\; \C^2
\end{displaymath}
(where $\T^1_N$ is generated by~$\X{N}$) as
\begin{displaymath}
   (z_1 z_2, z_3) \;\; \mapsto \;\;
   (\ee^{2 \pi i t} z_1 z_2, \ee^{-2 \pi i t} z_3)
   \enspace ,
\end{displaymath}
with complex co-ordinates $z_1 z_2$ and~$z_3$ on $N^{-1}(\mu<0) / \T^1_N$.
We note here that the two co-ordinates $z_1 z_2$ and~$z_3$ define an
orientation that coincides with the one induced by symplectic reduction.
Since $\T^1_J$ has weights $1{:}0{:}{-}1$ the $\Phi$--orbit is positive
(it would have been negative if the weights were $1{:}0{:}1$) and we
conclude that
\begin{displaymath}
   \vec{m}^{[\gamma_2]} \;\; = \;\; (0,1) \enspace .
\end{displaymath}
We can repeat the same argument to compute~$\vec{m}^{[\gamma_3]}$.
The isotropy in this case is~$\T^1_N$ and in the basis
$(\X{N}, \X{J})$ we have $\vec{m}^{[\gamma_3]} = \pm (1,0)$.
The $\T^1_N$--action on the reduced space
$J^{-1}(\iota<0) / \T^1_J$ reads as
\begin{displaymath}
   (z_1 z_3, z_2) \;\; \mapsto \;\;
   (\ee^{2 \pi i t} z_1 z_3, \ee^{-2 \pi i t} z_2)
\end{displaymath}
whence the corresponding $\Phi$--orbit is again positive.
The only difference is that $[\gamma_3]$ is now represented by a
negatively oriented (clockwise) circle~$C_3$ on the plane
$J = \iota < 0$.
Therefore
\begin{displaymath}
   \vec{m}^{[\gamma_3]} \;\; = \;\; (-1,0)
\end{displaymath}
and with~\eqref{eq:relation monodromy vectors} finally
\begin{displaymath}
   \vec{m}^{[\gamma_1]} \;\; = \;\; (1,-1) \enspace .
\end{displaymath}
Note that in these computations of the monodromy vectors the specific
form of the Hamiltonian~$\cH_{\lambda}$ is not used.
This implies that in subsequent cases, where the parameters of the
Hamiltonian change, the monodromy vectors remain the same for paths
$[\gamma_k]$, $k=1,2,3$, provided that such paths can be defined
(see \sref{sec:MonCases345} below).

\subsubsection{Case $\boldsymbol{\delta < 0,\, \lambda_1 = \lambda_2 = 0}$}
\label{sec:MonCase2}

\noindent
In this case $\Reg$ consists of two connected components: one
outside~$\TV$ denoted~$\Reg'$, and one inside~$\TV$ denoted~$\Reg''$.
The fundamental group $\pi_1(\Reg')$ is isomorphic to $\Z * \Z$ and
the whole discussion from the previous subsection can be transferred
almost verbatim here.
Indeed, for closed paths in~$\Reg'$ the tetrahedral surface~$\TV$
together with its interior~$\Reg''$ is indistinguishable from the
mere point value $(\mu, \iota, h) = (0, 0, 0)$ where the three
threads meet when $\delta = \lambda_1 = \lambda_2 = 0$.
The fundamental group~$\pi_1(\Reg'')$ is trivial and thus its
image under the monodromy mapping is the identity.

However, monodromy is also meaningful~\cite{EB13} for paths~$\gamma$
that do not lie completely in $\Reg'$ or~$\Reg''$, but pass from one
connected component to the other through~$\FV_e^\circ$.
Then $\EM^{-1}(\gamma)$ is the disjoint union of a $\T^3$--bundle
for which we can define monodromy and another manifold that can be
ignored.
We use this to extend the monodromy mapping~$M$ to~$\pi_1(\Reg^+)$,
where $\Reg^+$ contains both the interior and the exterior of~$\TV$,
but not the face $\FV_h$ of~$\TV$.
In this extension the `ignorable' manifold does not contribute
to~$M$, reflecting that $\pi_1(\Reg'')$ is trivial.
For example, if $\gamma$ enters and exits~$\Reg''$ through the same
face~$\FV_{\plminul}$, then we can reduce the whole
$\T^2$--action~$\Phi$; the family $\T^3_B(v)$ becomes a cylinder
$\T^1_B(v)$ shrinking to points at the two values
$v \in \FV_{\plminul}$, revealing the reduced ignorable manifold to
be~$\setS^2$ and the ignorable manifold itself to be diffeomorphic to
$\setS^2 \times \T^2$.
If $\gamma$ enters and exits~$\Reg''$ once through different faces
of~$\FV_e^\circ$, then only a $\T^1$--subaction of~$\Phi$ can be
regularly reduced (the normal mode `between' the two faces
of~$\FV_e^\circ$ having non-trivial isotropy) and the relevant
$\T^2_B(v)$ shrink to $\T^1_B(v)$, $v \in \FV_e^\circ$ forming
an~$\setS^3$; the ignorable manifold is diffeomorphic to
$\setS^3 \times \T^1$.

Note that we may also consider paths~$\gamma$ that pass through one
of the parts of the~$\CV_{ij}$ that form the common topological
boundaries of the~$\FV_{\plminul}$.
Indeed, while this may allow for e.g.\ a deformation of
$\setS^2 \times \T^2$ into $\setS^3 \times \T^1$, this manifold is
then ignored anyway.

We therefore let $\Reg^+ := \Reg \cup \FV_e$ and define a monodromy
mapping $M : \pi_1(\Reg^+) \longrightarrow \SL(3,\Z)$
by considering only the monodromy of the $\T^3$--bundle connected
component of $\EM^{-1}(\gamma)$ for $\gamma$ in~$\Reg^+$.
The fundamental group~$\pi_1(\Reg^+)$ is isomorphic to~$\pi_1(\Reg)$
and is also generated by the same paths $[\gamma_1]$ and~$[\gamma_2]$
for which we found in \sref{sec:MonCase1} that
$\vec{m}^{[\gamma_1]} = (1,-1)$ and $\vec{m}^{[\gamma_2]} = (0,1)$.
Therefore, the results for the monodromy mapping
$\pi_1(\Reg) \longrightarrow \SL(3,\Z)$ for the case
$\delta = \lambda_1 = \lambda_2 = 0$ apply without any further
modifications to determine the monodromy mapping
$M : \pi_1(\Reg^+) \longrightarrow \SL(3,\Z)$.
The extension of $M$ from $\pi_1(\Reg')$ to~$\pi_1(\Reg^+)$ means
that we may interpret the three threads (which meet at the single
value $(\mu, \iota, h) = (0, 0, 0)$ for
$\delta = \lambda_1 = \lambda_2 = 0$) as meeting at~$\FV_h$ instead
of meeting at $\TV \cup \Reg''$.

\subsubsection{Case $\boldsymbol{\delta > 0,\, \lambda_1 = \lambda_2 = 0}$}
\label{sec:MonCases345}

\noindent
Following the structure of the discussion in \sref{sec:CVCases345},
we again distinguish the three cases of small, intermediate and
large positive detuning separated by $\delta = \frac{1}{2}$ and
$\delta = 1$.

\paragraph{Case $\boldsymbol{0 < \delta < \frac{1}{2}}$.}
This case is exactly the same as the case $\delta < 0$.
We again define $\Reg^+ = \Reg \cup \FV_e$ with the only difference
being that now $\FV_e$ is the ``lower'' part of~$\TV^{\prime}$.
The fundamental group $\pi_1(\Reg^+)$ remains isomorphic to $\Z * \Z$
and the rest of the discussion goes through without any other changes.

\paragraph{Case $\boldsymbol{\frac{1}{2} < \delta < 1}$.}
Here the set of regular values consists again of two connected
components $\Reg'$ and~$\Reg''$ which are, respectively, outside
and inside~$\TV^{\prime \prime}$.
However, there are two important changes here with respect to
previous cases.
First, $\pi_1(\Reg')$ is isomorphic to~$\Z$.
It is generated by $[\gamma_3]$ which winds once around the thread
of critical values $\CV_{12}$ and for which we computed in
\sref{sec:MonCase1} that the monodromy vector is
$\vec{m}^{[\gamma_3]} = (-1,0)$.

Second, we can no longer consider paths that enter $\Reg''$ through
one of the sides $\FV_\pm$ and exit through~$\FV_0$ or vice versa.
For such paths~$\gamma$, the pre-image~$\EM^{-1}(\gamma)$ does not
contain a $\T^3$--bundle over a circle and therefore we cannot define
monodromy.
However, we can still consider paths that enter and exit $\Reg''$
through the union of $\FV_+$ and~$\FV_-$ with their common topological
boundary, and paths that both enter and exit $\Reg''$ through~$\FV_0$.
Recall that the topological boundary of~$\FV_0$ consists of the
common boundary with~$\FV_h$ and the two curve segments
$\mathcal{L}_{\plmin}$ on~$\BV$.
The space of such paths together with paths that lie entirely
in $\Reg'$ or~$\Reg''$ is generated by~$[\gamma_3]$ and therefore
the corresponding homotopy structure is isomorphic to~$\Z$.

\paragraph{Case $\boldsymbol{1 < \delta}$.}
In this case $\pi_1(\Reg)$ is isomorphic to~$\Z$ and generated
by~$[\gamma_3]$.
The monodromy then is completely determined by the monodromy vector
$\vec{m}^{[\gamma_3]} = (-1,0)$.
One may think of the passage of~$\delta$ through $\delta = 1$ from
$\delta > 1$ to $\delta < 1$ as replacing the value
$(\delta, 0, -\delta^2)$, $\delta > 1$ by
$\FV_h \cup \mathcal{L}_{\plmin}$ conditional on paths~$\gamma$ not
encircling $\mathcal{L}_+$ or $\mathcal{L}_-$ (i.e.\ paths that
enter the interior $\Reg''$ of~$\TV^{\prime \prime}$ through~$\FV_0$
exit $\Reg''$ through~$\FV_0$ as well).

\subsection{Global monodromy}
\label{globalmonodromy}

\noindent
Letting the detuning parameter $\delta = \lambda$ (i.e.\
$\lambda_1 = \lambda_2 = 0$) vary we may consider
$\EM = (N, J, \cH_{\lambda})$ as a mapping
\begin{displaymath}
\begin{array}{ccc}
   \R^6 \times \R^1 & \longrightarrow & \R^3 \times \R^1 \\
   (q, p, \lambda) & \mapsto & (\mu, \iota, h, \lambda)
\end{array} \enspace ,
\end{displaymath}
thereby stacking all $\delta$--values together to let the
parametrised sets of critical values form a single subset of~$\R^4$.
This results in a single monodromy mapping assembling the
$\delta$--family of monodromy mappings.

\begin{remark}
\label{detuningbyaction}
This approach is less theoretical than it seems since in
applications the detuning parameter~$\delta$ may easily arise
as the value of some additional action~$D$, see \sref{Conclusions}.
\end{remark}

\section{Conclusions}
\label{Conclusions}

\noindent
An integrable\footnote{Admitting both the axial symmetry generated
by the third component~$N$ of the angular momentum and the oscillator
symmetry generated by the quadratic part~$L$ of the Hamiltonian.}
Hamiltonian system in three degrees of freedom with an equilibrium
in $1{:}1{:}{-}2$~resonance, has a set~$\CV$ of critical values of
the energy-momentum mapping $\EM = (N, L, H_{N, L}^0)$ of the normal
form~\eqref{symmetricdetunedhamiltonian} depicted in
\fref{fig:cvlambda=0}, with three threads parametrising unstable
normal modes meeting at the value of the equilibrium.
This results in a monodromy mapping
\begin{equation}
\label{undetunedmonodromy}
   \Z * \Z \;\; \longrightarrow \;\; \SL(3,\Z)
\end{equation}
with commutative image isomorphic to~$\Z^2$ spanned by the monodromy
vectors $\vec{m}^{[\gamma_1]} = (1, -1)$ and
$\vec{m}^{[\gamma_2]} = (0, 1)$.
A small amount of detuning leaves~\eqref{undetunedmonodromy}
unchangend, but stabilizes the detuned resonant equilibrium and with
it the three normal modes, deforming the set~$\CV$ of critical values
of~$\EM$ into those depicted in figures \ref{cv lambda<0}
and~\ref{cv 0<lambda<1/2} and thereby turning the monodromy
into island monodromy.
For large detuning, depending on the relative size (and sign) of
third and $4$th order terms in the normal
form~\eqref{symmetricdetunedhamiltonian}, the monodromy mapping
becomes
\begin{equation}
\label{largedetunedmonodromy}
   \Z \;\; \longrightarrow \;\; \SL(3,\Z)
\end{equation}
with $\pi_1(\Reg)$ generated by a loop around the single thread in
\fref{cv 1<lambda}.
The mapping~\eqref{largedetunedmonodromy} describes the (island)
monodromy already for intermediate values of the detuning where the
set~$\CV$ of critical values of~$\EM$ has the more complicated form
depicted in \fref{cv 1/2<lambda<1}.
As discussed in~\cite{rink04, BCFT07} monodromy remains meaningful
for the non-integrable
Hamiltonian~\eqref{originaldetunedhamiltonian} approximated by the
integrable normal form~\eqref{symmetricdetunedhamiltonian} and
persists even under small perturbations that destroy the axial
symmetry.
However, for the latter perturbations the non-degeneracy conditions
mentioned in remark~\ref{lowordernormalinternal} that exclude low
order normal-internal resonances become important.

Our choice to retain in the normal
form~\eqref{symmetricdetunedhamiltonian} next to the third order
term~$X$ also the terms of order~$4$ had the dynamical consequence
that all motions remained bounded --- mostly spinning densely around
invariant $3$--tori --- but furthermore turned the bifurcation
diagram of \fref{fig:bif-kappa-0-3d} into the one depicted in
\fref{fig:3dpictureofbifurcationset}.
This resulted in additional bifurcations --- compare
table~\ref{tbl:bif-kappa-0} with table~\ref{tbl:bifurcations} ---
which lead to the different type~\eqref{largedetunedmonodromy} of
monodromy for large and intermediate detuning.
To actually prove that the phenomena discovered in a normal form can
also be observed in the original system one often uses a scaling
that zooms in on smaller and smaller neighbourhoods of the
equilibrium.
In the present situation this would make the $4$th order terms
smaller and smaller and correspondingly already intermediate detuning
larger and larger, see~\eqref{kappa=1 scaling}.
In the similar situation of periodic orbits in normal-internal
resonance the semi-global approach in~\cite{hh07} to the
$1{:}3$~resonance reveals the hyperbolic $3$--periodic orbit that
causes the transitional instability of the initial periodic orbit
at the $1{:}3$~resonance to undergo a periodic centre-saddle
bifurcation for a parameter value nearby the $1{:}3$~resonance;
a phenomenon observed in many applications.
We therefore expect that also the cusp and supercritical
Hamiltonian Hopf bifurcations emanating from the degenerate
Hamiltonian Hopf bifurcations do accompany $1{:}1{:}{-}2$~resonances
where these occur.

In order to interpret the `external' detuning as an internal parameter
one can study periodic orbits in four degrees of freedom instead of
equilibria in three degrees of freedom.
Indeed, while the latter are generically isolated the former form
$1$--parameter families, parametrised by the action~$D$ conjugate to
the angle along the periodic orbit (for non-zero Floquet exponents
one may even parametrise by the energy).
Let us impose axial symmetry and assume that the Floquet exponents
encounter a $1{:}1{:}{-}2$~resonance.
Then it is generic for the value $\delta$ of~$D$ to detune the
resonance.
Normalizing with respect to the periodic motion --- possible under
non-degeneracy conditions that exclude normal-internal resonances
between the normal frequencies and the period --- then allows to
reduce the $\T^1$--action generated by~$D$ with reduced Hamiltonian
in three degrees of freedom of the
form~\eqref{originaldetunedhamiltonian}.
See also remark~\ref{detuningbyaction} in \sref{globalmonodromy}.

One should keep in mind that even in axially symmetric Hamiltonian
systems it is {\em not} generic for the $1$--parameter families of
periodic orbits to encounter a normal $1{:}1{:}{-}2$~resonance.
The reason is that adding $\beta N$
to~\eqref{originaldetunedhamiltonian} detunes the
$1{:}1$~subresonance, see remark~\ref{addingamultiple}.
To account for this second parameter one could study invariant
$2$--tori (with their two actions $D$ and~$B$ conjugate to the toral
angles acting as two internal parameters) in an integrable system of
five degrees of freedom and then reduce the $\T^2$--symmetry along
the $\T^2$--tori to three degrees of freedom, resulting in a
(relative) equilibrium in $1{:}1{:}{-}2$~resonance.
Hoewver, when breaking the symmetries of the integrable system
(the above $\T^2$--symmetry, the $\T^1$--symmetry generated by~$L$
and the axial symmetry generated by~$N$) the $2$--parameter family
of invariant $2$--tori needs {\sc kam}~theory to persist and a
single torus in normal $1{:}1{:}{-}2$~resonance may disappear in
a resonance gap (opened by the necessary Diophantine conditions
on the two internal frequencies).

It is only in six (or more) degrees of freedom that families of
invariant lower dimensional tori with three (or more) internal
frequencies may encounter a normal $n_1{:}n_2{:}n_3$~resonance in
such a way that the normally resonant tori even after perturbation
away from integrability form a non-empty Cantor family parametrised
by a Cantor set of dimension one (or more).
In this way the detuning does become one of the internal parameters
and the phenomena of the previous sections do persistently occur
in six or more degrees of freedom.
For the $1{:}1{:}{-}2$~resonance the occurence co-dimension
increases because of the $1{:}1$~subresonance, again see
remark~\ref{addingamultiple}, and one needs at least eight
degrees of freedom.
The non-integrable $1{:}1{:}{-}2$~resonance, just like its definite
counterpart, is a rather degenerate phenomenon.

\appendix

\section{Proof of Proposition~\ref{description of bifurcation set}}
\label{proof of bifurcation set}

\noindent
While the proposition is formulated for $\kappa = 1$, the value
actually used in figures \ref{fig:3dpictureofbifurcationset}
and~\ref{fig:bif01} and in table~\ref{tbl:bifurcations}, we give
the proof here for general $\kappa \neq 0$ to provide for
complete formulas.
Lemma~\ref{lemma bd char} allows us to compute the bifurcation
diagram of the system via the triple roots of $F(R)$ that lie in
$[R_{\min},\infty[$.
We obtain such roots by factorizing $F(R)$ as
\begin{equation}
\label{factorizing}
   F(R) \;\; = \;\; \frac{\kappa^2}{4} (R - a)^3 (R - b)
   \enspace , 
\end{equation}
where $a$ is the sought out triple root and $b \in \setR$ is the
remaining root of~$F(R)$.
Comparing coefficients of powers of~$R$ in the two expressions
\eqref{F(R)} and~\eqref{factorizing} for~$F(R)$ we obtain the relations
\begin{subequations}
\label{coefficients-triple-roots}
\begin{align}
   - \kappa^2 a^3 b \; + \; 4 h^2 \; - \; 4 \mu^2 \ell
   & \;\; = \;\; 0 \label{coefficients-triple-roots:a}\\
   \kappa^2 a^3 \; + \; 3 \kappa^2 a^2 b \; - \; 8 h \lambda \; + \; 4 \mu^2
   & \;\; = \;\; 0 \label{coefficients-triple-roots:b}\\
   - 3 \kappa^2 a^2  \; - \; 3 \kappa^2 a b \; - \; 4 \kappa h \; + \;
   4 \lambda ^2 \; + \; 4 \ell
   & \;\; = \;\; 0 \label{coefficients-triple-roots:c}\\
   3 \kappa^2 a  \; + \; \kappa^2 b \; + \; 4 \kappa \lambda \; - \; 4
   & \;\; = \;\; 0 \label{coefficients-triple-roots:d} \enspace .
\end{align}
\end{subequations}
Solving
\eqref{coefficients-triple-roots:b}--\eqref{coefficients-triple-roots:d}
for $b$, $\ell$ and $h$ we find
\begin{subequations}
\label{solved-triple-roots}
\begin{align}
   b & \;\; = \;\; 4\kappa^{-2} \; - \; 3 a \; - \; 4 \kappa^{-1} \lambda
\label{solved-triple-roots:a}\\
   2 \lambda \ell & \;\; = \;\; - 2 \kappa^3 a^3 \; - \; 6 \kappa^2 a^2 \lambda
   \; + \; 3 \kappa a^2 \; - \; 6 \kappa a  \lambda^2 \; + \; 6 a \lambda
   \; - \; 2 \lambda ^3 \; + \; \kappa  \mu ^2
\label{solved-triple-roots:b}\\
   2 \lambda h & \;\; = \;\; \mu^2 \; + \; 3 a^2 \; - \; 2 \kappa^2 a^3
   \; - 3 \; \kappa a^2 \lambda
\label{solved-triple-roots:c}
\end{align}
\end{subequations}
which for $\lambda \neq 0$ yields an explicit parametrisation by
$\lambda$ and~$a$.
The remaining equation~\eqref{coefficients-triple-roots:a} becomes
\begin{displaymath}
   \frac{1}{\lambda^2} Q(\mu) \;\; = \;\; 0 \enspace ,
\end{displaymath}
where
\begin{eqnarray*}
   Q(\mu) & = & \left( 1 - 2 \kappa \lambda \right) \mu^4  \\
   & & \!\!\!\! \mbox{} + \; \left( 4 \kappa^3 a^3 \lambda - 4 \kappa^2 a^3
   + 12 \kappa^2 a^2 \lambda^2 - 12 \kappa a^2 \lambda + 6 a^2
   + 12 \kappa a \lambda^3 - 12 a \lambda ^2 + 4 \lambda^4 \right) \mu^2  \\
   & & \!\!\!\! \mbox{} + \; \left( 4 \kappa^4 a^6 + 12 \kappa^3 a^5 \lambda
   - 12 \kappa^2  a^5 + 12 \kappa^2  a^4 \lambda ^2 - 18 \kappa  a^4 \lambda
   + 9 a^4 + 4 \kappa  a^3 \lambda^3 - 4 a^3 \lambda^2 \right)
\end{eqnarray*}
is quadratic in $\mu^2$ for $\lambda \ne \Frac{1}{2\kappa}$.
\vspace{-5pt}
This makes $\lambda = \Frac{1}{2\kappa}$ a special case, next to
$\lambda = 0$.
Below we shall first check these two special cases before treating
\vspace{-5pt}
the general cases $\lambda < 0$, $0 < \lambda < \Frac{1}{2\kappa}$
and $\lambda > \Frac{1}{2\kappa}$.

More degenerate than a triple root of~$F$ is having a quadruple root,
i.e.\ $b=a$ in~\eqref{factorizing}.
From~\eqref{solved-triple-roots:a} then follows
\begin{equation}
\label{quadrupleroot}
   a \;\; = \;\; \frac{1}{\kappa^2} \; - \; \frac{\lambda}{\kappa}
   \enspace .
\end{equation}
As~\eqref{F(R)} is a polynomial of degree~$4$ we have
$F^{(4)} \equiv 6 \kappa^2$ and in particular $F^{(4)}(a) \neq 0$
--- the quadruple root~\eqref{quadrupleroot} is not of order~$5$
or higher.
In \sref{degenerateHHb} we have seen that for $a = R_{\min}$ this
yields the degenerate Hamiltonian Hopf bifurcations at
\begin{displaymath}
   \left( \lambda, \mu, \ell \right) \; = \; \left( \frac{1}{2 \kappa},
   \frac{\pm 1}{2 \kappa^2}, \frac{1}{2 \kappa^2} \right)
   \quad \mbox{and} \quad
   \left( \lambda, \mu, \ell \right) \; = \;
   \left( \frac{1}{\kappa}, 0, \frac{-1}{\kappa^2} \right)
\end{displaymath}
and from lemma~\ref{lemma bd char} we conclude that for
$a > R_{\min}$ this yields cusp bifurcations.

\subsection{Special cases}

\noindent
We first check the special cases $\lambda = 0$ and
$\lambda = \Frac{1}{2\kappa}$.

\paragraph{Case $\boldsymbol{\lambda = 0}$.}
Setting $\lambda = 0$ in~\eqref{solved-triple-roots} gives
\begin{displaymath}
   b \; = \; \frac{4}{\kappa^2} \, - \, 3 a
   \quad \mbox{and} \quad
   \mu^2 \; = \; (2 \kappa^2 a - 3) a^2
   \enspace .
\end{displaymath}
Inserting this in~\eqref{coefficients-triple-roots} we obtain the
quadratic equation
\begin{equation}
\label{quadraticinell}
   \ell^2 \; + \;
   \left( -2 \kappa^4 a^3 + 6 \kappa^2 a^2 - 6a \right) \ell
   \; + \; \left( 3 \kappa^4 a^4 - 10 \kappa^2 a^3 + 9 a^2 \right)
   \;\; = \;\; 0
\end{equation}
in~$\ell$ with discriminant $a^3 (\kappa^2 a - 2)^3$.
The condition $a \geq R_{\min} \geq |\mu|$ implies
$a^2 \ge \mu^2$ and $a \ge 0$, and gives
\begin{displaymath}
   a^2 (2 - \kappa^2 a) \ge 0 \enspace ,
\end{displaymath}
that is $0 \le a \le 2\kappa^{-2}$.
For these values~\eqref{quadraticinell} has negative discriminant
and thus non-real roots except for the end points $a=0$ and
$a = 2 \kappa^{-2}$ of the interval.
For $a = 0$ we find $\mu = \ell = a = 0$, recovering the equilibrium
in $1{:}1{:}{-}2$~resonance with values
$(\lambda, \mu, \ell) = (0, 0, 0)$.
While this is what makes $\lambda = 0$ special, for
$a = 2 \kappa^{-2}$ we find $\pm \mu = \ell = a = 2\kappa^{-2}$,
corresponding to a supercritical Hamiltonian Hopf bifurcation.
These are not special, but belong to the $1$--parameter families of
supercritical Hamiltonian Hopf bifurcations \HHsup{1} and~\HHsup{2},
see below.

\paragraph{Case $\boldsymbol{\lambda = \Frac{1}{2\kappa}}$.}
Here the equation $Q(\mu) = 0$ becomes linear in~$\mu^2$ and in
particular it factorizes to
\begin{displaymath}
   (2 \kappa^2 a - 1)^3 (\mu^2 - 2 \kappa^2 a^3) \;\; = \;\; 0
\end{displaymath}
whence $a = \Frac{1}{2\kappa^2}$ or $\mu^2 = 2 \kappa^2 a^3$.
\vspace{-5pt}
Where both equations are satisfied we recover the two degenerate
Hamiltonian Hopf bifurcations \HHdeg{1} and~\HHdeg{2}.
For $a = \Frac{1}{2\kappa^2}$ we have
\begin{displaymath}
  b \; = \; \frac{1}{2\kappa^2} \; = \; a \enspace ,
  \quad \ell \; = \; \frac{1}{4\kappa^2} \, + \, \kappa^2 \mu^2
  \quad \mbox{and} \quad
  h \; = \; \frac{1}{8\kappa^3} \, + \, \kappa \mu^2
\end{displaymath}
where the parametrisation of $\ell$ and~$h$ by~$\mu$ is restricted
by $|\mu| \leq \Frac{1}{2\kappa^2}$, as obtained from
$a \geq R_{\min} \geq \max(|\mu|, \ell)$.
Therefore
\begin{displaymath}
   (\lambda, \mu, \ell) \;\; = \;\; \left( \frac{1}{2 \kappa},\,
   \mu,\, \frac{1}{4 \kappa^2} + \kappa^2 \mu^2 \right)
   \enspace , \quad
   |\mu| \; \le \; \frac{1}{2 \kappa^2}
\end{displaymath}
parametrises a $1$--parameter family of cusp bifurcations that we
denote by~\Cusp{3} and which extends between \HHdeg{1} and~\HHdeg{2}.
While this is what makes $\lambda = \Frac{1}{2\kappa}$ special, for
$\mu^2 = 2 \kappa^2 a^3$ we have 
\begin{displaymath}
   b \; = \; -3 a \, + \, \frac{2}{\kappa^2} \enspace ,
   \quad \ell \; = \; \frac{6 \kappa^2 a - 1}{4 \kappa^2}
   \quad \mbox{and} \quad
   h \; = \; \frac{3 \kappa a^2}{2} \enspace .
\end{displaymath}
Since $a^2 \ge \mu^2 = 2 \kappa^2 a^3$ and $a \ge 0$ we find
$0 \le a \le \Frac{1}{2\kappa^2}$ (the second inequality also
\vspace{-3pt}
following from $a \ge \ell$) where $a = \pm \mu$ gives the end
points $a = 0$ and $a = \Frac{1}{2\kappa^2}$.
\vspace{-3pt}
We have already seen that the right end point
$a = \Frac{1}{2\kappa^2}$ yields the two degenerate Hamiltonian
\vspace{-3pt}
Hopf bifurcations \HHdeg{1} and~\HHdeg{2}.
The left end point $a = 0$ yields
$(\lambda, \mu, \ell) = (\Frac{1}{2\kappa}, 0, \Frac{-1}{4\kappa^2})$
which belongs to the family~\HHsub{3} of subcritical Hamiltonian
Hopf bifurcations, see below.
In between the parametrisation
\begin{displaymath}
   (\lambda, \mu, \ell) \;\; = \;\; \left( \frac{1}{2\kappa},\,
   \pm \sqrt{2 \kappa^2 a^3},\, \frac{6\kappa^2 a-1}{4\kappa^2}
   \right) \enspace , \quad
   0 \; < \; a \; < \; \frac{1}{2\kappa^2}
\end{displaymath}
yields centre-saddle bifurcations which turn out to belong to the
$2$--parameter families of centre-saddle bifurcations \CS{1}
and~\CS{2}, see again below.

\subsection{General case:
$\boldsymbol{\lambda \neq 0}$,
$\boldsymbol{\Frac{1}{2\kappa}}$}

\noindent
Solving the quadratic equation $Q(\mu) = 0$ for $\mu^2$ we find the
two solutions
\begin{subequations}
\label{general mu and ell}
\begin{align}
   \mu_\pm^2 & \;\; = \;\; \frac{\pm 2  |\lambda|}{2 \kappa \lambda - 1}
   \, \bigl[\, (\kappa a + \lambda)^2 - 2a \,\bigr]^{3/2}
\label{general mu}\\
   & \;\;\;\; \mbox{} + \frac{2 \kappa^3 a^3 \lambda - 2 \kappa^2 a^3
   + 6 \kappa^2 a^2 \lambda^2 - 6 \kappa a^2 \lambda + 3 a^2
   + 6 \kappa a \lambda^3 - 6 a \lambda^2 + 2 \lambda^4}{2 \kappa \lambda - 1}
\nonumber\\
\intertext{and}
   2 \lambda \ell_\pm & \;\; = \;\; - 2 \kappa^3 a^3 \; - \;
   6 \kappa^2 a^2 \lambda \; + \; 3 \kappa a^2 \; - \; 6 \kappa a \lambda^2
   \; + \; 6 a \lambda \; - \; 2 \lambda ^3 \; + \; \kappa \mu_\pm^2
   \enspace .
\label{general ell}
\end{align}
\end{subequations}
These solutions are real provided that the discriminant
$16 \lambda^2 [(\kappa a + \lambda)^2 - 2 a]^3$ of $Q(\mu)$,
the latter seen as a quadratic polynomial in $\mu^2$, is
non-negative.
While this is always true for $\lambda > \Frac{1}{2\kappa}$,
this gives for $\lambda < \Frac{1}{2\kappa}$ the sub-cases
\begin{displaymath}
   0 \; \le \; \kappa^2 a \; \le \; 1 \, - \, \kappa \lambda \, - \,
   \sqrt{1 - 2 \kappa \lambda}
   \quad \text{and} \quad
   \kappa^2 a \; \ge \; 1 \, - \, \kappa \lambda \, + \,
   \sqrt{1 - 2 \kappa \lambda}
   \enspace .
\end{displaymath}

\paragraph{Case $\boldsymbol{\lambda > \Frac{1}{2\kappa}}$.}
Here, the condition $a \ge \ell_+$ is not satisfied for any $a \ge 0$.
Therefore, the solutions $\mu_+$, $\ell_+$
in~\eqref{general mu and ell} can be rejected.

We have checked with Mathematica that the condition
$a \ge \ell_-$ is satisfied for all $a \ge 0$ and 
that $a^2 \ge \mu_-^2$ is also true.
We note that the condition $\mu_-^2 \ge 0$ gives
\begin{displaymath}
   a \, g(a) \;\; \le \;\; 0
\end{displaymath}
with
\begin{equation}
\label{g(a)}
   g(a) \;\; = \;\; 4 \kappa^4 a^3 \; + \;
   \left( 12 \kappa^3 \lambda -12 \kappa^2 \right) a^2 \; + \;
   \left( 12 \kappa^2 \lambda^2 - 18 \kappa \lambda + 9 \right) a
   \; + \; \left( 4 \kappa \lambda^3 - 4 \lambda^2 \right)
   \enspace .
\end{equation}
For $a = 0$ the inequality $a \, g(a) \le 0$ is satisfied for any
value of $g(a)$.
If $a > 0$ then we require that $g(a) \le 0$ and check that
$g(0) = 4 \kappa^2 \lambda^2 (\kappa \lambda - 1)$ and that
$g(a)$ has two extrema at strictly negative values of~$a$.
This means that for $\Frac{1}{2\kappa} < \lambda < \Frac{1}{\kappa}$
the cubic equation $g(a) = 0$ has a unique positive root
$a_0(\lambda)$ and thus $g(a) \le 0$ for $0 \le a \le a_0(\lambda)$.
This yields centre-saddle bifurcations at the triple roots
$a > R_{\min}$ while for $\kappa^2 a = 1 - \kappa \lambda$ we have
$b=a$ and get cusp bifurcations (recall that the non-degeneracy
condition $F^{(4)}(a) \neq 0$ is always true in our system).

We have $0 \le 1 - \kappa \lambda \le \kappa^2 a_0(\lambda)$,
therefore the bifurcations are split to several $2$--parameter
families of centre-saddle bifurcations separated by cusp
bifurcations.
The centre-saddle bifurcations are the family~\CS{4} parametrised by
\begin{displaymath}
   \frac{1}{2\kappa} \; < \; \lambda \; < \; \frac{1}{\kappa}
   \enspace , \quad
   \mu \; = \; \pm \mu_-
   \enspace , \quad
   \ell \; = \; \ell_-
   \enspace , \quad
   1 - \kappa \lambda \; < \; \kappa^2 a \; < \; \kappa^2 a_0(\lambda)
   \enspace .
\end{displaymath}
Then, a part of the family~\CS{1} is parametrised by
\begin{displaymath}
  \frac{1}{2\kappa} \; < \; \lambda \; < \; \frac{1}{\kappa}
   \enspace , \quad
   \mu \; = \; \mu_-
   \enspace , \quad
   \ell \; = \; \ell_-
   \enspace , \quad
   0 \; < \; \kappa^2a \; < \; 1 - \kappa \lambda
\end{displaymath}
and a part of \CS{2} is parametrised by
\begin{displaymath}
   \frac{1}{2\kappa} \; < \; \lambda \; < \; \frac{1}{\kappa}
   \enspace , \quad
   \mu \; = \; -\mu_-
   \enspace , \quad
   \ell \; = \; \ell_-
   \enspace , \quad
   0 \; < \; \kappa^2a \; < \; 1 - \kappa \lambda
   \enspace .
\end{displaymath}
There are two families of cusp bifurcations denoted by \Cusp{1}
and~\Cusp{2}.
They can be obtained from~\CS{4} by setting
$\kappa^2 a = 1 - \kappa \lambda$.
This gives for~\Cusp{1} that
\begin{displaymath}
   \frac{1}{2\kappa} \; < \; \lambda \; < \; \frac{1}{\kappa}
   \enspace , \quad
   \mu \; = \; \mu_- \; = \;
   \kappa^{-2}(\kappa \lambda \, - \, \sqrt{2 \kappa \lambda - 1})
   \enspace , \quad
   \ell_- \; = \; 1 \, - \, \kappa \lambda \, - \,
   \sqrt{2 \kappa \lambda - 1}
\end{displaymath}
and for~\Cusp{2} the same parametrisation up to $\mu = - \mu_-$.

For $\lambda \ge \kappa^{-1}$ we have $g(a) > 0$ for all $a > 0$
and therefore the only possibility that is left is $a = 0$.
Subsequently, for $\lambda \ge \kappa^{-1}$, we obtain by
substituting $a = 0$ that
\begin{displaymath}
   \mu_{\pm}^2 \; = \; 0
   \quad \mbox{and} \quad
   \ell_{\pm} \; = \; - \lambda^2
   \enspace .
\end{displaymath}
So, for $\lambda > \kappa^{-1}$ we have the family of supercritical
Hamiltonian Hopf bifurcations parametrised by
$(\lambda, \mu, \ell) = (\lambda, 0, -\lambda^2)$ and denoted
by~\HHsup{3}.
This can be checked using the derivative
$F'''(a=0) = 6(\kappa \lambda - 1) > 0$.

\paragraph{Case $\boldsymbol{\lambda < \Frac{1}{2\kappa},\, \lambda \neq 0}$.}
The condition $a \ge \ell_-$, together with $a \ge 0$ and
$a^2 \ge \mu_-^2 \ge 0$, gives 
\begin{displaymath}
   \kappa^2 a \;\; \geq \;\; 1 \; - \; \kappa \lambda \; + \;
   \sqrt{1 - 2 \kappa \lambda}
\end{displaymath}
or
\begin{displaymath}
   \kappa^2 a_0(\lambda) \;\; \le \;\; \kappa^2 a \;\; \le \;\;
   1 \; - \; \kappa \lambda \; - \; \sqrt{1 - 2 \kappa \lambda}
   \enspace , \quad
   \lambda \; < \; 0
\end{displaymath}
and
\begin{displaymath}
   0 \;\; \le \;\; \kappa^2 a \;\; \le \;\;
   1 \; - \; \kappa \lambda \; - \; \sqrt{1 - 2 \kappa \lambda}
   \enspace , \quad
   \lambda \; > \; 0
   \enspace .
\end{displaymath}
Here $a_0(\lambda)$ is the unique real root of~\eqref{g(a)}.
This parametrises the part of the family~\CS{1} with $\lambda < \frac{1}{2}$
for $\mu = \mu_-$ and the corresponding part of~\CS{2} for $\mu = -\mu_-$.
The condition $a \ge \ell_+$, together with $a \ge 0$ and
$a^2 \ge \mu_+^2 \ge 0$, gives
\begin{displaymath}
   \kappa^2 a \;\; \geq \;\; 1 \; - \; \kappa \lambda \; + \;
   \sqrt{1 - 2 \kappa \lambda}
\end{displaymath}
or
\begin{displaymath}
   0 \;\; \le \;\; \kappa^2 a \;\; \le \;\;
   1 \; - \; \kappa \lambda \; - \; \sqrt{1 - 2 \kappa \lambda}
   \enspace , \quad
   \lambda \; < \; 0
\end{displaymath}
and
\begin{displaymath}
   \kappa^2 a_0(\lambda) \;\; \le \;\; \kappa^2 a \;\; \le \;\;
   1 \; - \; \kappa \lambda \; - \; \sqrt{1 - 2 \kappa \lambda}
   \enspace , \quad
   \lambda \; > \; 0
   \enspace .
\end{displaymath}
This parametrises the family~\CS{3} for $\mu = \pm \mu_+$. 
Note that for
$\kappa^2 a = 1 - \kappa \lambda - \sqrt{1 - 2 \kappa \lambda}$
we have $\mu_+^2 = \mu_-^2 = a^2$ and $\ell_+ = \ell_- = a$.
Therefore we obtain two curves parametrising subcritical Hamiltonian
Hopf bifurcations since here
$F'''(a) = - 6 \sqrt{1-2\kappa\lambda} < 0$.
The two curves are \HHsub{1} and~\HHsub{2}.
For $\kappa^2 a = 1 - \kappa \lambda + \sqrt{1 - 2 \kappa \lambda}$
we obtain the two curves
\begin{displaymath}
   \ell \; = \; a
   \enspace , \quad
   \mu \; = \; \pm a
   \enspace , \quad
   \lambda \; = \; \pm \sqrt{2a} - \kappa a
   \enspace , \quad
   a \; \ge \; 0
\end{displaymath}
which can alternatively be written as
\begin{displaymath}
   \ell \; = \; \kappa^{-2}
   (1 \, - \, \kappa \lambda \, + \, \sqrt{1 - 2 \kappa \lambda})
   \enspace , \quad
   \mu \; = \; \pm \ell
   \enspace , \quad
   \lambda \; < \; \frac{1}{2 \kappa}
   \enspace .
\end{displaymath}
These two curves parametrise families of supercritical Hamiltonian
Hopf bifurcations since here $F'''(a) = 6 \sqrt{1-2\kappa\lambda} > 0$.
The curves are \HHsup{1} and~\HHsup{2}.

Note that for $\lambda = 0$ these expressions give
$\ell = 2\kappa^{-2}$, $\mu = \pm 2\kappa^{-2}$
which are two of the values we already identified for $\lambda=0$.
Considering now the constraint
$0 \le \kappa^2 a \le 1 - \kappa \lambda - \sqrt{1 - 2 \kappa \lambda}$
for $\lambda = 0$ we find $a = 0$ and thus $\ell = \mu = 0$.
This is the third value we identified for $\lambda = 0$.
Therefore, we can extend the parametrisation for
$\lambda < \Frac{1}{2\kappa}$ to include the case $\lambda = 0$.

\section*{Acknowledgment}
A.M. was supported by INdAM (Istituto Nazionale di Alta Matematica
"F. Severi") and  by the Grant Agency of the Czech Republic,
project~17-11805S.

\end{document}